\documentclass[12pt]{amsart}

\usepackage[T1]{fontenc}
\usepackage{imakeidx}
\usepackage{yfonts}
\usepackage{amssymb}
\usepackage{fancyhdr}
\usepackage{amsmath}
\usepackage{amsfonts}
\usepackage{slashed} 
\usepackage[utf8]{inputenc}
\usepackage{mathrsfs}
\usepackage[all]{xy} % For commutative diagrams
\usepackage{graphicx} % For figures
\usepackage{xcolor}
\usepackage{enumitem}
\usepackage{glossaries}

\newcommand{\disk}{\ensuremath{\mathbb{D}} } % unit disk
\newcommand{\sphere}{\bar{\Bbb{C}}} %Riemann sphere
\newcommand{\riem}{\Sigma}  %Riemann surface
\renewcommand{\Bbb}[1]{\ensuremath{\mathbb{#1}}}
\newglossaryentry{bergman}{%
name=\ensuremath{\mathcal{A}},
    description={Bergman space}
}

\newglossaryentry{Abw}{%
name=\ensuremath{\mathcal{A}_\mathrm{bw}},
    description={bridgeworthy harmonic oneforms}
}

\newglossaryentry{Aharm}{%
name=\ensuremath{\mathcal{A}_{\mathrm{harm}}},
    description={Harmonic Bergman space}
}

\newglossaryentry{Ahm}{%
name=\ensuremath{\mathcal{A}_{\mathrm{hm}}},
    description={Complex linear span of harmonic measures}
}

\newglossaryentry{exactA}{%
name=\ensuremath{\mathcal{A}^\mathrm{e}},
    description={Space of exact forms}
}

\newglossaryentry{seform}{%
name=\ensuremath{\mathcal{A}^{{\mathrm{se}}}},
    description={Semi-exact forms}
}

\newglossaryentry{peform}{%
name=\ensuremath{\mathcal{A}_{\mathrm{harm}}^{\mathrm{pe}}},
    description={Piecewise exact harmonic forms}
}

\newglossaryentry{annulus}{%
name=\ensuremath{\mathbb{A}_{a,b}},
    description={Annulus with inner radius $a$ and outer radius $b$}
}

\newglossaryentry{gota}{%
name=\ensuremath{ \mathrm{\textgoth{A}} },
    description={Forms with prescribed periods}
}

\newglossaryentry{Bphi}{%
name=\ensuremath{ \mathbf{B}(\phi)},
    description={Boundary map}
}

\newglossaryentry{cf}{%
name=\ensuremath{\mathbf{C}_{f}},
    description={right-composition with $f$}
}

\newglossaryentry{cl}{%
name=\ensuremath{\text{cl}},
    description={Closure of a set}
}

\newglossaryentry{D}{%
name=\ensuremath{\mathcal{D}},
    description={Dirichlet space}
}

\newglossaryentry{Dbw}{%
name=\ensuremath{\mathcal{D}_\mathrm{bw}},
    description={bridgeworthy harmonic functions}
}

\newglossaryentry{Dir}{%
name=\ensuremath{\mathbf{Dir}},
    description={Solution map to the Dirichlet problem}
}

\newglossaryentry{Dharm}{%
name=\ensuremath{\mathcal{D}_{\mathrm{harm}}},
    description={Harmonic Dirichlet space}
}

\newglossaryentry{Dhom}{%
name=\ensuremath{\dot{\mathcal{D}}},
    description={Dirichlet space modulo constants}
}

\newglossaryentry{harmeasure}{%
name=\ensuremath{d\omega_{k}},
    description={Harmonic measure}
}

\newglossaryentry{E}{%
name=\ensuremath{\mathbf{E}},
    description={Data to solution map}
}

\newglossaryentry{greenb}{%
name=\ensuremath{G_\Sigma},
    description={Green's function of a bordered Riemann surface $\Sigma$}
}
\newglossaryentry{greenc}{%
name=\ensuremath{\mathscr{G}},
    description={Green's function of a compact Riemann surface}
}

\newglossaryentry{bounce}{%
name=\ensuremath{\mathbf{G}_{U,\riem}},
    description={Bounce operator}
}

\newglossaryentry{grunsk}{%
name=\ensuremath{\mathbf{Gr}_{f}},
    description={Grunsky operator}
}

\newglossaryentry{sobolev}{%
name=\ensuremath{H^s},
    description={Sobolev space}
}
\newglossaryentry{homsobolev}{%
name=\ensuremath{\dot{H}^s},
    description={Homogeneous Sobolev space}
}

\newglossaryentry{sobolevconf}{%
name=\ensuremath{{H}^{1}_{\mathrm{conf}}},
    description={Conformal Sobolev space}
}

\newglossaryentry{Dbvaluescomp}{%
name=\ensuremath{\mathcal{H}'(\partial_k \riem)},
    description={Dirichlet boundary values for one forms}
}

\newglossaryentry{Dbvalues}{%
name=\ensuremath{\mathcal{H}'(\partial\riem)},
    description={Dirichlet boundary values for one forms}
}

\newglossaryentry{BVexactcomp} {%
name=\ensuremath{\dot{H}'(\partial_k \riem)},
    description={Boundary values with exact representative}
}

\newglossaryentry{BVexact} {%
name=\ensuremath{\dot{H}'(\partial \riem)},
    description={Boundary values with exact representative}
}

\newglossaryentry{crop}{%
name=\ensuremath{\mathbf{J}_{1}^{q}},
    description={Cauchy-Royden operator}
}
    
\newglossaryentry{rcrop}{%
name=\ensuremath{\mathbf{J}_{1,k}^{q}},
    description={Restricted Cauchy-Royden operator}
}
\newglossaryentry{Jdot}{%
name=\ensuremath{\dot{\mathbf{J}}_1},
    description={Cauchy-Royden operator on $\dot{\mathcal{D}}$}
}

\newglossaryentry{bergmank}{%
name=\ensuremath{K},
    description={Bergman kernel}
}
\newglossaryentry{schifferk}{%
name=\ensuremath{L},
    description={Schiffer kernel}
}
\newglossaryentry{rest}{%
name=\ensuremath{\mathbf{R}},
    description={Restriction operator}
}

\newglossaryentry{harmrest}{%
name=\ensuremath{\mathbf{R}^{\mathrm{h}}},
    description={Harmonic restriction operator}
}

\newglossaryentry{S}{%
name=\ensuremath{\mathbf{S}},
    description={The Schiffer comparison operator}
}
  
\newglossaryentry{Sharm}{%
name=\ensuremath{\mathbf{S}_{k}^{\mathrm{h}}},
    description={Harmonic Schiffer operator}
}
  
\newglossaryentry{theta}{%
name=\ensuremath{\Theta},
    description={Map}
}

\newglossaryentry{Tmix}{%
name=\ensuremath{\mathbf{T}_{\riem_{j},\riem_{k}}},
    description={Schiffer operator}
}
\newglossaryentry{T}{%
name=\ensuremath{\mathbf{T}},
    description={Schiffer comparison operator}
}

\newglossaryentry{overfare}{%
name=\ensuremath{\mathbf{O}},
    description={Overfare operator}
}
\newglossaryentry{doto}{%
name=\ensuremath{\dot{\mathbf{O}}},
    description={Overfare operator on $\dot{\mathcal{D}}$}
}

\newglossaryentry{exacto}{%
name=\ensuremath{\mathbf{O}^{\mathrm{e}}_{2,1}},
    description={Exact overfare operator}
}
    
\newglossaryentry{ohat}{%
name=\ensuremath{\hat{\mathbf{O}}},
    description={Operator}
}

\newglossaryentry{augo}{%
name=\ensuremath{\mathbf{O}^{\mathrm{aug}}},
    description={Augmented overfare operator}
}
    
\newglossaryentry{oprime}{%
name=\ensuremath{\mathbf{O}'},
    description={Operator}
}
    
\newglossaryentry{oprimedot}{%
name=\ensuremath{\dot{\mathbf{O}}'},
    description={Operator}
}
    
\newglossaryentry{pcap}{%
name=\ensuremath{\mathbf{P}_{\mathrm{cap}}}, description={Projection operator}
}

\newglossaryentry{period}{%
name=\ensuremath{\mathbf{\Upsilon}},
    description={Period map}
}

\theoremstyle{plain}
        \newtheorem{theorem}{Theorem}[section]
        \newtheorem{lemma}[theorem]{Lemma}
        \newtheorem{proposition}[theorem]{Proposition}
        \newtheorem{corollary}[theorem]{Corollary}

\theoremstyle{definition}
        \newtheorem{definition}[theorem]{Definition}
        \newtheorem{example}{Example}[section]

\theoremstyle{remark}
    \newtheorem{remark}[theorem]{Remark}

\numberwithin{equation}{section} % Equation labels are 'section'.'eq #'
\numberwithin{figure}{section} % Figures labela are 'section.'fig #'

\usepackage{fullpage}
\author[E. Schippers]{Eric Schippers}
\author[W. Staubach]{Wolfgang Staubach}

\address{\newline
       Eric Schippers \newline
       Machray Hall, Dept. of Mathematics,
   University of Manitoba, \newline Winnipeg, MB
   Canada R3T 2N2}
       \email{eric.schippers@umanitoba.ca}

\address{\newline
       Wolfgang Staubach \newline
       Department of  Mathematics, Uppsala University, \newline
       S-751 06 Uppsala, Sweden}
       \email{wulf@math.uu.se}
       

 
 \title{Scattering theory on Riemann surfaces I: Schiffer operators, cohomology, and index theorems}

\makenoidxglossaries{}
\begin{document}
\begin{abstract} 
 We consider a compact Riemann surface $\mathscr{R}$ with a complex of non-intersecting Jordan curves, whose complement is a pair of Riemann surfaces with boundary, each of which may be possibly disconnected. 
 We investigate conformally invariant integral operators of Schiffer, which act on $L^2$ anti-holomorphic one-forms on one of these surfaces with boundary and produce holomorphic one-forms on the disjoint union. These operators arise in potential theory, boundary value problems, approximation theory, and conformal field theory, and are closely related to a kind of Cauchy operator.  

 We develop an extensive calculus for the Schiffer and Cauchy operators, including a number of adjoint identities for the Schiffer operators. In the case that the Jordan curves are quasicircles, we derive a Plemelj-Sokhotski jump formula for Dirichlet-bounded functions. We generalize a theorem of Napalkov and Yulmukhametov, which shows that a certain Schiffer operator is an isomorphism for quasicircles. Finally, we characterize the kernels and images, and derive index theorems for the Schiffer operators, which will in turn connect conformal  invariants to topological invariants.  

%We construct a scattering theory for harmonic one-forms on Riemann surfaces, obtained from boundary value problems through systems of curves and the jump problem. We obtain an explicit expression for the scattering matrix in terms of integral operators which we call Schiffer operators, and show that the matrix is unitary. As a consequence of this scattering theory, we prove index theorems relating these conformally invariant integral operators to topological invariants. We also obtain a general association of positive polarizing Lagrangian spaces to bordered Riemann surfaces, which unifies the classical polarizations for compact surfaces of algebraic geometry with the infinite-dimensional period map of the universal Teichm\"uller space. 
\end{abstract}

\maketitle

%\tableofcontents
\begin{section}{Introduction}
\begin{subsection}{Literature and statement of results}

{This paper, together with its sequel \cite{Schippers_Staubach_scattering_IV}, is concerned with the global analytic aspects of a kind of scattering process as manifested in integral operators of Schiffer.  
 The scattering process occurs on a compact Riemann surface separated into two pieces (which themselves might not be connected) by a collection of Jordan curves. Alternatively, we can think of a compact surface obtained by sewing together several surfaces, and the Jordan curves are the seams. This situation arises in both the Teichm\"uller theory of bordered surfaces and two-dimensional conformal field theory. {A non-technical exposition of some aspects of this scattering theory can be found in \cite{Schippers_Staubach_Carlos_paper}.}

One considers functions or one-forms which are separately harmonic on the pieces, and share boundary values on the seams. The function (or one-form) on one side of the surface is obtained from the function (or one-form) on the other side using a mapping which we refer to as "overfare", which is the basis of this particular kind of scattering theory. Teichm\"uller theory requires that the seams can be quasicircles, and much evidence exists that quasicircles are analytically natural for the scattering theory \cite{Schippers_Staubach_Grunsky_expository}.  In \cite{Schippers_Staubach_scattering_I,Schippers_Staubach_scattering_II} it was shown that the overfare process is well-defined and bounded for quasicircles.  Furthermore, in order that the results be applicable to Teichm\"uller theory and conformal field theory, it is necessary that these seams can be quasicircles. } 
 
The cornerstone of the paper at hand is the theory of certain integral operators of Schiffer, which appear in association to the scattering process on account of their relation to solutions of the boundary value problems arising in overfare. These operators are integral operators on holomorphic and anti-holomorphic one-forms, whose integral kernels are the two possible second derivatives of Green's function, often called the Bergman and Schiffer kernels.

The precise definition is as follows. Let $\mathscr{R}$ be a compact Riemann surface  split into two surfaces $\riem_1$ and $\riem_2$ by a collection of Jordan curves. Let $\mathscr{G}(w;z,q)$ be Green's function of $\mathscr{R}$ (the fundamental harmonic function with logarithmic singularities at $z$ and $q$ of opposite weight, defined up to an additive constant).  We have, denoting the Bergman space of holomorphic one-forms on $\riem_k$ by $\mathcal{A}(\riem_k)$ for $k=1,2$, 
\begin{align*}
  \mathbf{T}_{1,k}:\overline{\mathcal{A}(\riem_1)} & \rightarrow \mathcal{A}(\riem_k) \\
  \overline{\alpha} & \mapsto \iint_{\riem_1} \partial_w \partial_z \mathscr{G}(w;z,q) \wedge_w \overline{\alpha(w)}.  
\end{align*} 
The two choices of $k$ are obtained by restricting $z$ to $\riem_k$. 
If $k=1$, this has a singularity and can be regarded as a Calder\'on-Zygmund singular integral operator. We also have the operator
\begin{align*}
    \mathbf{S}_{1}: {\mathcal{A}(\riem_1)} & \rightarrow \mathcal{A}(\mathscr{R}) \\
  \overline{\alpha} & \mapsto \iint_{\riem_1} \overline{\partial}_w \partial_z \mathscr{G}(w;z,q) \wedge_w \overline{\alpha(w)}. 
\end{align*}
We may of course place ${\riem}_2$ in the role taken by ${\riem}_1$ above. 
These were investigated extensively by M. Schiffer with various co-authors \cite{BergmanSchiffer} \cite{Schiffer_first}, in relation to potential theory and conformal mapping,  eventually culminating in a comparison theory of domains \cite{Courant_Schiffer}. The Schiffer kernel is closely related to the so-called fundamental bidifferential and figures in the geometry of function spaces on Riemann surfaces  \cite{Eynard_notes}, \cite{Schiffer_Spencer}.  

By a striking result of V. Napalkov and R. Yulmukhametov \cite{Nap_Yulm}, if $\mathscr{R}$ is the Riemann sphere, and $\riem_1$ and $\riem_2$ are the two complementary components of a Jordan curve $\Gamma$ on the sphere, then the Schiffer operator $\mathbf{T}_{1,2}$ is an isomorphism if and only if $\Gamma$ is a quasicircle. {This is closely related to the fact that functions can be approximated in the Dirichlet seminorm by Faber series precisely for domains bounded by quasicircles; see \cite{Schippers_Staubach_Grunsky_expository} for an overview.} The authors showed in \cite{Schippers_Staubach_Plemelj} that, for a compact Riemann surface divided in two by a quasicircle,   $\mathbf{T}_{1,2}$ is an isomorphism on the orthogonal complement of anti-holomorphic one-forms on $\mathscr{R}$. This was further generalized by M. Shirazi to the case of many curves where all but one of the components is simply connected in \cite{Shirazi_thesis}, \cite{Schippers_Shirazi_Staubach}.  The boundedness of overfare plays a central role in the formulation and proof of this fact.
{This extension of the isomorphism theorem was used by the authors and Shirazi to show that one-forms on a domain in a Riemann surface bounded by quasicircles can be approximated in $L^2$ on a larger domain \cite{Schippers_Shirazi_Staubach}. Approximability theorems  for general $k$-differentials with respect to the conformally invariant $L^2$ norm and less regular boundaries were obtained by N. Askaripour and T. Barron \cite{AskBar,AskBar2} using very different methods. So far as we know, these were the first results for nested domains on Riemann surfaces in the $L^2$ setting}.
\\

In this paper, we characterize the kernel and image of $\mathbf{T}_{1,2}$ in the case of a Riemann surface split by a complex of quasicircles. This has been applied to prove $L^2$ approximability of one-forms on bordered Riemann surfaces by Faber series, by the first author and Shirazi \cite{Schippers_Shirazi_Faber}. The results in the present paper are derived using an extended Plemelj-Sokthoski jump formula, which is in turn based on a relation between the Schiffer operators and a generalization of the Cauchy operator originating with H. Royden \cite{Royden} which we call the Cauchy-Royden operator. As quasicircles are not rectifiable, we are required to define the Cauchy-Royden integral using curves which approach the boundary. In the sphere with one curve, the authors showed that the resulting Plemelj-Sokhotski jump decomposition is an isomorphism if and only if the curve is a quasicircle.  The analytic issues in those papers, as in this one, are resolved by the fact that the limiting integral is the same from both sides up to constants. This in turn is a consequence of the anchor lemmas and boundedness of the bounce operator. The equality of the limiting integral from both sides is also a key geometric tool; in combination with the bounded overfare it allows one to find preimages of elements of the image of $\mathbf{T}_{1,2}$. \\

We further use this  to investigate the cohomology of the images of $\mathbf{T}_{1,1}$, $\mathbf{T}_{1,2}$ and $\mathbf{S}_1$. In particular we show that for any anti-holomorphic one form $\overline{\alpha}$ in $\riem_1$, $\mathbf{T}_{1,2} \overline{\alpha}$ and $\overline{\mathbf{S}}_1 \overline{\alpha}$ are in the same cohomology class. This simple fact is surprisingly versatile. Along with the characterization of the kernels and images of $\mathbf{T}_{1,2}$ mentioned above, we also show that in the case that $\riem_1$ and $\riem_2$ are connected, the Fredholm index of $\mathbf{T}_{1,2}$ is $g_1-g_2$ where $g_1$ and $g_2$ are the genuses of $\riem_1$ and $\riem_2$. This index theorem relates a conformal invariant (the index of $\mathbf{T}_{1,2}$) to the topological invariant $g_1-g_2$. \\

Finally, we derive a number of new identities for Schiffer operators and their adjoints, as well as extend identities obtained earlier in \cite{Schippers_Staubach_Plemelj} to the case of a compact surface split by a complex of curves.
These identities play a central role in the scattering theory. It should be mentioned that one of these identities is a reformulation and significant generalization of an norm identity of Bergman and Schiffer for planar domains \cite{BergmanSchiffer}. This identity can be used to derive generalizations of the Grunsky inequality and characterize boundary values of holomorphic one-forms, as will be seen in the sequel \cite{Schippers_Staubach_scattering_IV}.
\\

We remark that this paper is one part of a longer work \cite{Schippers_Staubach_scattering_arxiv} establishing the scattering theory of one-forms on Riemann surfaces, which we have divided into four parts. The first two papers \cite{Schippers_Staubach_scattering_I,Schippers_Staubach_scattering_II} established the analytic basis for the scattering process in terms of boundary value problems for $L^2$ harmonic functions and one-forms. We intend for these four papers to pave the way for the investigation of geometric and representation theoretic aspects of Teichm\"uller theory and conformal field theory. \\

{{\bf{Acknowledgements.}} The first author was partially supported by the National Sciences and Engineering Research Council of Canada. The second author is grateful to Andreas Strömbergsson for partial financial support through a grant from Knut and Alice Wallenberg Foundation.
 The authors are also grateful to Rigun\dh\, Staubach for preparing the figures in this manuscript. 
}
\end{subsection}
\begin{subsection}{Outline}
 Section \ref{se:preliminaries} collects definitions and material required throughout the paper. This includes Riemann surfaces with border, spaces of functions and forms and their norms, boundary values of harmonic functions of finite Dirichlet energy, the overfare process and main theorems derived in \cite{Schippers_Staubach_scattering_I}, Green's function, and some basic results about the harmonic measures (one-forms which vanish on the boundary). 

 Section \ref{se:Schiffer_Cauchy} contains the foundational results about the Schiffer and Cauchy-Royden operators. 
 In Section \ref{subsec:defSchiffer} we define the Schiffer operators, and an associated integral operator which we call the Cauchy-Royden operator. We show that these are bounded, and derive a simple set of relations between the Schiffer and Cauchy-Royden operators. We derive a kind of Plemelj-Sokhotski jump formula which mixes overfare and the Cauchy-Royden operator. This is the main tool in investigating the effect of the operators on cohomology classes, as well as the investigation of their kernels and images.  In Section \ref{se:Schiffer_adjoint_identities} we derive a number of identities for the adjoints of these operators, which play a role in this paper and in the proof of the unitarity of the scattering operator {in part II}.  In Section \ref{se:Schiffer_on_harmonic_measures} we establish the action of the Schiffer operators on harmonic measures, which also plays a role in the determination of the kernels and images of the Schiffer and Cauchy-Royden operators in this paper, and in the scattering theory {in part II}. 

  Section \ref{se:index_cohomology} contains geometric and algebraic results about the Schiffer operators introduced in Subsection \ref{subsec:defSchiffer}.  We establish the isomorphism theorem generalizing those of Napalkov and Yulmukhametov \cite{Nap_Yulm} and Shirazi \cite{Shirazi_thesis}.  We give a characterization of the image and kernel of $\mathbf{T}_{1,2}$, and use this to prove an index theorem for this operator in the case that $\riem_1$ and $\riem_2$ are connected, and in the case that $\riem_2$ is of genus $g$ with $n$ boundary curves capped by $n$ simply connected domains.  This index theorem relates the conformally invariant index to purely topological quantities.
 
  We investigate the effect of the Schiffer operators $\mathbf{T}_{j,k}$ and $\mathbf{S}_k$ on cohomology in Section \ref{se:Schiffer_cohomology}. The main tool is the ``overfared'' jump formula, which is used to prove Theorem \ref{th:Schiffer_cohomology} which says that certain linear combinations of the Schiffer operators produce exact forms. Together with the fact that $\mathbf{S}_k \mathbf{R}_k$ is an isomorphism, this completely characterizes the effect of $\mathbf{T}_{j,k}$ on cohomology classes. We also prove the generalization of the aforementioned isomorphism theorem in the case of capped surfaces. In Section \ref{se:Schiffer_kernel_image} we determine the kernel and image of the operator $\mathbf{T}_{1,2}$ in greater generality. These results will also play a central role in in the construction of the generalized period matrix in {part II}. Once this is accomplished, we prove the index theorems in Section \ref{th:index_and_examples}, which relate the index of the Schiffer operators to topological invariants.  Although the cokernel and kernel are conformally invariant, it is open question whether the cokernel and kernel are topological invariants. We conclude with an example which might suggest a way to investigate this question.
\end{subsection}
\begin{subsection}{Assumptions throughout the paper}
 \label{se:assumptions_scattering_section} 
 In order to prevent needless repetition, we will use consistent notation throughout the paper for the Riemann surfaces. 
 The following assumptions
 will be in force throughout Section \ref{se:Schiffer_Cauchy}. Additional hypotheses are added to the statement of each theorem where necessary.
 \begin{enumerate}
     \item $\mathscr{R}$ is a compact Riemann surface;
     \item $\Gamma = \Gamma_1 \cup \cdots \cup \Gamma_n$ is a collection of quasicircles in $\mathscr{R}$, which have pair-wise empty intersection;
     \item $\Gamma$ separates $\mathscr{R}$ into Riemann surfaces $\riem_1$ and $\riem_2$.
 \end{enumerate} 
 By ``separates'', we mean that $\mathscr{R} \backslash \Gamma$ consists of two Riemann surfaces $\riem_1$ and $\riem_2$, where for any $\Gamma_k$ both $\riem_1$ and $\riem_2$ bound $\Gamma_k$. That is, $\riem_1$ and $\riem_2$ are on opposite sides of each curve $\Gamma_k$ for $k=1,\ldots,n$.
 The precise meaning of ``separates'' is given ahead in Definition \ref{de:separating_complex}.
 It is not assumed a priori that $\riem_1$ or $\riem_2$ is connected, although we will sometimes add the assumption as necessary.
  
  We will furthermore assume that the ordering of the connected components of $\partial \riem_1$ and $\partial \riem_2$ is such that $\partial_k \riem_1 = \partial_k \riem_2 = \Gamma_k$ as sets for $k=1,\ldots,n$.    
\end{subsection}
\end{section}
\begin{section}{Preliminaries} \label{se:preliminaries}

\begin{subsection}{Bordered surfaces and quasicircles}

 We first establish some notation and terminology regarding Riemann surfaces and curves within them. Details can be found in \cite{Schippers_Staubach_scattering_I}. 
 
In what follows we let $\mathbb{C}$ denote the complex plane, $\gls{annulus}$ denote the annulus $ \{ z;\, a<|z|<b \}$, and $\disk = \{ z \in \mathbb{C} : |z|<1 \}$ denote the unit disk in the plane.

One of our main objects of study is a particular type of bordered Riemann surface which is defined as follows:
\begin{definition}\label{defn:gn-type surface}
We say that $\riem$ is a \emph{bordered Riemann surface of type} $(g,n),$ if it is bordered (see e.g. \cite{Ahlfors_Sario}), the border has $n$ connected components, each of which is homeomorphic to $\mathbb{S}^1$, and its double $\riem^d$ is a compact surface of genus $2g + n-1$.  
 \end{definition}
Visually, a bordered surface of type $(g,n)$ is a $g$-handled surface bounded by $n$ simple closed curves.  
 We order the connected components of the border and label them accordingly, so that $\partial \riem = \partial_1 \riem \cup \cdots \cup \partial_n \riem$.  The borders can be identified with analytic curves in the double $\riem^d$, and we denote the union $\riem\cup \partial \riem$ by $\text{cl}(\riem)$.  
 
 Finally, we observe that borders are conformally invariant.  That is, if $\riem_1$ and $\riem_2$ are bordered surfaces, then any biholomorphism $f:\riem_1 \rightarrow \riem_2$ extends to a homeomorphism of the borders.  In fact, $f$ extends to a biholomorphism between the doubles $\riem_1^d$ and $\riem_2^d$ which takes $\partial \riem_1$ to $\partial \riem_2$.  Finally, if only one of the two surfaces has a border, say $\riem_1$, then one can endow $\riem_2$ with a border using $f$. In particular, there is a unique maximal border structure.  
 \begin{remark}
  Note that if $\riem$ has type $(g,n)$, the border structure is maximal, since $\riem^d$ is a compact surface.  
 \end{remark}
 
 \begin{definition}
 We say that a homeomorphic image $\Gamma$ of $\mathbb{S}^1$ is a \emph{strip-cutting Jordan curve} if it is contained in an open set $U$ and there is a biholomorphism $\phi:U \rightarrow \mathbb{A}_{r,R} $ for some annulus 
 \[  \mathbb{A}_{r,R} \subset \mathbb{C}, \ \ \  r<1<R,  \] in such a way that $\phi(\Gamma)$ is isotopic to the circle $|z|=1.$  We call $U$ a doubly-connected neighbouhood of $\Gamma$ and $\phi$ a doubly-connected chart. 
 \end{definition}
 \begin{remark}
  If $\Gamma$ is a strip-cutting curve, by shrinking $\mathbb{A}_{r,R} $, we can assume that (1) $\phi$ extends biholomorphically to an open neighourhood of $\text{cl} \,(U)$, (2), that the boundary curves of $U$ are themselves strip cutting (in fact analytic), and  (3) that $\Gamma$ is isotopic to each of the boundary curves (using $\phi^{-1}$ to provide the isotopy).  
 \end{remark}
 
 \begin{definition}
     We say that a Jordan curve $\Gamma$ in in a Riemann surface is a {\it quasicircle} if it is contained in an open set $U$ and there is a biholomorphism $\phi:U \rightarrow V \subset \mathbb{C}$ such that $\phi(\Gamma)$ is a quasicircle in $\mathbb{C}$. 
 \end{definition}
 \begin{remark} \label{re:analytic_curve_stripcutting}
  A quasicircle (and in particular, an analytic Jordan curve), is by definition strip-cutting. 
 \end{remark}

 Throughout the paper we consider \emph{nested Riemann surfaces}.  That is, we are 
 given a type $(g,n)$ bordered surface $\riem$, another Riemann surface $\mathscr{R}$ which is compact, and a holomorphic inclusion map $\iota( \riem )\subset \mathscr{R}$.  Assume that the closure of $\riem$ is compact in $R$, and furthermore the boundary consists of $n$ closed strip-cutting Jordan curves, which do not intersect.  In that case, the inclusion map $\iota$ extends homeomorphically to a map from the border to the strip-cutting Jordan curves.  Thus $\partial \riem$ is in one-to-one correspondence with its image under the homeomorphic extension of $\iota$, and in fact the image is the boundary of $\iota (\riem)$
 in the ordinary topological sense.  For this reason, we will not notationally distinguish $\riem$ from $\iota (\riem)$.  We will also use the notation $\partial \riem$ for both the boundary of $\iota (\riem)$ in $\mathscr{R}$ and the abstract border of $\riem$, and denote both closures by $\text{cl}\, (\riem)$.

 In fact, it is not necessary to assume that $\riem$ is bordered. That is, if $\riem \subseteq \mathscr{R}$ is bounded by strip-cutting Jordan curves in $\mathscr{R}$, then it can be shown that $\riem$ is bordered and the borders can be bijectively identified with the boundary curves as above (see \cite[Theorem 2.8]{Schippers_Staubach_scattering_I}).  

 \begin{remark}
   The embedding of the border $\partial \riem$ in $\mathscr{R}$ need not be regular.  That is, the inclusion map does not extend to a smooth or analytic map from $\partial \riem$ onto its image under inclusion $\iota$, unless the image consists of smooth or analytic curves. 
 \end{remark}
 
 By another application of \cite[Theorems 3.3, 3.4 Sect 15.3]{ConwayII}, it is easily shown that if $\riem_1\subset \mathscr{R}_1$ and $\riem_2 \subset \mathscr{R}_2$
 satisfy the conditions above, and $f:\riem_1 \rightarrow \riem_2$ is a biholomorphism, then $f$ extends continuously to a map taking each Jordan curve in $\partial \riem_1$ homeomorphically to one of the Jordan curves of $\partial \riem_2$. \\

 \end{subsection}

 \begin{subsection}{Function spaces and holomorphic and harmonic forms}
  In this paper, we will denote positive constants in the inequalities by $C$ whose 
value is not crucial to the problem at hand. The value of $C$ may differ
from line to line, but in each instance could be estimated if necessary.  Moreover, when the values of constants in our estimates are of no significance for our main purpose, then we use the notation $a\lesssim b$ as a shorthand for $a\leq Cb$. If $a\lesssim b$ and $b\lesssim a$ then we write $a\approx b.$  \\
 
 On any Riemann surface,  define the dual of the almost 
 complex structure,  $\ast$ in local coordinates $z=x+iy$,  by 
 \[  \ast (a\, dx + b \, dy) = a \,dy - b \,dx. \]
This is independent of the choice of coordinates.
It can also be computed in coordinates that for any complex function $h$ 
\begin{equation} \label{eq:Wirtinger_to_hodge}
    2 \partial_z h = dh + i \ast dh.
\end{equation}

\begin{definition}
 We say a complex-valued function $f$ on an open set $U$ is \emph{harmonic} if it is $C^2$ on $U$ and $d \ast d f =0$. We say that a complex one-form $\alpha$ is harmonic if it is $C^1$ and satisfies
 both $d\alpha =0$ and $d \ast \alpha =0$. 
\end{definition}
  Equivalently, harmonic one-forms are those which can be expressed locally as $df$ for some harmonic function $f$. Harmonic one-forms and functions must of course be $C^\infty$. \\
  
   Denote complex conjugation of functions and forms with a bar, e.g. $\overline{\alpha}$.
 A holomorphic one-form is one which can be written in coordinates as $h(z)\,dz$ for a holomorphic function $h$, while an anti-holomorphic one-form is one which can be locally written $\overline{h(z)}\, d\bar{z}$ for a holomorphic function $h$.  
 
Denote by $L^2(U)$ the set of one-forms $\omega$ on an open set $U$ which satisfy
\[   \iint_U \omega \wedge \ast \overline{\omega} < \infty  \]
(observe that the integrand is positive at every point, as can be seen by writing the expression in local coordinates).  
This is a Hilbert space with respect to the inner product
\begin{equation} \label{eq:form_inner_product}
 (\omega_1,\omega_2) =  \iint_U \omega_1 \wedge \ast \overline{\omega_2}.
\end{equation}

\begin{definition}\label{defn:bergman spaces}
The \emph{Bergman space of holomorphic one forms} is 
\begin{equation}
    \gls{bergman}(U) = \{ \alpha \in L^2(U) \,:\, \alpha \ \text{holomorphic} \}.
\end{equation} 
 The anti-holomorphic Bergman space is denoted $\overline{\mathcal{A}(U)}$.   We will also denote 
\begin{equation}
    \gls{Aharm}(U) =\{ \alpha \in L^2(U) \,:\, \alpha \ \text{harmonic} \}.
\end{equation}
\end{definition}

Observe that $\mathcal{A}(U)$ and $\overline{\mathcal{A}(U)}$ are orthogonal with respect to the inner product \eqref{eq:form_inner_product}.  In fact we have the direct sum decomposition
\begin{equation}\label{direct sum decomposition}
 \mathcal{A}_{\mathrm{harm}}(U) = \mathcal{A}(U) \oplus \overline{\mathcal{A}(U)}.    
\end{equation}      
If we restrict the inner product to 
 $\alpha, \beta \in \mathcal{A}(U)$ then since $\ast \overline{\beta} = i \overline{\beta}$, we have   
\[  (\alpha,\beta) = i \iint_U \alpha \wedge \overline{\beta}.      \]

Denote the projections induced by this decomposition by 
\begin{align}  \label{eq:hol_antihol_projections_Bergman}
  \mathbf{P}_U: \mathcal{A}_{\mathrm{harm}}(U) & \rightarrow \mathcal{A}(U) \nonumber \\
  \overline{\mathbf{P}}_U : \mathcal{A}_{\mathrm{harm}}(U) & \rightarrow  \overline{\mathcal{A}(U)}.
\end{align}

Let $f: U \rightarrow V$ be a biholomorphism. We denote the pull-back of $\alpha \in \mathcal{A}_{\mathrm{harm}}(V)$
under $f$ by $f^*\alpha.$ 
Explicitly, if $\alpha$ is given in local coordinates $w$ by $a(w)\, dw + \overline{b(w)} \, d\bar{w}$ and $w=f(z),$ then the pull-back is given by 
\[   f^* \left( a(w)\, dw + \overline{b(w)} \,d\bar{w} \right)= a(f(z)) f'(z)\, dz + \overline{b(f(z))} \overline{f'(z)}\, d\bar{z}.   \]
The Bergman spaces are all conformally invariant, in the sense that if $f:U \rightarrow V$ is a biholomorphism, then $f^*\mathcal{A}(V) = \mathcal{A}(U)$ and this preserves the inner product.  Similar statements hold for the anti-holomorphic and harmonic spaces.\\ 

\begin{definition}\label{def: exact holo and harm forms}
 We define the space $\gls{exactA}_{\mathrm{harm}}(U)$ as the subspace of exact elements of $\mathcal{A}_{\mathrm{harm}}(U)$, and similarly for $\mathcal{A}^\mathrm{e}(\riem)$ and $\overline{\mathcal{A}^\mathrm{e}(\riem)}$.  
\end{definition}

{We also consider one-forms which have zero boundary periods, which we call semi-exact.
\begin{definition} \label{de:semi_exact}
 Let $\riem$ be a bordered surface of type $(g,n)$. We say that an $L^2$ one-form $\alpha \in \mathcal{A}_{\mathrm{harm}}(\riem)$ is semi-exact if for any simple closed curve $\gamma$ homotopic to a boundary curve $\partial_k \riem$, 
 \[ \int_{\gamma}  \alpha =0. \]
 The class of semi-exact one-forms on $\riem$ is denoted $\mathcal{A}^{\mathrm{se}}_{\mathrm{harm}}(\riem)$. The holomorphic and anti-holomorphic semi-exact one-forms are denoted by $\gls{seform}(\riem)$  and 
 $\overline{\mathcal{A}^{\mathrm{se}}(\riem)}$ respectively.
\end{definition}}

The following spaces also play significant roles in this paper.
\begin{definition}\label{defn:dirichlet spaces}
The \emph{Dirichlet spaces of functions} are defined by 
\begin{align*}
   \gls{Dharm}(U)& = \{ f:U \rightarrow \mathbb{C}, f \in C^2(U), \,:\, 
   df\in L^2 (U)\,\,\,\mathrm{and}\,\, \, d\ast df =0 \},\\
   \gls{D}(U)& = \{ f:U \rightarrow \mathbb{C} \,:\, 
    df \in \mathcal{A}(U) \}, \ \text{and} \\
    \overline{\mathcal{D}(U)} & = \{ f:U \rightarrow \mathbb{C} \,:\, 
    df \in \overline{\mathcal{A}(U)} \}. \\
\end{align*}
\end{definition}
We also consider homogeneous Dirichlet spaces $\dot{\mathcal{D}}_{\mathrm{harm}}(U)$, $\dot{\mathcal{D}}(U)$, and $\dot{\overline{\mathcal{D}}}(U)$ obtained by quotienting by constants. More precisely, we say two functions $f,g \in \mathcal{D}(U)$ are equivalent if they differ by a locally constant function.  The homogeneous space $\dot{\mathcal{D}}(U)$ is the quotient of $\mathcal{D}(U)$ with respect to this equivalence relation. The same procedure can be applied to ${\overline{\mathcal{D}}}(U)$ and $\mathcal{D}_{\mathrm{harm}}(U)$ to obtain $\dot{\overline{\mathcal{D}}}(U)$ and $\dot{{\mathcal{D}}}_{\mathrm{harm}}(U)$.

We can define a degenerate inner product on $\mathcal{D}_{\mathrm{harm}}(U)$ by 
\[   (f,g)_{\mathcal{D}_{\mathrm{harm}}(U)} = (df,dg)_{\mathcal{A}_{\mathrm{harm}}(U)},   \] 
where the right hand side is the inner product (\ref{eq:form_inner_product}) restricted to elements of $\mathcal{A}_{\mathrm{harm}}(U)$.  The inner product can be used to define a seminorm on $\mathcal{D}_{\mathrm{harm}}(U)$, by letting $$\Vert f\Vert^2_{\mathcal{D}_{\mathrm{harm}}(U)}:=(df,df)_{\mathcal{A}_{\mathrm{harm}}(U)}.$$ 
This is a norm on the homogeneous spaces.

We note that if one defines the \emph{Wirtinger operators} via their local coordinate expressions
\[   \partial f = \frac{\partial f}{\partial z}\, dz,  \ \ \ 
   \overline{\partial} f =  \frac{\partial f}{\partial \bar{z}}\, d \bar{z}, \]
then the aforementioned inner product can be written as
\begin{equation} \label{eq:inner_product_with_dbar_and_d}
 (f,g)_{\mathcal{D}_{\mathrm{harm}}(U)} =i \iint_{U} \left[ \partial f \wedge \overline{\partial  g} -  \overline{\partial} f \wedge 
 \partial \overline{g} \right].    
\end{equation}
Although this implies that $\mathcal{D}(U)$ and $\overline{\mathcal{D}(U)}$ are orthogonal, there is no direct sum decomposition into $\mathcal{D}(U)$ and $\overline{\mathcal{D}(U)}$.  This is because in general there exist exact harmonic one-forms whose holomorphic and anti-holomorphic parts are not exact.  

Observe that the Dirichlet spaces are conformally invariant in the same sense as the Bergman spaces.  That is, if $f: U \rightarrow V$ is a biholomorphism then 
\begin{equation*}
 \gls{cf} h = h \circ f
\end{equation*} 
satisfies
\[  \mathbf{C}_f :\mathcal{D}(V) \rightarrow \mathcal{D}(U) \]
and this is a seminorm preserving bijection. On the homogeneous spaces it is an isometry. Similar statements hold for the anti-holomorphic and harmonic spaces. 

We also note that if $h \in \mathcal{D}(U)$ and $\tilde{h}(z) = h \circ \phi^{-1}(z)$ is the expression for $h$ in local coordinates $z= \phi(w)$ in an open set $\phi(U) \subseteq \mathbb{C}$,  then we have the local expression
\[  (h,h)_{\mathcal{D}(U)} = \iint_{\phi(U)}  |\tilde{h}'(z)|^2 dA_z   \] 
where $dA$ denotes Lebesgue measure in the plane. 
 Similar expressions hold for the other Dirichlet spaces.\\

\end{subsection}
\begin{subsection}{Harmonic measures} \label{ harmonic measures}
We start with the definition of harmonic measure in the context of bordered Riemann surfaces.
\begin{definition}\label{def:harmonic measure}
Let $\omega_k$, $k=1,\ldots,n$ be the unique harmonic function which is continuous 
 on the closure of $\riem$ and which satisfies
 \begin{equation*}
  \omega_k = \left\{ \begin{array}{ll} 
    1 & \mathrm{on}\,\,\, \partial_k \riem \\ 0 & \mathrm{on}\,\,\, \partial_j \riem, \ \  j \neq k.  \end{array}  \right.   
 \end{equation*}
 The one-forms $\gls{harmeasure}$ are the {\it harmonic measures}.  We denote the complex linear span of the harmonic measures by $\gls{Ahm}(\riem).$ Moreover we define $\ast \mathcal{A}_{\mathrm{hm}}(\riem) = \{ \ast \alpha : \alpha \in \mathcal{A}_{\mathrm{hm}}(\riem) \}.$
\end{definition}

By definition any element of $\mathcal{A}_{\mathrm{hm}}(\riem)$ is exact, and its anti-derivative $\omega$ is constant on each boundary curve. On the other hand, the elements of $\ast \mathcal{A}_{\mathrm{hm}}(\riem)$ are all closed but not necessarily exact. Elements of $\mathcal{A}_{\mathrm{hm}}(\riem)$ and $\ast \mathcal{A}_{\mathrm{hm}}(\riem)$ extend real analytically to the border, in the sense that they are restrictions to $\riem$ of harmonic one-forms on the double.  In particular they are square-integrable, which explains our choice of notation above. Thus to summarize:
 
\begin{proposition} \label{pr:harmonic_measures_L2}
 Let $\riem$ be a bordered surface of type $(g,n)$.  Then $\mathcal{A}_{\mathrm{hm}}(\riem) \subseteq \mathcal{A}_{\mathrm{harm}}^{\mathrm{e}}(\riem)$ and 
 $\ast \mathcal{A}_{\mathrm{hm}}(\riem) \subseteq \mathcal{A}_{\mathrm{harm}}(\riem)$.  
\end{proposition} 

 \begin{definition}\label{periodmatrix}
 The \emph{boundary period matrix} $\Pi_{jk}$ of a non-compact surface $\riem$ of type $(g,n)$ is defined by
 \[  \Pi_{jk} := \int_{\partial \riem} \omega_j  \ast d\omega_k = \int_{\partial_j \riem} \ast d \omega_k.   \]
\end{definition} 
 \begin{theorem} \label{th:period_matrix_invertible} If we let $j,k$ run from $1$ to $n$, omitting one fixed value $m$ say, then the resulting matrix
  $\Pi_{jk}$ is symmetric and positive definite.  
 \end{theorem}
 
 Thus $\Pi_{jk}$, $j,k = 1,\ldots\hat{m},\ldots,n$ is an invertible matrix, and we can specify $n-1$ of the boundary periods of elements of $\ast \mathcal{A}_{\mathrm{hm}}(\riem)$.
 \begin{corollary}  \label{co:boundary_periods_specified_starmeasure}
  Let $\riem$ be of type $(g,n)$ and $\lambda_1,\ldots,\lambda_n \in \mathbb{C}$ be such that $\lambda_1 + \cdots + \lambda_n =0$.  Then there is an $\alpha \in \ast \mathcal{A}_{\mathrm{hm}}(\riem)$ such that 
  \begin{equation}\label{periodjaveln}
    \int_{\partial_k \riem} \alpha = \lambda_k   
  \end{equation}   
  for all $k=1,\ldots,n$.  
 \end{corollary}
\end{subsection}
\begin{subsection}{Green's functions} 
Another basic notion which is of fundamental importance in our investigations is that of Green's functions.
 
 \begin{definition}\label{defn:greens function} Let $\riem$ be a type $(g,n)$ surface.
For fixed $z \in \riem$, we define \emph{Green's function of} $\riem$ to be a function $\gls{greenb}(w;z)$ such that 
 \begin{enumerate}
     \item for a local coordinate $\phi$ vanishing at $z$ the function $w\mapsto G_\Sigma(w;z) + \log|\phi(w)|$ is harmonic in an open neighbourhood
     of $z$;
     \item $\lim_{w \rightarrow \zeta} G_\Sigma (w;z) =0$ for any $\zeta \in \partial \riem$.   
 \end{enumerate}
 \end{definition}
  That such a function exists, follows from \cite[II.3 11H, III.1 4D]{Ahlfors_Sario}, considering $\riem$ to be a subset of its double $\riem^d$.\\  
\begin{definition}
 For compact surfaces $\mathscr{R}$, one defines the \emph{Green's function} $\gls{greenc}$ (see e.g. \cite{Royden}) as the unique function
 $\mathscr{G}(w,w_0;z,q)$ satisfying 
 \begin{enumerate}
  \item $\mathscr{G}$ is harmonic in $w$ on $\mathscr{R} \backslash \{z,q\}$;
  \item for a local coordinate $\phi$ on an open set $U$ containing $z$, $\mathscr{G}(w,w_0;z,q) + \log| \phi(w) -\phi(z) |$ is harmonic 
   for $w \in U$;
  \item for a local coordinate $\phi$ on an open set $U$ containing $q$, $\mathscr{G}(w,w_0;z,q) - \log| \phi(w) -\phi(q) |$ is harmonic 
   for $w \in U$;
  \item $\mathscr{G}(w_0,w_0;z,q)=0$ for all $z,q,w_0$.  
 \end{enumerate}
\end{definition}

 The existence of such a function is a standard fact about Riemann surfaces, see for example \cite{Royden}.  It satisfies the following identities:
 \begin{align}
  \mathscr{G}(w,w_1;z,q) & = \mathscr{G}(w,w_0;z,q) - \mathscr{G}(w_1,w_0;z,q) \label{eq:g_w_0_dependence} \\
  \mathscr{G}(w_0,w;z,q) & = - \mathscr{G}(w,w_0;z,q) \label{eq:g_interchange_w} \\
  \mathscr{G}(z,q;w,w_0) & = \mathscr{G}(w,w_0;z,q).  \label{eq:g_interchange_both}
 \end{align}
 In particular, $\mathscr{G}$ is also harmonic in $z$ where it is non-singular.
 
 \begin{remark}
  The condition (4) involving the point $w_0$ simply determines an arbitrary additive constant, and is not of any interest in the paper. 
 This is because by the property (\ref{eq:g_w_0_dependence}), $\partial_w \mathscr{G}$ is independent
 of $w_0$, and only such derivatives enter in the paper.  For this reason, we usually leave $w_0$ out of the expression for $\mathscr{G}$.
 \end{remark}

 Green's function is conformally invariant.  That is, if $\riem$ is of type $(g,n)$, and $f:\riem \rightarrow \riem'$ is conformal, then 
 {\begin{equation} \label{eq:Greens_conf_inv_gntype}
  G_{\riem'}(f(w);f(z)) = G_{\riem}(w;z). 
 \end{equation}}
 
 Similarly if $\mathscr{R}$ is compact and $f:\mathscr{R} \rightarrow \mathscr{R}'$ is a biholomorphism, then 
 \begin{equation} \label{eq:Greens_conf_inv_compact}
   \mathscr{G}_{\mathscr{R}'}(f(w),f(w_0);f(z);f(q)) = \mathscr{G}_{\mathscr{R}}(w,w_0;z,q). 
 \end{equation}

 These follow from uniqueness of Green's function; in the case of type $(g,n)$ surfaces, one also needs the fact that a biholomorphism extends to a homeomorphism of the boundary curves. 

 \end{subsection}
\begin{subsection}{ Sobolev spaces via harmonic measure and Green's functions}

     The harmonic measure and the Green's function can be used to define a conformally invariant Sobolev space that is of great significance in our investigations. 
 Let $d\omega_k$ be the harmonic measures given in Definition \ref{def:harmonic measure}. For a collar neighbourhood $U_k$ of $\partial_k \riem$ and $h_k \in \mathcal{D}_{\mathrm{harm}}(U_k)$, 
 we can fix a simple closed analytic curve $\gamma_k$ which is isotopic to $\partial_k \riem$, and define
 \begin{equation}\label{defn:integration over rough}
  \int_{\partial_k \riem}  h_k \ast d \omega_k: =  \iint_{V_k} dh_k \wedge \ast d\omega_k + \int_{\gamma_k} h_k \ast d \omega_k  
 \end{equation}   
 where $V_k$ is the region bounded by $\partial_k \riem$ and $\gamma_k$.  Here $\partial_k \riem$ is oriented positively with respect to $\riem$ and $\gamma_k$ has the same orientation as $\partial_k \riem$ (this is independent of $\gamma_k$).
 Equivalently, for a curves $\Gamma_r = \phi^{-1}(|z|=r)$ defined by a collar chart $\phi$, 
 \[   \int_{\partial_k \riem}  h_k \ast d \omega_k  = \lim_{r \nearrow 1} \int_{\Gamma_r}  h_k \ast d \omega_k.  \]
Now given $h_k \in \mathcal{D}_{\mathrm{harm}}(\riem)$ we set
   \begin{equation}\label{harmonisk matt}
    \mathscr{H}_k  := \int_{\partial_k \riem} h_k \ast d\omega_k.    
   \end{equation} 
\begin{definition}\label{def:norm on H1confU}
Let $\riem$ be a bordered surface of type $(g,n)$ and 
  let $U_k \subseteq \riem$ be collar neighbourhoods of $\partial_k \riem$ for $k=1,\ldots,n$. Set $U = U_1 \cup \cdots \cup U_n$. By ${H}^{1}_{\mathrm{conf}}(U)$ we denote the harmonic Dirichlet space $\mathcal{D}_{\mathrm{harm}}(U)$ endowed with the norm
  \begin{equation}\label{defn:Dconf norm}
 \| h \|_{H^1_{\mathrm{conf}}(U)} := \Big(\| h \|^2_{\mathcal{D}_{\mathrm{harm}}(U)} + 
  \sum_{k=1}^{n}   |\mathscr{H}_k |^2\Big)^{\frac{1}{2}} 
 \end{equation} 
 for $n>1$.  
 In the case that $n=1$,  fix a point $p\in \riem \setminus U_1$ and let $G_\Sigma$ be the Green function of $\Sigma$ given in Definition \ref{defn:greens function}. Then define instead 
  \begin{equation} \label{eq:constant_hk_definition_onecurve}
    \mathscr{H}_1 := \int_{\partial_1 \riem} h_1 \ast d G_{\riem}(w,p).
  \end{equation}

   For the Riemann surface $\riem$, assuming that $\riem$ is connected, we need only one boundary integral to obtain a norm.  If $n >1$, we can choose any fixed boundary curve $\partial_n \riem$ say, and define the norm 
 \begin{equation}\label{equivalent sobolev norm}  
     \| h \|_{{H}^1_{\mathrm{conf}}(\riem)} := \left( \| h \|^2_{\mathcal{D}_{\mathrm{harm}}(\riem)} + |\mathscr{H}_n |^2 \right)^{1/2},
 \end{equation}  
where any of the $\mathscr{H}_k$ could in fact be used in place of $\mathscr{H}_n$ and the result is an equivalent norm. In the case that $n=1$ we use (\ref{eq:constant_hk_definition_onecurve}) to define $\mathscr{H}_1$.
 \end{definition}
 \end{subsection}
\begin{subsection}{Boundary values and the bounce and the overfare operators}
We also consider a space of boundary values of Dirichet-bounded harmonic functions. In the following, let $\Gamma$ be a border of a Riemann surface $\riem$, which is homeomorphic to $\mathbb{S}^1$. For any such $\Gamma$ there is a biholomorphism 
 \[  \phi:U \rightarrow \mathbb{A}_{r,1} \]
 where $U$ is an open set in $\riem$ bounded by two Jordan curves, one of which is $\Gamma$, and $\mathbb{A}_{r,1}$ is an annulus $r<|z|<1$. We call this a collar chart of $\Gamma$. Such a $\phi$ extends homeomorphically to a map from $\Gamma$ onto $\mathbb{S}^1$; in fact, viewing $\Gamma$ as an analytic curve in the double of $\riem$, $\phi$ has a biholomorphic extension to an open neighbourhood of $\Gamma$. 
 
 Given any $h \in \mathcal{D}_{\mathrm{harm}}(\riem)$, $h$ has ``conformally non-tangential'' or ``CNT'' boundary values on $\Gamma$. That is, given a collar chart $\phi$ of $\Gamma$ , $h \circ \phi^{-1}$ has non-tangential boundary values except possibly on a Borel set of logarithmic capacity zero in $\mathbb{S}^1$; we define the CNT boundary value of $h$ at $p$ to be the non-tangential boundary value of $h \circ \phi^{-1}(p)$.  Define a null set on $\Gamma$ to be the image of a logarithmic set of capacity zero under $\phi^{-1}$. Let $\mathcal{H}(\Gamma)$ denote the set of functions on $\Gamma$ arising as CNT boundary values of elements of $\mathcal{D}_{\mathrm{harm}}(\riem)$, where we say two functions are the same if they agree except possibly on a null set.  
 It can be shown that $\mathcal{H}(\Gamma)$ is independent of the choice of collar chart.  The definition also extends immediately to finite collections of distinct borders $\Gamma = \Gamma_1 \cup \cdots \cup \Gamma_n$.  It can also be shown that $\mathcal{H}(\Gamma)$ is the set of CNT boundary values of elements of $\mathcal{D}_{\mathrm{harm}}(U)$ for any collar neighbourhood $U$ of $\Gamma$. 
 
 Treating $\Gamma$ as an analytic curve in the double, we can define the Sobolev space $H^{1/2}(\Gamma)$. Every element of $H^{1/2}(\Gamma)$ uniquely defines an element of $\mathcal{H}(\Gamma)$. 

 We also define the homogeneous space $\dot{\mathcal{H}}(\Gamma)$ of elements of $\mathcal{H}(\Gamma)$ mod functions which are constant on each connected component of $\Gamma$. Every element of the homogeneous Sobolev space $\dot{H}^{1/2}(\Gamma)$ defines a unique element of $\dot{\mathcal{H}}(\Gamma)$.  

 Finally, we will require a result regarding an operator we call the ``bounce operator''. This operator takes a Dirichlet bounded harmonic function defined in collars of the boundary to Dirichlet bounded functions on the surface with the same boundary values. 
  Let $\riem$ be a bordered surface of type $(g,n)$ as above, and
  let $U_k \subseteq \riem$ be collar neighbourhoods of $\partial_k \riem$ for $k=1,\ldots,n$. 
 \begin{definition}[Bounce operator] \label{defn:bounce op}
 Set $U = U_1 \cup \cdots \cup U_n$ and let $h:U \rightarrow \mathbb{C}$ be the function whose restriction to $U_k$ is $h_k$ for each $k=1,\ldots,n$.  We define 
  \begin{align*}
   \mathbf{G}_{U,\riem}: \mathcal{D}_{\mathrm{harm}}(U) & \rightarrow \mathcal{D}_{\mathrm{harm}}(\riem) \\ h & \mapsto H
  \end{align*}
  where $H$ is the element of $\mathcal{D}_{\mathrm{harm}}(\riem)$ with CNT boundary values agreeing with $h$ up to a null set, which exists and is unique by Theorem \cite[Theorem 3.10]{Schippers_Staubach_scattering_I}. 
 \end{definition}
  By conformal invariance of CNT limits, the bounce operator is conformally invariant, that is, if $f:\riem \rightarrow \riem'$ is a biholomorphism and $f(U)=U'$, then 
  
  \begin{equation} \label{eq:bounce_conformally_invariant}
   \gls{bounce} \mathbf{C}_f = \mathbf{C}_f \mathbf{G}_{U',\riem'}.
  \end{equation}

\begin{theorem}[Boundedness of bounce operator] \label{th:bounce_bounded} 
  Let $\riem$ and $U_k$ be as above for $k=1,\ldots,n$.
  Then $\mathbf{G}_{U,\riem}$ is bounded from ${H}^{1}_{\mathrm{conf}}(U)$ to ${H}^{1}_{\mathrm{conf}}(\riem)$.
 \end{theorem} 

\end{subsection}

\begin{subsection}{Overfare of functions}
 We will consider a compact Riemann surface $\mathscr{R}$ separated into two ``pieces'' $\riem_1$ and $\riem_2$ by a collection of quasicircles. It is possible for one or both of $\riem_1$ and $\riem_2$ to be disconnected. The precise definition of ``separating'' is the following.
  \begin{definition}  \label{de:separating_complex}
 Let $\mathscr{R}$ be a compact Riemann surface, and let $\Gamma_1,\ldots,\Gamma_m$ be a collection of  quasicircles in $\mathscr{R}$.  Denote $\Gamma = \Gamma_1 \cup \cdots \cup \Gamma_m$.  We say that $\Gamma$ \emph{separates} $\mathscr{R}$ into $\riem_1$ and $\riem_2$ if 
 \begin{enumerate}
     \item there are doubly-connected neighbourhoods $U_k$ of $\Gamma_k$ for $k=1,\ldots,n$   such that $U_k \cap U_j$ is empty for all $j \neq k$, 
     \item one of the two connected components of $U_k \backslash \Gamma_k$ is in $\riem_1$, while the other is in $\riem_2$; 
     \item $\mathscr{R} \backslash \Gamma = \riem_1 \cup \riem_2$;
     \item $\mathscr{R} \backslash \Gamma$ consists of finitely many connected components;
     \item $\riem_1$ and $\riem_2$ are disjoint.
 \end{enumerate}
 \end{definition}
 In this case we also call $\Gamma$ a separating complex of quasicircles. 
 
 As was shown in in \cite[Proposition 3.33]{Schippers_Staubach_scattering_I}, it turns out that one can identify the borders $\partial \riem_1$ and $\partial \riem_2$ pointwise with $\Gamma$. 

 Now fix a single quasicircle $\Gamma_k$ in the complex. Let $U_1$ and $U_2$ be collar neighbourhoods of $\Gamma_k$ in $\riem_1$ and $\riem_2$ respectively. In principle the set of CNT boundary values of harmonic functions in $\mathcal{D}_{harm}(U_1)$ might not agree with the set of boundary values of harmonic functions in $\mathcal{D}_{harm}(U_2)$.  In fact, it was shown in \cite{Schippers_Staubach_scattering_I} that they agree, so that we can denote this $\mathcal{H}(\Gamma_k)$.  

 More can be said.
  \begin{theorem}[Existence of overfare]\label{thm:bounded_overfare_existence}  Let $\mathscr{R}$ be a compact Riemann surface and let $\Gamma = \Gamma_1 \cup \cdots \cup \Gamma_m$ be a collection of  quasicircles separating $\mathscr{R}$ into $\riem_1$ and $\riem_2$. 
  Let $h_1 \in \mathcal{D}_{\mathrm{harm}}(\riem_1)$.  There is a $h_2 \in \mathcal{D}_{\mathrm{harm}}(\riem_2)$ whose \emph{CNT} boundary values agree with those of $h_1$ up to a null set, and this $h_2$ is unique.   
  \end{theorem}
  
    This theorem motivates the definition of the following operator which plays an important role in the scattering theory that is developed here.
  \begin{definition}\label{defn:overfare operator} 
Thus we can define the \emph{overfare operator} $\gls{overfare}_{\riem_1,\riem_2}$ by   
  \begin{align*}
   \mathbf{O}_{\riem_1,\riem_2} : \mathcal{D}_{\mathrm{harm}}(\riem_1) & \rightarrow 
   \mathcal{D}_{\mathrm{harm}}(\riem_2) \\
   h_1 & \mapsto h_2    
 \end{align*}
  \end{definition}
 One obviously has that
 \[  \mathbf{O}_{\riem_2,\riem_1} \mathbf{O}_{\riem_1,\riem_2} = \text{Id} \]
 and of course one can switch the roles of $\riem_1$ and $\riem_2$.  
 
 The overfare operator is conformally invariant.  That is, if $f:\mathscr{R} \rightarrow \mathscr{R}'$ is a biholomorphism and we set $f(\riem_k) = \riem_k'$ for $k=1,2$ then it follows immediately from conformal invariance of CNT limits that
 \begin{equation} \label{eq:overfare_conformally_invariant}
  \mathbf{O}_{\riem_1,\riem_2} \mathbf{C}_f = \mathbf{C}_f \mathbf{O}_{\riem_1',\riem_2'}.   
 \end{equation} 
 We use the abbreviated notation $\mathbf{O}_{1,2}$ and $\mathbf{O}_{2,1}$ when the underlying surfaces are clear from context.

 This operator is bounded under certain conditions \cite{Schippers_Staubach_scattering_I}. We summarize the main results here.one must assume that the originating surface is connected. 
 \begin{theorem}[Bounded overfare theorem for general quasicircles]\label{thm:bounded_overfare_Dirichlet} Let $\mathscr{R}$ be a compact Riemann surface and let $\Gamma = \Gamma_1 \cup \cdots \cup \Gamma_m$ be a collection of quasicircles separating $\mathscr{R}$ into $\riem_1$ and $\riem_2$.  Assume that $\riem_1$ is connected. 
  There is a constant $C$ such that 
     \[  \| \mathbf{O}_{1,2} h \|_{\mathcal{D}_{\mathrm{harm}}(\riem_2)}  \leq C \| h \|_{\mathcal{D}_{\mathrm{harm}}(\riem_1)}  \]
     for all $h \in \mathcal{D}_{\mathrm{harm}}(\riem_1)$.  
  \end{theorem}

 If $\riem_1$ is connected and $c$ is a constant, then $\mathbf{O}_{\riem_1,\riem_2} c$ is also constant on $\riem_2$ so the operator
 \begin{equation} \label{eq:O_dot_definition}
 \gls{doto}_{\riem_1,\riem_2} : \dot{\mathcal{D}}_{\mathrm{harm}}(\riem_1) \rightarrow  \dot{\mathcal{D}}_{\mathrm{harm}}(\riem_2) 
 \end{equation}
 is well-defined.  

 \begin{remark}
 Note that it if the originating surface (say $\riem_1$) is not connected, this is not well-defined. For example, if $\riem_1$ consists of $n$ disks and $\riem_2$ is connected, then by choosing $h$ to be constant on the connected components of $\riem_1$, the overfare $\mathbf{O}_{1,2} h$ can be an arbitrary element of $\mathcal{A}_{\mathrm{harm}}(\riem_2)$.  
 \end{remark} 
 \begin{corollary}
  Let $\mathscr{R}$ be a compact Riemann surface, separated by quasicircles into $\riem_1$ and $\riem_2$.  Assume that $\riem_1$ is connected.  Then $\dot{\mathbf{O}}_{\riem_1,\riem_2}$ is bounded with respect to the Dirichlet norm.  
 \end{corollary}

 For more regular quasicircles, constants in the overfare process can be controlled.   
 A BZM (bounded zero mode) quasicircle $\gamma$ in $\mathscr{R}$ are quasicircles with the following property. Given an open neighbourhood $U$ of $\gamma$ and a conformal map $\phi:U \rightarrow V$ where $V$ is a doubly-connected domain in $\sphere$, let $\gamma'=\phi(\gamma)$ and denote the connected components of the complement by $\Omega_1$ and $\Omega_2$.  A BZM quasicircle is one such that overfare $\mathbf{O}_{\Omega_1,\Omega_2}$ is bounded from $H^1_{\mathrm{conf}}(\Omega_1)$ to $H^1_{\mathrm{conf}}(\Omega_2)$.  For such quasicircles one can remove the assumption of connectedness. 

 \begin{theorem}[Bounded overfare theorem for BZM quasicircles]\label{thm:bounded_overfare_conf}  Let $\mathscr{R}$ be a compact Riemann surface and let $\Gamma = \Gamma_1 \cup \cdots \cup \Gamma_m$ be a collection of \emph{BZM} quasicircles separating $\mathscr{R}$ into $\riem_1$ and $\riem_2$. 
  There is a constant $C$ such that 
     \[  \| \mathbf{O}_{1,2} h \|_{H^1_{\mathrm{conf}}(\riem_2)}  \leq C \| h \|_{H^1_{\mathrm{conf}}(\riem_1)}  \]
     for all $h \in \mathcal{D}_{\mathrm{harm}}(\riem_1)$.  
  \end{theorem}
\end{subsection}
\end{section}

\begin{section}{Schiffer and Cauchy-Royden operators} \label{se:Schiffer_Cauchy}
\begin{subsection}{Definitions of Schiffer and Cauchy operators}\label{subsec:defSchiffer} 
 In this section, we define the Schiffer and Cauchy-Royden operators, and prove their boundedness. We derive identities relating them, and demonstrate their bounedness. We also show that for quasicircles, the limiting integrals in the definition of the Cauchy-Royden operator are the same from either side, up to constants -- this point is analytically subtle, because quasicircles need not be rectifiable. This allows the formulation of a Plemelj-Sokhotski type result for quasicircles. 

Denote by $\mathscr{G}$ Green's function of $\mathscr{R}$, and let $G_{\Sigma_k}$ be Green's functions of $\riem_k$, $k=1,2$.   Here, if $\riem_k$ has more than one connected component, then $G_{\Sigma_k}$ stands for the function whose restriction to each connected component is the Green's function of that component.

  First, we define the Schiffer operators.  
 To that end, we need to define certain bi-differentials, which will be the integral kernels of the Schiffer operators. 
 
 \begin{definition}
  For a compact Riemann surface $\mathscr{R}$ with Green's function $\mathscr{G}(w,w_0;z,q),$ the \emph{Schiffer kernel} is defined by
 \[  \gls{schifferk}_{\mathscr{R}}(z,w) =  \frac{1}{\pi i} \partial_z \partial_w \mathscr{G}(w,w_0;z,q),    \]
 and the \emph{Bergman kernel} is given by
 \[  \gls{bergmank}_{\mathscr{R}} (z,w) = - \frac{1}{\pi i} \partial_z \overline{\partial}_{{w}} \mathscr{G}(w,w_0;z,q).    \]

 For a non-compact surface $\riem$ of type $(g,n)$ with Green's function $g_\riem$, we define
 \[  L_\riem(z,w) =   \frac{1}{\pi i} \partial_z \partial_w G_\riem(w,z),  \]
 and 
 \[  K_\riem (z,w) = - \frac{1}{\pi i} \partial_z \overline{\partial}_{{w}}G_\riem(w,z).  \]
 \end{definition}

  The kernel functions satisfy the following:
 \begin{enumerate} 
  \item[$(1)$] $L_{\mathscr{R}}$ and $K_{\mathscr{R}}$ are independent of $q$ and $w_0$.  
  \item[$(2)$] $K_{\mathscr{R}}$ is holomorphic in $z$ for fixed $w$, and anti-holomorphic in $w$ for fixed $z$.
  \item[$(3)$] $L_{\mathscr{R}}$ is holomorphic in $w$ and $z$, except for a pole of order two when $w=z$.
\item[$(4)$] $L_{\mathscr{R}}(z,w)=L_{\mathscr{R}}(w,z)$.  
  \item[$(5)$] $K_{\mathscr{R}}(w,z)= - \overline{K_{\mathscr{R}}(z,w)}$.  
 \end{enumerate}

For non-compact Riemann surfaces $\riem$ with Green's function, $(2)-(5)$ hold with $L_{\mathscr{R}}$ and $K_{\mathscr{R}}$ replaced by $L_\riem$ and $K_\riem$. Moreover for any vector $v$ tangent to $\Gamma^{w}_{\varepsilon}$ at a point $z$, we have
  \begin{equation} \label{eq:level_curve_identity}
   \overline{K_\riem(z,w)}(\cdot,v) = -L_\riem(z,w)(\cdot,v).
  \end{equation} 
   Note that here we can treat the boundary as an analytic curve in the double, so that it makes sense to consider vectors tangent to the boundary.
 Also, the well-known reproducing property of the Bergman kernel holds, i.e.
 \begin{equation}  \label{eq:Bergman_reproducing}
  \iint_\riem K_\riem(z,w) \wedge h(w) = h(z),
 \end{equation}
 for $h \in A(\riem)$ \cite{Royden}.

Another basic fact about the kernels above is that they are conformally invariant.  That is, for a compact surface $\mathscr{R}$ and a biholomorphism $f:\mathscr{R} \rightarrow \mathscr{R}'$ we have 
 \begin{align} \label{eq:kernels_conformally_invariant_compact}
  (f^* \times f^*)  \; L_\mathscr{R'} & = L_{\mathscr{R}}  \nonumber \\
  (f^* \times f^*) \;  K_\mathscr{R'} & = K_{\mathscr{R}} 
 \end{align}
 and similarly, for surfaces $\riem$, $\riem'$ of type $(g,n)$ and a biholomorphism
 $f:\riem \rightarrow \riem'$, 
 \begin{align} \label{eq:kernels_conformally_invariant_gntype}
  (f^* \times f^*)  \; L_{\riem'} & = L_{{\riem}}  \nonumber \\
  (f^* \times f^*)\;  K_{\riem'} & = K_{{\riem}}.  
 \end{align}
 These follow immediately from conformal invariance of Green's function 
 (\ref{eq:Greens_conf_inv_compact},\ref{eq:Greens_conf_inv_gntype}).\\ 
 
 \begin{definition}\label{def: restriction ops}
For $k=1,2$ define the \emph{restriction operators}
 \begin{align*}
  \mathbf{R}_{\riem_k}:\mathcal{A}(\mathscr{R}) & \rightarrow \mathcal{A}(\riem_k) \\
  \alpha & \mapsto \left. \alpha \right|_{\riem_k}
 \end{align*}
 and 
 \begin{align*}
  \mathbf{R}^0_{\riem_k}: \mathcal{A}(\riem_1 \cup \riem_2) & \rightarrow \mathcal{A}(\riem_k) \\
    \alpha & \mapsto \left. \alpha \right|_{\riem_k}.
 \end{align*}
 \end{definition}
 
 It is obvious that these are bounded operators. In a similar way, we also define the restriction operator
 \[  \gls{harmrest}_{\riem_k} : \mathcal{A}_{\mathrm{harm}}(\mathscr{R}) \rightarrow \mathcal{A}_{\mathrm{harm}}(\riem_k).   \] 

Having the Bergman and Schiffer kernels and the restriction operators at hand, we can now define the Schiffer operators as follows.
\begin{definition}
For $k=1,2$, we define the {\it Schiffer comparison operators} by
 \begin{align*}
  \gls{T}_{\riem_k}: \overline{\mathcal{A}(\riem_k)} & \rightarrow \mathcal{A}(\riem_1 \cup \riem_2)  \\
  \overline{\alpha} & \mapsto \iint_{\riem_k} L_{\mathscr{R}}(\cdot,w) \wedge \overline{\alpha(w)}.
 \end{align*}
\\
and 
\begin{align*}
 \gls{S}_{\riem_k}: \mathcal{A}(\riem_k) & \rightarrow \mathcal{A}(\mathscr{R}) \\
  \alpha & \mapsto \iint_{\riem_k}  K_{\mathscr{R}}(\cdot,w) \wedge \alpha(w).
\end{align*}
The integral defining $\mathbf{T}_{\riem_k}$ is interpreted as a principal value integral whenever $z \in \riem_k$.  
Also, we define for $j,k \in \{1,2\}$
 \begin{equation}\label{defn: T sigmajsigmak}
     \gls{Tmix} = \mathbf{R}^0_{\riem_k} \mathbf{T}_{\riem_j}: \overline{\mathcal{A}(\riem_j)} \rightarrow \mathcal{A}(\riem_k). 
 \end{equation}   
\end{definition} 
 \begin{theorem} \label{th:Schiffer_operators_bounded}  $\mathbf{T}_{\riem_k}$, $\mathbf{T}_{\riem_j,\riem_k}$, and $\mathbf{S}_{\riem_k}$ are bounded for all $j,k =1,2$.  
 \end{theorem}
 \begin{proof}
The operator $\mathbf{T}_{\Sigma_k}$ is defined by integration against the $L_{\mathscr{R}}$-Kernel which in local coordinates $\zeta = f(z)$, $\eta = f(w)$ is given by 
\[  (f \times f)^*L_{\mathscr{R}}(\zeta,\eta) = \frac{d\zeta \, d \eta}{\pi(\zeta-\eta)^2} + \alpha(\zeta,\eta) \] where $\alpha$ is a holomorphic bi-differential. Since the singular part of the kernel is a Calder\'on-Zygmund kernel we can use the theory of singular integral operators to conclude that the operators with kernel $L_\mathscr{R}(z,w)$ are bounded on $L^2$.  The same proof applies to $L_{\riem}$. The boundedness of $\mathbf{T}_{\riem_j, \riem_k}$ follows from this and the fact that $\mathbf{R}
^0_{\riem_k}$ is also bounded.\\
That the operator $\mathbf{S}_{\Sigma_k}$ is bounded and its image is $\mathcal{A}(\mathscr{R}),$  can be seen from the fact that the kernel $K_{\mathscr{R}}(., w)$ is holomorphic in $w$ and $\mathscr{R}$ is compact.
 \end{proof}
 
{\bf{Notation.}} As in the case of the overfare operator $\mathbf{O}$, we will use the notations 
 \[  \mathbf{S}_k, \ \ \mathbf{T}_{j,k}, \ \ \mathbf{T}_k, \ \ \mathbf{R}_{k}, \ \  \mathbf{P}_k=\mathbf{P}_{\riem_k}, \]
 wherever the choice of surfaces $\riem_1$ and $\riem_2$ is clear from context.\\
 
For any operator $\mathbf{M}$, we define the complex conjugate operator by 
 \[  \overline{\mathbf{M}} \overline{\alpha} = \overline{\mathbf{M} \alpha}  \]
 So for example 
 \[   \overline{\mathbf{T}}_{1,2}:\mathcal{A}(\riem_1) \rightarrow \overline{\mathbf{A}(\riem_2)} \]
 and similarly for $\overline{\mathbf{R}}_{\riem_k}$, etc.  
 
 The restriction operator is conformally invariant by conformal invariance of Bergman space of one-forms. 
 By (\ref{eq:kernels_conformally_invariant_compact}), the operators $\mathbf{T}$ and $\mathbf{S}$ are also conformally invariant.  Explicitly, if $f:\mathscr{R} \rightarrow \mathscr{R}'$ is a biholomorphism between compact surfaces, and we denote
 $\riem_k'= f(\riem_k)$ for $k=1,2$, then 
 \begin{align}  \label{eq:Schiffer_operators_conformally_invariant}
   f^* \; \mathbf{R}_{\riem_k'} & = \mathbf{R}_{\riem_k} \; f^* \nonumber \\
   f^*  \; \mathbf{R}^0_{\riem_k'} & = \mathbf{R}^0_{\riem_k} \; f^* \nonumber \\
   f^* \; \mathbf{T}_{\riem_k'} & = \mathbf{T}_{\riem_k}\; f^* \\
   f^* \; \mathbf{T}_{\riem_j',\riem_k'} & = \mathbf{T}_{\riem_j,\riem_k} \; f^* \nonumber \\ 
   f^* \; \mathbf{S}_{\riem_k'} & = \mathbf{S}_{\riem_k} \; f^*. \nonumber
 \end{align}
 
 The following basic lemma which we will used frequently in this paper, is crucial in establishing some of the forthcoming identities concerning Schiffer and Bergman kernels.
\begin{lemma} \label{le:limiting_circle_Schiffer_identity}  Fix 
 a  point $z$ and local coordinates $\phi$ near $z$.  Let $\gamma_r$ be a curve such that $|\phi(w)-\phi(z)|=r$.  Then 
 for any holomorphic one-form $\alpha$ defined near $z$, and fixed $q$, we have 
 \[   \lim_{r \searrow 0} \int_{\gamma_r,w} \frac{1}{\pi i} \partial_z \mathscr{G}(w;z,q) \overline{\alpha(w)}=0    \]
 and
 \[   \lim_{r \searrow 0} \int_{\gamma_r,w} \frac{1}{\pi i} \partial_z \mathscr{G}(w;z,q)  {\alpha(w)}= \alpha(z).    \]
 Similarly for $z \in \riem_k$ we have 
 \[   \lim_{r \searrow 0} \int_{\gamma_r,w} \frac{1}{\pi i} \partial_z {G}_{\Sigma_k}(w;z) \overline{\alpha(w)}=0    \]
 and
 \[   \lim_{r \searrow 0} \int_{\gamma_r,w} \frac{1}{\pi i} \partial_z {G}_{\Sigma_k}(w;z)  {\alpha(w)}= \alpha(z).    \]
\end{lemma}
\begin{proof}
 In coordinates denote $\phi(w)=\zeta$ and $\phi(z)=\eta$.  We have, writing $\alpha(w) = h(\zeta)d\zeta$ (with $h$ holomorphic)
 and observing that  
 \[ \partial_z \mathscr{G}(w;z,q) = G_\eta(\zeta) d\eta  + \frac{1}{2}\frac{1}{\zeta-\eta}d\eta    \]
 where $G_\eta(\zeta)$ is non-singular at $\eta$, 
 \begin{align*}
  \lim_{r \searrow 0} \int_{\gamma_r,w} \frac{1}{\pi i} \partial_z \mathscr{G}(w;z,q)  {\alpha(w)} & = 
     \lim_{r \searrow 0} \int_{|\zeta-\eta|=r,\zeta} \frac{1}{2 \pi i} \frac{d\zeta}{\zeta-\eta} h(\zeta) d\eta = h(\eta) d\eta 
     \\ & = \alpha(z).
 \end{align*}
 Similarly 
 \begin{align*}
  \lim_{r \searrow 0} \int_{\gamma_r,w} \frac{1}{\pi i} \partial_z \mathscr{G}(w;z,q) \overline{\alpha(w)} & = 
     \lim_{r \searrow 0} \int_{|\zeta-\eta|=r,\zeta} \frac{1}{2 \pi i} \frac{d\overline{\zeta}}{\zeta-\eta} \overline{h(\zeta)} d {\eta} 
     \\ & = 0
 \end{align*}
 by writing a power series expansion of $h$ and integrating in polar coordinates.  The proof for $g_k$ is identical. 
\end{proof}

 We will frequently use the following identity, which we refer to as \emph{Schiffer's identity}.
 \begin{theorem}  \label{th:Schiffer_vanishing_identity}  Let $\riem$ be a bordered surface of type $(g,n)$.  
  For all $\overline{\alpha} \in \overline{\mathcal{A}(\riem)}$ 
  \[   \iint_{\riem,w}  L_\riem(z,w) \wedge \overline{\alpha(w)}=0. \]
 \end{theorem}
 \begin{proof} Let $\riem$ be embedded in its double $\riem^d$, so that the boundary is an analytic curve.  Fixing $z \in \riem$ and applying Stokes' theorem we then have
 \[ \iint_{\riem} L_\riem(z,w) \overline{\alpha(w)} = \int_{\partial \riem} \frac{1}{\pi i} \partial_z G_{\riem}(z;w) \overline{\alpha(w)} - \lim_{r \searrow 0}\int_{\gamma_{r,w}}
 \frac{1}{\pi i} \partial_z G_{\riem}(z;w) \overline{\alpha(w)}    \]
 with $\gamma_{r,w}$ as in Lemma \ref{le:limiting_circle_Schiffer_identity}. 
 The claim now follows from Lemma \ref{le:limiting_circle_Schiffer_identity} and the fact that for any fixed $z$, $\partial_z g_k(z;w)$ vanishes for all $w \in \partial \riem$.
 \end{proof} 
 This implies that  
 \begin{equation}  \label{eq:nonsingular_Schiffer}
  \mathbf{T}_{1,1} \alpha(z)= \iint_{\riem_1,w} \left(L_\mathscr{R}(z,w) - L_{\riem_1}(z,w) \right) 
  \wedge \overline{\alpha(w)}.  \end{equation}
 This desingularizes the kernel function, and will be useful below in some of the proofs. Incidentally, it also gives a direct way to see that the principal value integral defining $\mathbf{T}_{1,1}$ is independent of the choice of local coordinate, even though the omitted disks in the integral depend on this choice.

\begin{example} \label{ex:sphere_kernels}
 If $\mathscr{R}$ is the Riemann sphere $\sphere$, we have
 \[  \mathscr{G}(w,\infty;z,q) = - \log{ \frac{|w-z|}{|w-q|}}.     \]
 Thus
 \[  K_{\sphere}(z,w) =0  \]
 and
 \[  L_{\sphere}(z,w) = - \frac{1}{2 \pi i} \frac{dw \,dz}{(w-z)^2}.  \]
 So the Schiffer operators are given by, for $\overline{\alpha(z)} = \overline{h(z)}d \bar{z}$,
 \[  \ \mathbf{T}_1  \, \overline{\alpha} (z) =
    \frac{1}{\pi} \iint_{\riem_1} \frac{\overline{h(w)}}{(w-z)^2} \frac{d\bar{w} 
    \wedge d w}{2i}  \cdot dz  \]
 and $\mathbf{S}_k=0$, $k=1,2$.    
 
 By the uniformization theorem, if $\mathscr{R}$ is a compact surface of genus zero, it is biholomorphic to $\sphere$.  Thus by conformal invariance of the Schiffer kernels \eqref{eq:kernels_conformally_invariant_compact} we see that $K_{\mathscr{R}}=0$ and $\mathbf{S}_k=0$.   
 \end{example}    
 \begin{example} \label{ex:disk_kernels}    
  For $\riem = \disk$, we have
 \[  \mathscr{G}(z,w) = - \log{\frac{|z-w|}{|1-\bar{w}z|}}.  \]
 So
 \[  L_{\disk}(z,w) = \frac{-1}{2\pi i} \frac{dw \, dz}{(w-z)^2}    \]
 and 
 \[  K_{\disk}(z,w) = \frac{1}{2\pi i}  \frac{d\overline{w} \, dz}{(1-\bar{w}z)^2}. \]
 
 For a M\"obius transformation $M$, we can verify the identities
 \[   \frac{M'(w) M'(z)}{(M(w)-M(z))^2} = \frac{1}{(z-w)^2}     \]
 and 
 \[  \frac{\overline{M'(w)} M'(z)}{(1-\overline{M(w)}M(z))^2} = \frac{1}{(1-\bar{w}z)^2}. \]
 By conformal invariance of the Schiffer kernels (\ref{eq:kernels_conformally_invariant_gntype}) we see that for any disk or half plane $U$
 \[  L_{U}(z,w) = \frac{-1}{2\pi i} \frac{dw \, dz}{(w-z)^2}    \]
 and 
 \[  K_{U}(z,w) = \frac{1}{2\pi i}  \frac{d\overline{w} \, dz}{(1-\bar{w}z)^2}. \]
\end{example}    
 
 Next we consider a kind of Cauchy operator defined using Green's function.  This operator involves integrals over the separating quasicircles, which are not in general rectifiable.  So we define the integral using limits along analytic curves which approach the quasicircle.  This is well-defined by the Anchor Lemmas \cite[Lemmas 3.14 and 3.15]{Schippers_Staubach_scattering_I}. Furthermore for quasicircles, up to constants, this limit does not depend on the side from which the curve is approached.  This significant fact, which depends on the bounded overfare theorem, is one of the motivations for the use of quasicircles throughout the paper.    
 We now define the Cauchy operators.
   
   \begin{definition}
   Let $A = A_1 \cup \cdots \cup A_n$ be a union of non-intersecting collar neighbourhoods of $\Gamma$ in $\riem_1$. For $q \in \mathscr{R} \backslash \Gamma$ and $h \in \mathcal{D}_{\mathrm{harm}}(A)$  define, for $z \in \mathscr{R} \backslash \Gamma$, the \emph{Cauchy-{Royden} operator} by
  \begin{align}  \label{eq:J_definition}
      \gls{crop}(\Gamma) h(z) = {-} {\frac{1}{\pi i}} \int_{\partial \riem_1} \partial_w \mathscr{G}(w;z,q) h(w) = {-} {\frac{1}{\pi i}} \sum_{k=1}^n 
      \int_{\partial_k \riem_1} \partial_w \mathscr{G}(w;z,q) h(w),
  \end{align} 

and the \emph{restricted Cauchy-Royden operators} by
\begin{equation}
    \gls{rcrop}(\Gamma)=\mathbf{J}_1^q(\Gamma) h |_{\riem_k}
\end{equation}
where, as will be shown later,  $\mathbf{J}_1^q(\Gamma):\mathcal{D}_{\mathrm{harm}}(\riem_1) \rightarrow \mathcal{D}_{\mathrm{harm}}(\riem_1 \cup \riem_2)$ and $\mathbf{J}_{1,k}^q(\Gamma): \mathcal{D}_{\mathrm{harm}}(\riem_1) \rightarrow \mathcal{D}_{\mathrm{harm}}(\riem_k).$

\end{definition} 
    Note that by Definition \ref{de:separating_complex} and \cite[Proposition 2.17]{Schippers_Staubach_scattering_I} non-intersecting collections of collar charts exist, and the integral exists by the first anchor lemma \cite[Lemma 3.14]{Schippers_Staubach_scattering_I}. 
  
 The Cauchy operator is closely related to the Schiffer operators, as the following theorem shows. 
 { \begin{theorem}  \label{th:jump_derivatives}  
  For all $h \in \mathcal{D}_{\mathrm{harm}}(\riem_1)$ and any $q \in \mathscr{R} \backslash \Gamma$, 
   \begin{align*}
    \partial \mathbf{J}_{1} ^q(\Gamma)h(z)   & = \mathbf{T}_{1,2} \overline{\partial} h(z),\, \, \, z\in \riem_2  \\
    \partial \mathbf{J}_{1}^q(\Gamma)  h(z) & = \partial h + \mathbf{T}_{1,1} \overline{\partial} h,\, \, \, z\in \riem_1   \\
    \overline{\partial} \mathbf{J}_{1}^q(\Gamma) h(z) & = \overline{\mathbf{S}}_1 \overline{\partial} h(z),\, \, \, z\in \riem_1 \cup \riem_2 
   \end{align*}
  \end{theorem}
  }
  \begin{remark}
   There is a sign error in \cite{Schippers_Staubach_Plemelj}, {which is corrected here.}
  \end{remark}
  \begin{proof}
    { Assume first that $q \in \riem_2$. The first claim follows from the application of the Stokes theorem to \eqref{eq:J_definition} and the fact that the integrand is non-singular.  Similarly for $q,z \in \riem_2$, the third claim follows 
    from the same reasoning.  
    
    The second claim also follows from Stokes theorem, namely if $\Gamma_\varepsilon$ are curves given by $|w-z|=\varepsilon$ in local coordinates, positively oriented with respect to $z$,
    \begin{align} \label{eq:add_q_temp}
     \partial \mathbf{J}_{1} ^q(\Gamma)h(z) & = \partial_z \left( - \frac{1}{\pi i} \lim_{\varepsilon \searrow 0}
     \int_{\Gamma_\varepsilon} (\partial_w \mathscr{G}(w;z,q) - \partial_w G_{\Sigma_1}(w,z) )\,h(w) \right) \nonumber \\
     & \ \ \ \   -   \partial_z \lim_{\varepsilon \searrow 0} \frac{1}{\pi i} \int_{\Gamma_\varepsilon}
     \partial_w G_{\Sigma_1} (w,z) \,h(w) \nonumber\\
     & = \partial_z \left(  \frac{1}{\pi i}  
     \iint_{\riem_1} (\partial_w \mathscr{G}(w;z,q) - \partial_w G_{\Sigma_1}(w,z) )\wedge_w \overline{\partial} h(w) \right) \nonumber \\
     & \ \ \ \  - \partial_z \lim_{\varepsilon \searrow 0} \frac{1}{\pi i} \int_{\Gamma_\varepsilon}
     \partial_w G_{\Sigma_1} (w,z) \,h(w) \nonumber\\
     & = \frac{1}{\pi i}  
     \iint_{\riem_1} (\partial_z \partial_w \mathscr{G}(w;z,q) - \partial_z \partial_w G_{\Sigma_1}(w,z) )
      \wedge_w \overline{\partial} h(w)  + \partial h(z) 
    \end{align}
    where we have used the harmonicity of $h$. Derivation under the integral sign in the first term is justified by the fact that the integrand
    of the first term is non-singular and holomorphic in $z$ for each $w\in \riem_1$, and that $$\iint_{\riem_1, w}|( \partial_w \mathscr{G}(w;z,q) - \partial_w G_{\Sigma_1}(w,z) ) \wedge_w \overline{\partial}_{{w}} h(w)|$$ is locally bounded in $z$. 
    
    Similarly removing the singularity using $\partial_w G_{\riem_1}$, and then using the harmonicity of $h$ and Stokes' theorem yield that
    \begin{align*}
      \overline{\partial} \mathbf{J}_{1} ^q(\Gamma) h(z) & = - \overline{\partial}_{{z}} \frac{1}{\pi i} \lim_{\varepsilon \searrow 0}
     \int_{\Gamma_\varepsilon} ( \partial_w \mathscr{G}(w;z,q) -  \partial_w G_{\Sigma_1}(w,z) )\,h(w)  + \overline{\partial} h(z) \\
     & =  \frac{1}{\pi i} \iint_{\riem_1}(\overline{\partial}_{{z}} \partial_w \mathscr{G}(w;z,q) - \overline{\partial}_{{z}} \partial_w G_{\Sigma_1}(w,z) ) \wedge_w \overline{\partial}_{{w}} h(w)  + \overline{\partial} h(z).
    \end{align*}
    The third claim now follows by observing that the second term in the integral is just $- \overline{\partial} h$ because the integrand is just the complex conjugate of the Bergman kernel.
    
    Now assume that $q \in \riem_1$.  We show the second claim in the theorem.  We argue as in equation (\ref{eq:add_q_temp}), except that we must also add 
    a term $\partial_w G_{\Sigma_1}(w;q) h(w)$.  We obtain instead
    \[ \partial \mathbf{J}_{1} ^q(\Gamma)h(z)=  \frac{1}{\pi i} 
     \iint_{\riem_1} (\partial_z \partial_w \mathscr{G}(w;z,q) - \partial_z \partial_w G_{\Sigma_1}(w;z) )
      \wedge_w \overline{\partial}_{{w}} h(w)  + \partial_z \left( h(z) + h(q) \right)  \]
     and the claim follows from $\partial_z h(q) =0$. The remaining claims follow similarly.} 
  \end{proof}
  
  Combining this with Theorem \ref{th:Schiffer_operators_bounded}, we obtain
  \begin{theorem}   \label{th:jump_bounded_Dirichlet}
   $\mathbf{J}_1^q(\Gamma):\mathcal{D}_{\mathrm{harm}}(\riem_1) \rightarrow \mathcal{D}_{\mathrm{harm}}(\riem_1 \cup \riem_2)$ 
   is bounded with respect to the Dirichlet seminorm. 
  \end{theorem}
  Of course, the roles of the surfaces $\riem_1$ and $\riem_2$ can be switched.\\

  It follows from conformal invariance of Green's functions (\ref{eq:Greens_conf_inv_compact},\ref{eq:Greens_conf_inv_gntype}) and Dirichlet space that the Cauchy-Royden operator $\mathbf{J}$ is conformally invariant.  That is,
  if $f:\mathscr{R} \rightarrow \mathscr{R}'$ is a biholomorphism between compact surfaces, $\Gamma' = f(\Gamma)$, and $\riem_k' = f(\riem_k)$ for $k=1,2$, then 
  \begin{equation} 
    \mathbf{C}_f \mathbf{J}_k(\Gamma') = \mathbf{J}_k (\Gamma)  \mathbf{C}_f 
  \end{equation}
  which of course implies the same for $\mathbf{J}_{j,k}(\Gamma)$ and $\mathbf{J}_{j,k}(\Gamma')$ for $j,k=1,2$.

  The operator $\mathbf{J}_1^q$ is in fact bounded with respect to the $H^1_{\text{conf}}$-norm.  
  \begin{theorem} \label{th:J_bounded_Hconf}
   $\mathbf{J}^q_{1,k}(\Gamma):H^1_{\mathrm{conf}}(\riem_1) \rightarrow H^1_{\mathrm{conf}}(\riem_k)$ is bounded for $k=1,2$. 
  \end{theorem}
  Note that strictly speaking, this is not stronger than Theorem \ref{th:jump_bounded_Dirichlet}, since that theorem shows that the $H^1_{\mathrm{conf}}$-norm is not necessary to control the Dirichlet norm of the output. 
  
  The proof requires a lemma.   
  \begin{lemma} \label{le:difference_Greens_in_Dirichlet} Let $g_1$ denote Green's function of $\riem_1$ 
   for  $k=1$ and $\mathscr{G}$ denote Green's function of $\mathscr{R}$.  Then for any fixed $p \in \riem_1$ and $q \in \riem_2$
   \[  \partial_w \mathscr{G}(w,w_0;p,q) - \partial_w G_{\Sigma_1}(w;p) \in \mathcal{A}_{\mathrm{harm}}(\riem_1).     \]
   If $q \in \riem_1$ then 
    \[  \partial_w \mathscr{G}(w,w_0;p,q) - \partial_w G_{\Sigma_1}(w;p) + \partial_w G_{\Sigma_1}(w;q) \in \mathcal{A}_{\mathrm{harm}}(\riem_1).     \]
   The same holds with $1$ and $2$ switched. 
  \end{lemma}
  \begin{proof}
   By definitions of $\mathscr{G}$ and $g_1$, this is a non-singular harmonic function on $\riem_1$.  So it suffices to show that the function is in $\mathcal{A}_{\mathrm{harm}}(A)$ for some collar neighbourhood of $A= A_1\cup \cdots \cup A_n$ of $\partial \riem_1$. 
   The first term $\partial_w \mathscr{G}(w,w_0;p,q)$ is obviously in $\mathcal{A}_{\mathrm{harm}}(A)$ since it is holomorphic on an open neighbourhood of the closure of $A$. By conformal invariance of Green's function and the Bergman norm, the second term can be evaluated on the double $\riem^d$, where the boundary $\partial \riem$ is then an analytic curve.  Assuming that the inner boundary of $A$ consists of $n$ analytic curves $\Gamma = \Gamma_1 \cup \cdots \cup \Gamma_n$ we get 
   \[ \iint_{A} \partial_{\bar{w}} G_{\Sigma_1}(w;p) \wedge_w \partial_w G_{\Sigma_1}(w;p) = - \int_{\Gamma} G_{\Sigma_1}(w;p)\,  \partial_w G_{\Sigma_1}(w;p) <\infty  \]
   where we have used Stokes' theorem and the fact that $g_1$ vanishes on $\partial \riem_1$.  The proof for $q \in \riem_1$ is similar.
  \end{proof}

  We can now prove Theorem \ref{th:J_bounded_Hconf}.
  \begin{proof}(of Theorem \ref{th:J_bounded_Hconf}.)  
   By Theorem \ref{th:jump_bounded_Dirichlet} and \cite[Lemma 3.21]{Schippers_Staubach_scattering_I}, to prove that $\mathbf{J}_{1,k}^q$ is bounded, it's enough to show that for a $p$ in one of the connected components of $\riem_k$, $|(\mathbf{J}_{1,k}^q h)(p)| \lesssim \| h \|_{H^1_{\mathrm{conf}}}.$ 
   
   We first do the case of $\mathbf{J}_{1,1}^q$. First assume that $q \in \riem_2$, and $p \in \riem_1$. Then, we have using the reproducing property of Green's function (see \cite[Proposition 3.19]{Schippers_Staubach_scattering_I} for the proof in this analytic setting) and Stokes' theorem 
   \begin{align*}
     \mathbf{J}_{11}^q h(p) & = - \lim_{\epsilon \searrow 0} \frac{1}{\pi i}  \int_{\Gamma_\epsilon} \partial_w \mathscr{G}(w;p,q)h(w) \\
     & =  \lim_{\epsilon \searrow 0} \frac{1}{\pi i}  \int_{\Gamma_\epsilon} \left( - \partial_w \mathscr{G}(w;p,q) 
     + \partial_w G_{\Sigma_1}(w;p) \right) h(w) + h(p) \\
     & =   \frac{1}{\pi i}  \iint_{\riem_1} \left( \partial_w \mathscr{G}(w;p,q) 
     - \partial_w G_{\Sigma_1}(w;p) \right) \wedge_w \overline{\partial} h(w) + h(p). \\
   \end{align*}
   By \cite[Lemma 3.21]{Schippers_Staubach_scattering_I} we have $|h(p)|\leq \| h \|_{H^1_{\mathrm{conf}}(\riem_1)}$, and by Cauchy-Schwarz and Lemma 
   \ref{le:difference_Greens_in_Dirichlet} we obtain
   \[  \left|  \frac{1}{\pi i}  \iint_{\riem_1} \left( \partial_w \mathscr{G}(w;p,q) 
     - \partial_w G_{\Sigma_1}(w;p) \right) \wedge_w \overline{\partial} h(w)\right| \leq C \| \overline{\partial} h \|_{\mathcal{A}_{\mathrm{harm}}(\riem_2)} \leq C \| h \|_{H^1_{\mathrm{conf}}(\riem_1)}.  \]
    If on the other hand $q \in \riem_1$, the claim follows similarly from the second part of Lemma \ref{le:difference_Greens_in_Dirichlet} and  
    \[  \mathbf{J}_{11}^q h(p) = \frac{1}{\pi i}  \iint_{\riem_1} \left( \partial_w \mathscr{G}(w;p,q) 
     - \partial_w G_{\Sigma_1}(w;p) + \partial_w G_{\Sigma_1}(w;q) \right) \wedge_w \overline{\partial} h(w) + h(p) - h(q).  \] 
     Because any point can be used in \cite[Lemma 3.21]{Schippers_Staubach_scattering_I} to obtain a norm equivalent to the $H^1_{\mathrm{conf}}$ norm, it holds that $|h(q)| \lesssim \| h \|_{H^1_{\mathrm{conf}(\riem_1)}}$ for the norm determined by $p$. 
     
     Now we estimate $\mathbf{J}_{12}^q$. If $q \in \riem_2$, then for $p \in \riem_2$ we have similarly by Stokes' theorem
     \[  \left| (\mathbf{J}_{11}^q h)(p) \right| = \left| 
       \lim_{\epsilon \searrow 0} \frac{1}{\pi i}  \iint_{\riem_1} \partial_w \mathscr{G}(w;p,q)\wedge_w \overline{\partial} h(w) \right|     \]
     so the claim follows once again by Cauchy-Schwarz and the fact that $\partial_w \mathscr{G}(w;p,q) \in \mathcal{D}(\riem_1)$ for $p,q \in \riem_2$.  
     The case that $q \in \riem_1$ can be dealt with as above.
  \end{proof}

  Like the Cauchy integral, this operator reproduces holomorphic functions (up to constants).
  \begin{theorem} \label{th:jump_on_holomorphic}
   Assume that $h \in \mathcal{D}(\riem_1)$.  If $q \in \riem_1$, 
   let $c_q(z)$ be the function which is equal to $h(q)$ in the connected component of $\riem_k$ containing $q$ and $0$ otherwise.  
   Then 
   \[   \mathbf{J}^q_{1,1} h (z) = \left\{  \begin{array}{ll} h(z) - c_q(z)  & q \in \riem_1 \\
      h(z)   & q \in \riem_2 \end{array} \right. \]
   and 
   \[  \mathbf{J}^q_{1,2} h (z) = \left\{  \begin{array}{ll} - c_q(z)  & q \in \riem_1 \\
      0  & q \in \riem_2 \end{array} \right.  \]
   This holds with the roles of $1$ and $2$ interchanged.
  \end{theorem}
  \begin{proof}  
   Since $h \in \mathcal{D}(\riem_1)$, the integrand of $\mathbf{J}^q_1 h$ is holomorphic, except for possible singularities at $z$ and $q$ depending on their locations. 
   If $z$ is contained in $\riem_1$, and $C_r$ are curves given by $|w-z|=r$ in local coordinates, positively oriented with respect to $z$, then 
   \[  - \frac{1}{\pi i} \lim_{r \searrow 0}  \int_{C_r} \partial_w \mathscr{G}(w;z,q) h(w) = h(z)    \]
   and if $q$ is in $\riem_1$ and $C_r$ are the curves $|w-q|=r$ then 
   \[   - \frac{1}{\pi i} \lim_{r \searrow 0}  \int_{C_r} \partial_w \mathscr{G}(w;z,q) h(w) = -h(q).    \]
   The claim follows from Stokes' theorem applied to the connected components of $\riem_1$.  
  \end{proof}
  In particular, for any $q \notin \Gamma$, and any locally constant function $c$,  $\mathbf{J}^q_1 c$ is also locally constant.  Thus we obtain a well-defined operator 
  \[  \gls{Jdot}: \dot{\mathcal{D}}_{\mathrm{harm}}(\riem_1) \rightarrow   \dot{\mathcal{D}}(\riem_1 \cup \riem_2).  \]
  The Dirichlet norm becomes a seminorm on the homogeneous space, and $\dot{\mathbf{J}}_1$ is bounded with respect to this norm. It is easily verified that $\dot{\mathbf{J}}_1$ is independent of $q$.

   Next we will prove some results about the interaction with $\mathbf{J}_1^q$ with the bounce and overfare operators. 
  \begin{proposition} \label{prop:J_agrees_with_bounce} Let $A= A_1 \cup \cdots \cup A_n$ be a union of collar neighbourhoods $A_k$ of $\Gamma_k$ in $\riem$.  For $h \in \mathcal{D}_{\mathrm{harm}}(A)$ 
  \[   \mathbf{J}_1^q(\Gamma) h  = \mathbf{J}_1^q(\Gamma) \mathbf{G}_{A,\riem_1} h.     \]
  \end{proposition}
  \begin{proof}
    The kernel $\partial_w \mathscr{G}(w;z)$ is holomorphic in $w$ in an open neighbourhood of the boundary $\Gamma$, so $\partial_w \mathscr{G}(w;z) \in \mathcal{A}(A)$. The claim now follows from the second anchor lemma \cite[Lemma 3.15]{Schippers_Staubach_scattering_I}.
  \end{proof}
  \begin{remark} In fact, this applies for any collection of strip-cutting Jordan curves, but we do not require this here.
  \end{remark}
  A deeper result is that for quasicircles, the limiting integral is the same from both sides up to constants; for BZM quasicircles, they are the same. 
  \begin{theorem}  \label{th:J_same_both_sides} The following statements hold:\\
   \begin{enumerate} 
   \item [$(1)$] If $\Gamma$ consists of \emph{BZM} quasicircles, then 
   for any $h \in \mathcal{D}_{\mathrm{harm}}(\riem_1)$
    \[ \mathbf{J}_1^q(\Gamma) h = - \mathbf{J}_2^q(\Gamma) \mathbf{O}_{1,2} h.  \] 
  \item [$(2)$] If $\riem_1$ is connected and $\Gamma$ is an arbitrary complex of quasicircles, then for any $\dot{h} \in \dot{\mathcal{D}}_{\mathrm{harm}}(\riem_1)$
   \[ \dot{\mathbf{J}}_1(\Gamma) \dot{h} = - \dot{\mathbf{J}}_2(\Gamma) \dot{\mathbf{O}}_{1,2} \dot{h}.  \] 
  \end{enumerate}
  \end{theorem}
  \begin{proof} We prove first claim.
   Choose doubly-connected neighbourhoods $U_1,\ldots,U_n$ of the boundary curves $\Gamma$ with charts $\phi_m:U_m \rightarrow \mathbb{A}_m$, where each $\mathbb{A}_m = \{z : r_m < |z|<R_m \}$ is an annular region in the plane. For $k=1,2$ let $A_m^k = U_m \cap \riem_k$ be collar neighbourhoods of $\Gamma$ in $\riem_k$, and set $B_m^k=\phi_m(A_m^k)$.  We claim that $\mathcal{D}_{\text{harm}}(U_m)$ is dense in
   $\mathcal{D}_{\text{harm}}(A^k_m)$ for each $k=1,\ldots,n$ with respect to the $H^1_{\mathrm{conf}}$ norms. 
   By conformal invariance of the $H^1_{\mathrm{conf}}$-norm, it is enough to prove that $\mathcal{D}_{\mathrm{harm}}(\mathbb{A}_m)$ is dense in $\mathcal{D}_{\mathrm{harm}}(B_m^k)$. 
   This follows immediately from the fact that polynomials 
   \[ p(z) = \sum_{l=s}^t z^{l}, \ \ s,t \in \mathbb{Z}, t\geq s   \]
   are dense in both $\mathcal{D}(\mathbb{A}_m)$ and $\mathcal{D}(B_m^k)$.
   
   Now let $h \in \mathcal{D}_{\text{harm}}$, and let $\Gamma_\epsilon^k$ be the level sets of Green's function $g_k$ for $k=1,2$, which are analytic curves for $\epsilon$ sufficiently close to zero. Letting $E_\epsilon$ be the region enclosed by these analytic curves, we have 
   \[ - \frac{1}{\pi i} \int_{\Gamma^2_\epsilon} \partial_w \mathscr{G}(w;z,q) \,h(w) -  \frac{1}{\pi i} \int_{\Gamma^1_\epsilon} \partial_w \mathscr{G}(w;z,q) \,h(w)  =  \iint_{E_\epsilon} \partial_w \mathscr{G}(w;z,q) \wedge_w \overline{\partial} h(w)  \]
   (note that the reversal of orientation of the contour integrals is taken into account). 
   Applying the Cauchy-Schwarz inequality to the right hand side we get 
   \[ \left| \frac{1}{\pi i} \int_{\Gamma^2_\epsilon} \partial_w \mathscr{G}(w;z,q) \,h(w) +  \frac{1}{\pi i} \int_{\Gamma^1_\epsilon} \partial_w \mathscr{G}(w;z,q) \,h(w) \right| \leq   \|\partial_w \mathscr{G}(w;z,q) \|_{\mathcal{A}_{\mathrm{harm}}(E_\epsilon)}   \| \overline{\partial} h(w) \|_{\mathcal{A}_{\mathrm{harm}}(E_\epsilon)}.   \]
   Since quasicircles have measure zero and $\cap_\epsilon E_\epsilon = \Gamma$, the right hand side goes to zero as $\epsilon \searrow 0$.  Thus
   \[ - \lim_{\epsilon \searrow 0} \frac{1}{\pi i} \int_{\Gamma^2_\epsilon} \partial_w \mathscr{G}(w;z,q) \,h(w) = \lim_{\epsilon \searrow 0}  \frac{1}{\pi i} \int_{\Gamma^1_\epsilon} \partial_w \mathscr{G}(w;z,q) \,h(w). \]  
   Now set $U=U_1 \cup \cdots U_n$ and $A^k=A^k_1 \cup \cdots \cup A^k_n$ and assume that $h \in \mathcal{D}_{\mathrm{harm}}(U)$.
   Using the above, together with the second anchor lemma \cite[Lemma 3.15]{Schippers_Staubach_scattering_I} and the fact that $\mathbf{G}_{A^2,\riem_2} h = \mathbf{O}_{1,2}\mathbf{G}_{A^1,\riem_1} h$, we have
   \begin{align} \label{eq:thingy_temp}
       \mathbf{J}_1^q \mathbf{G}_{A^1,\riem_1} h & = \mathbf{J}_1^q h = - \mathbf{J}_2^q h \nonumber \\
       & = -\mathbf{J}_2^q \mathbf{G}_{A^2,\riem_2} h \nonumber \\
       & = -\mathbf{J}_2^q \mathbf{O}_{1,2} \mathbf{G}_{A^1,\riem_1} h.
   \end{align}
   The proof is completed by the density of $H^1_{\text{conf}}(U)$ in $H^1_{\text{conf}}(A^1)$, the density of $\mathbf{G}_{A^1,\riem_1} H^1_{\text{conf}}(A^1)$ in $H^1_{\text{conf}}(\riem_1)$ \cite[Theorem 3.29]{Schippers_Staubach_scattering_I}, and the boundedness of $\mathbf{J}^q_k$, $\mathbf{O}_{1,2}$, and $\mathbf{G}_{A^k,\riem_k}$ (Theorems \ref{th:J_bounded_Hconf}, \ref{thm:bounded_overfare_conf}, and \ref{th:bounce_bounded}).  
   
   {The proof of the second claim follows the same line, but requires a bit of care with the constants. First, observe that 
   \[  \mathbf{G}_{A^1,\riem_1} :\mathcal{D}_{\mathrm{harm}}(A^1) \rightarrow \dot{\mathcal{D}}_{\mathrm{harm}}(\riem_1)   \]
   is well-defined.  Furthermore, it is bounded with respect to the $H^1_{\mathrm{conf}}(A^1)$ and $\dot{\mathcal{D}}(\riem_1)$ norms, since the $H^1_{\mathrm{conf}}(\riem_1)$ norm dominates the Dirichlet seminorm. The image is dense.  
   
   By \eqref{eq:thingy_temp} we have 
   \[  \dot{\mathbf{J}}_1 \dot{H} = - \dot{\mathbf{J}}_2 \dot{\mathbf{O}}_{1,2} \dot{H}  \]
   for all $\dot{H}$ arising from $H \in \mathbf{G}_{A^1,\riem_1} \mathcal{D}_{\mathrm{harm}}(A^1)$.  The second claim now follows from boundedness of $\dot{J}$ (Theorem \ref{th:jump_bounded_Dirichlet}) and boundedness of $\dot{\mathbf{O}}_{1,2}$ (Theorem \ref{thm:bounded_overfare_Dirichlet}).}

  \end{proof}   
  {
  \begin{remark}
   On the other hand, $\mathbf{G}_{A^1,\riem_1} :\mathcal{D}_{\mathrm{harm}}(A^1) \rightarrow \dot{\mathcal{D}}_{\mathrm{harm}}(\riem_1)$ is not bounded with respect to the Dirichlet seminorms. To see this, let $\riem_1$ be the annulus $\{z : 1<|z|<4 \}$, and let $A^1 = \{ z: 1<|z|<2 \} \cup \{ z: 3<|z|<4 \}$.  The claim is falsified by considering the function which is $1$ on $\{ z: 1<|z|<2 \}$ and $N$ on $\{ z: 3<|z|<4 \}$, and letting $N \rightarrow \infty$.  
  \end{remark}}

  The operator $\mathbf{J}^q_k$ satisfies a Plemelj-Sokhotski jump formula. Although we will not emphasize this role in this paper, the following theorem represents this fact. 
    The following improvement of Theorem 4.13 in \cite{Schippers_Staubach_Plemelj}, can be viewed as a CNT version of the Plemelj-Sokhotski jump formula. However, rather than referring to a function on the curve, we express the result in terms of the extensions into $\riem_1$, with the help of the overfare operator.  
   \begin{theorem} \label{th:Overfare_with_correction_functions}  
    The following statements hold:\\
    \begin{enumerate}
        \item  [$(1)$] Assume that every curve in the complex $\Gamma$ is a \emph{BZM} quasicircle. 
    For any $h \in \mathcal{D}_{\mathrm{harm}}(\riem_1)$, 
    \[  \mathbf{O}_{2,1} \mathbf{J}^q_{1,2} h = \mathbf{J}^q_{1,1} h - h   \]
    and for all $h \in \mathcal{D}_{\mathrm{harm}}(\riem_2)$
    \[  \mathbf{O}_{2,1} \mathbf{J}^q_{2,2} h - \mathbf{J}^q_{2,1} h =\mathbf{O}_{2,1} h.   \]
    \item  [$(2)$] Assume that $\riem_2$ is connected and $\Gamma$ is an arbitrary complex of quasicircles. Then for any $\dot{h} \in \dot{\mathcal{D}}_{\mathrm{harm}}(\riem_1)$, 
    \[  \dot{\mathbf{O}}_{2,1} \dot{\mathbf{J}}_{1,2} \dot{h} = \dot{\mathbf{J}}_{1,1} \dot{h} - \dot{h} \]
    and 
    \[  \dot{\mathbf{O}}_{2,1} \dot{\mathbf{J}}_{2,2} \dot{h} - \dot{\mathbf{J}}_{2,1} \dot{h} = \dot{\mathbf{O}}_{2,1} \dot{h}.   \]
    \end{enumerate}
   \end{theorem}
   \begin{proof}  We prove (1). 
    Let $A_1$ be a collar neighbourhood of $\Gamma$ in $\riem_1$. 
     Assume that the boundary $\Gamma'$ is an analytic curve which is isotopic in the closure of $A_1$ to $\Gamma$.  Orient both curves positively with respect to $\riem_1$.    By shrinking $A$ and moving $\Gamma'$ we may assume that  $q$ is not in $A_1$. We assume that $z$ is in $A_1$.  Let $\gamma_r$ denote the curve $|w-z|=r$ in local coordinates, oriented positively with respect to $z$.
     
     Applying Stokes' theorem and assuming $h \in \mathcal{D}(A_1)$, for  $z \in A_1$ we have
    {\begin{equation*}
         - \frac{1}{\pi i}  \int_{\Gamma} \partial_w 
         \mathscr{G}(w;z,q) h(w) + \frac{1}{\pi i} \int_{\Gamma'} \partial_w \mathscr{G}(w;z,q) h(w)
          = - \frac{1}{\pi i} \lim_{r \searrow 0} \int_{\gamma_r}  \partial_w \mathscr{G}(w;z,q) h(w) = h(z).  
     \end{equation*}}
     The integrand of the second integral on the left hand side is holomorphic in $w$. Therefore the integral equals the limiting integral $-\mathbf{J}^q_{1,2} h$ for any $z \in \riem_2$, and furthermore, the integral over $\Gamma'$ is a harmonic function $H_2$ in $z$ extending $-\mathbf{J}^q_{1,2} h$ into $A_1 \cup \text{cl} (\riem_2)$. For $z \in A_1$ this function thus satisfies 
     \[  \mathbf{J}^q_{1,1} h(z) - H_2(z) = h(z).        \]
     Since the CNT boundary values of $\mathbf{J}^q_{1,2} h$ equal those of the extension $H_2$, we have proved that 
     \[  \mathbf{O}_{2,1} \mathbf{J}^q_{1,2} h(z) = \mathbf{G}_{A,\riem_1} H_2  =
      \mathbf{J}^q_{1,1} h(z) - \mathbf{G}_{A,\riem_1} h(z)  \]
      by the equation above. 
      
      Applying the second anchor lemma \cite[Lemma 3.15]{Schippers_Staubach_scattering_I}, we obtain for all $h \in \mathcal{D}(A_1)$
      \begin{equation} \label{eq:jump_on_dense_temp1}
        \mathbf{O}_{2,1} \mathbf{J}^q_{1,2} \mathbf{G}_{A_1,\riem_1} h(z)    =
      \mathbf{J}^q_{1,1}  \mathbf{G}_{A_1,\riem_1} h(z) - \mathbf{G}_{A_1,\riem_1} h(z),   
      \end{equation}
      as claimed.   A similar argument shows that for a collar neighbourhood $A_2$ of $\Gamma$ in $\riem_2$, for all $h \in \mathcal{D}(A_2)$ we have
      \begin{equation} \label{eq:jump_on_dense_temp2}
       { \mathbf{O}_{2,1} \mathbf{J}^q_{2,2} \mathbf{G}_{A_2,\riem_2} h(z)    =
      \mathbf{J}^q_{2,1}  \mathbf{G}_{A_2,\riem_2} h(z) +   \mathbf{O}_{2,1}\mathbf{G}_{A_2,\riem_2} h(z).}   
      \end{equation}
      Observe that the derivations of \eqref{eq:jump_on_dense_temp1} and \eqref{eq:jump_on_dense_temp2} required neither the assumption that $\Gamma$ is a BZM quasicircle nor the assumption that $\riem_2$ is connected. 
      
      A density argument completes the proof of the first claim of (1). 
     Now $\mathbf{G}_{A_1,\riem_1} \mathcal{D}(A_1)$ is dense in $H^1_{\mathrm{conf}}(\riem_1)$ by \cite[Theorem 3.29]{Schippers_Staubach_scattering_I}.
    Thus it is enough to prove the claim for $\mathbf{G}_{A_1,\riem_1} h$ for $h \in \mathcal{D}(A_1)$, since $\mathbf{G}_{A_1,\riem_1}$, $\mathbf{O}_{1,2}$, and $\mathbf{J}_{1,k}^q$ are bounded with respect to $H^1_{\mathrm{conf}}$ by Theorems \ref{th:bounce_bounded}, \ref{thm:bounded_overfare_conf}, and \ref{th:J_bounded_Hconf} respectively.  A similar density argument using \eqref{eq:jump_on_dense_temp2} shows the second claim of (1).  
      
      {We now prove the first claim of (2).  For any $h \in \mathcal{D}_{\mathrm{harm}}(A_1)$ we have that \eqref{eq:jump_on_dense_temp1} holds.  Arguing as in the proof of part (2) of Theorem \ref{th:J_same_both_sides}, we have that the set of $\dot{H}$ in $\dot{\mathcal{D}}_{\mathrm{harm}}(\riem_1)$
      of the form $H = \mathbf{G}_{A_1,\riem_1} h$ for $h \in \mathcal{D}(A_1)$ are dense in $\dot{\mathcal{D}}_{\mathrm{harm}}(\riem_1)$.  By \eqref{eq:jump_on_dense_temp1} we have for such $\dot{H}$ that
      \[   \dot{\mathbf{O}}_{2,1} \dot{\mathbf{J}}_{1,2} \dot{H} = \dot{\mathbf{J}}_{1,1} \dot{H} - \dot{H}.   \]
      The claim now follows from boundedness of $\dot{\mathbf{J}}_1$ and $\dot{\mathbf{O}}_{2,1}$, which is Theorems \ref{th:jump_bounded_Dirichlet} and \ref{thm:bounded_overfare_Dirichlet} respectively.  The proof of the second claim is similar. 
      }
   \end{proof}
  
\end{subsection}

\begin{subsection}{Adjoint identities for the Schiffer operators}
 \label{se:Schiffer_adjoint_identities}
 In this section, we prove some identities for the Schiffer operators.

 \begin{theorem}[Adjoint identities]  \label{th:adjoint_identities} For $j,k =1,2$ 
  $($not necessarily distinct$)$, 
  \[    \mathbf{T}_{j,k}^* = \overline{\mathbf{T}}_{k,j}.      \]
  If the genus of $\mathscr{R}$ is non-zero, then for $k=1,2$ we have 
  \[  \mathbf{R}_k^* = \mathbf{S}_k.      \]
 \end{theorem}
 \begin{proof}
  In the case of a single quasicircle $\Gamma$, these are \cite[Theorems 3.11, 3.12]{Schippers_Staubach_Plemelj}. The proofs there hold for the case of 
  several quasicircles.    
 \end{proof}
 Also, observe that if we define 
 \begin{equation}\label{defn: Shk}
      \gls{Sharm} = \mathbf{S}_k \mathbf{P}_k + \overline{\mathbf{S}}_k \overline{\mathbf{P}}_k : \mathcal{A}_{\text{harm}}(\riem_k) \rightarrow \mathcal{A}_{\text{harm}}(\mathscr{R}) 
 \end{equation} 
 then we have by an elementary computation
 \begin{corollary} If the genus of $\mathscr{R}$ is non-zero then for $k=1,2$
   $(\mathbf{R}_k^{\mathrm{h}})^* = \mathbf{S}^{\mathrm{h}}_k$.
 \end{corollary}
 
 \begin{theorem}[Quadratic adjoint identities, Part I]  \label{th:quadratic_adjoint_one} 
    If $\mathscr{R}$ is of genus $g>0$ then
  \[  \mathbf{S}_1 \mathbf{S}_1^* + \mathbf{S}_2 \mathbf{S}_2^* = \mathbf{I}  \]
  and 
  \[  \overline{\mathbf{S}}_1 \overline{\mathbf{S}}_1^* + \overline{\mathbf{S}}_2 \overline{\mathbf{S}}_2^* = \mathbf{I}.      \]
 \end{theorem}
 \begin{proof}
  These identities follow from the 
  reproducing property of Bergman kernel.  For $\alpha \in \mathcal{A}(\mathscr{R})$ we have, using the fact that quasicircles have measure zero (see e.g. \cite{Lehto})
  \begin{align*}
    (\mathbf{S}_1 \mathbf{S}_1^* + \mathbf{S}_2 \mathbf{S}_2^* )\, \alpha (z) & = \iint_{\riem_1} K_{\mathscr{R}}(z,w) \alpha(w)  + \iint_{\riem_2} K_{\mathscr{R}}(z,w) \alpha(w) \\ & = \iint_{\mathscr{R}} K_{\mathscr{R}}(z,w) \alpha(w)
    = \alpha(w)
  \end{align*}
  which proves the first identity.  The second identity is the complex conjugate of the first.
 \end{proof}
 We will repeatedly use the fact that quasicircles have measure zero in this way, in order to express an integral over $\mathscr{R}$ as the sum of integrals over $\riem_1$ and $\riem_2$, without mentioning it each time. 
 
 To prove quadratic adjoint identities involving $\mathbf{T}$, we require a lemma.
 \begin{lemma}  \label{le:double_L_vanishes_general_genus}
 For any $w,z \in \mathscr{R}$, 
  \[  \iint_{\mathscr{R},\zeta} L_{\mathscr{R}}(z,\zeta) \wedge \overline{L_\mathscr{R}(\zeta,w)} = K_\mathscr{R}(z,w)      \]
  where the integral is interpreted as a principal value integral.  In particular, if the genus of $\mathscr{R}$ is zero then
   \[  \iint_{\mathscr{R},\zeta} L_{\mathscr{R}}(z,\zeta) \wedge \overline{L_\mathscr{R}(\zeta,w)} = 0.     \]
 \end{lemma}
 {\begin{proof}
  Fix $w=w_0$ and $z=z_0$ in the integrals above.  Let  $\gamma^\varepsilon_{w_0}$ be curves such that $\psi \circ \gamma^\varepsilon_{w_0}$ are given by $|\eta|=\varepsilon$ for a chart $\psi(\zeta)=\eta$ near $w_0$, which takes $w_0$ to $0$. Define $\gamma^\varepsilon_{z_0}$ similarly.  Let $\mathscr{R}^\varepsilon$ be the region in $\mathscr{R}$ bounded by the curves $\gamma^\varepsilon_{z_0}$ and $\gamma^\varepsilon_{w_0}$ but not containing $z_0$ and $w_0$.  We assume these curves have positive orientation with respect to $\mathscr{R}^\varepsilon$. 
  
   In these coordinates, we have (setting $u=\psi(z)$ and $v=\psi(w)$)
 
  \[ \frac{1}{\pi i} \partial_{\bar{w}}    \mathscr{G}(\zeta,w) = - \left( \frac{1}{2\pi i} 
    \frac{1}{\bar{\eta}} + \overline{\phi(\eta)} \right)d\overline{v}  \]
  where $\phi$ is a smooth function of $\eta$ which is uniformly bounded near $z$. We suppress dependence on $u$ and $v$ because we are fixing $w=w_0$ and $z=z_0$; however, we retain $d\bar{v}$ to emphasize that the quantity is a form in the $w$ variable.
  
  We then have
  \begin{align*}
      \iint_\mathscr{R} L_\mathscr{R}(z;\zeta) \wedge_\zeta \overline{L_\mathscr{R}(\zeta;w)} & = \lim_{\varepsilon \searrow 0} \iint_{\mathscr{R}^\varepsilon}  L_\mathscr{R}(z;\zeta) \wedge_\zeta \overline{L_\mathscr{R}(\zeta;w)} \\
      & = \lim_{\varepsilon \searrow 0} \left[ \int_{\gamma^\varepsilon_{w_0}}  L_{\mathscr{R}}(z;\zeta)   \overline{ \frac{1}{\pi i} \partial_w \mathscr{G}(\zeta,w)}   
      + \int_{\gamma^\varepsilon_{z_0}} L_\mathscr{R}(z;\zeta)  \overline{ \frac{1}{\pi i} \partial_w \mathscr{G}(\zeta,w)}   \right]. \\
  \end{align*}
  
  For $\zeta$ near $w_0$, in $\eta$-coordinates we have $L(z;\zeta) = \rho(\eta)\, d\eta$ for some holomorphic function $\rho(\eta)$. So the first term is (where the integral is with respect to $\eta$)
  \[  \lim_{\varepsilon \rightarrow 0} \int_{\psi \circ \gamma^\varepsilon_{w_0}} \left( \rho(\eta)   \left( \frac{1}{2\pi i} 
    \frac{1}{\overline{\eta}} + \overline{\phi(\eta)}  \right)  d\eta \right)\,du \, d\bar{v} =0.   \]
   Here we have used the fact that if $\eta = \varepsilon e^{i\theta}$ then 
  \[  \frac{d\eta}{\overline{\eta}}=e^{2 i \theta} d\theta.      \]     
  
  On the other hand, in the second term it is $L_{\mathscr{R}}$ that is singular while $\partial_w \mathscr{G}$ is non-singular.  Fix $w$ and ignore the $dw$.  Now let $\eta=\phi(\zeta)$ be a holomorphic coordinate vanishing at $z$ and let the level curves $\gamma_z^\varepsilon$ be as above, and let $u=\phi(z)$ and $v=\phi(w)$. We may write 
  \[  \frac{1}{\pi i} \partial_{\bar{w}} \mathscr{G}(\zeta,w) = \left( h_1(\eta) + \overline{h_2(\eta)}\right) d\overline{u}       \]
  where $h_1$ and $h_2$ are holomorphic.  Now writing 
  $\overline{h_2(\eta)} = a_0 + a_1 \bar{\eta} + a_2 \bar{\eta}^2 + \cdots$
  and observing that (suppressing the fixed $z$, but keeping $dv$ to indicate the fact that it is a form) 
  \[  L_\mathscr{R}(z,\zeta) = \left(-  \frac{1}{2 \pi i} \frac{d\eta}{\eta^2} + k(\eta) d\eta  \right) dv    \]
  where $k$ is holomorphic.  Integrating this kernel against $\overline{h_2(\eta)} d\bar{u}$ is zero in the limit, so from this it is easily seen that 
  \begin{align*}
    \lim_{\varepsilon \rightarrow 0} \int_{\gamma^\varepsilon_z} L_\mathscr{R}(z;\zeta)   \overline{ \frac{1}{\pi i} \partial_w \mathscr{G}(\zeta,w)}  &  =   
    - \lim_{\varepsilon \rightarrow 0} \int_{\gamma^\varepsilon_z} L_\mathscr{R}(z;\zeta)    \frac{1}{\pi i} \partial_{\bar{w}} \mathscr{G}(\zeta,w) \, du d \bar{v} \\
    & = \lim_{\varepsilon \rightarrow 0} \int_{\psi \circ \gamma^\varepsilon_z}  
    h_1(\eta) \frac{d \eta}{2 \pi i \eta^2} \, du \, d \bar{v}  =  - h_1'(0)  \, du \, d \bar{v}
  \end{align*}
  where the final sign change results from the fact that the curve $\gamma^\varepsilon_z$ is negatively oriented with respect to $z$.   Now observing that 
  \[  - h_1'(0)  \, du \, d \bar{v}= - \frac{1}{\pi i} \partial_z \partial_{\bar{w}} \mathscr{G}(z,w)        = K_\mathscr{R}(z,w) \]
  the proof of the first claim is complete.   
  
  In the case that $\mathscr{R}$ has genus zero, by  Example \ref{ex:sphere_kernels} we have that $K_{\mathscr{R}}=0$, which proves the second claim. 
 \end{proof}}
 
Using this lemma, we can prove the following. 
 \begin{theorem}[Quadratic adjoint identities, part II]  \label{th:general_adjoint_double_identities}
  If $\mathscr{R}$ has genus $g>0$, then 
  \begin{align*}
    \mathbf{I} & = \mathbf{T}_{1,1}^* \mathbf{T}_{1,1} + \mathbf{T}_{1,2}^* \mathbf{T}_{1,2} + \overline{\mathbf{S}}_1^* \overline{\mathbf{S}}_1 \\
    \mathbf{I} & = \mathbf{T}_{2,1}^* \mathbf{T}_{2,1} + \mathbf{T}_{2,2}^* \mathbf{T}_{2,2} + \overline{\mathbf{S}}_2^* \overline{\mathbf{S}}_2 \\
    0 & = \mathbf{T}_{1,1}^* \mathbf{T}_{2,1} + \mathbf{T}_{1,2}^* \mathbf{T}_{2,2} + \overline{\mathbf{S}}_1^* \overline{\mathbf{S}}_2 \\
    0 & = \mathbf{T}_{2,2}^* \mathbf{T}_{1,2} + \mathbf{T}_{2,1}^* \mathbf{T}_{1,1} + \overline{\mathbf{S}}_2^* \overline{\mathbf{S}}_1. 
  \end{align*}
  If $\mathscr{R}$ has genus $g=0$, then 
   \begin{align*}
    \mathbf{I} & = \mathbf{T}_{1,1}^* \mathbf{T}_{1,1} + \mathbf{T}_{1,2}^* \mathbf{T}_{1,2}  \\
    \mathbf{I} & = \mathbf{T}_{2,1}^* \mathbf{T}_{2,1} + \mathbf{T}_{2,2}^* \mathbf{T}_{2,2}  \\
    0 & = \mathbf{T}_{1,1}^* \mathbf{T}_{2,1} + \mathbf{T}_{1,2}^* \mathbf{T}_{2,2} \\
    0 & = \mathbf{T}_{2,2}^* \mathbf{T}_{1,2} + \mathbf{T}_{2,1}^* \mathbf{T}_{1,1}.
  \end{align*}
 \end{theorem}
 \begin{proof}  Assume that $\mathscr{R}$ has genus $g >0$.    
  The first identity was proven in \cite{Schippers_Staubach_Plemelj}, in the case of one boundary curve.  The proof given there extends verbatim to the case of several boundary curves and disconnected components without issue. The second identity is just the first, with the roles of $\riem_1$ and $\riem_2$ switched. 
  The fourth identity is just the third with the roles of $\riem_1$ and $\riem_2$ interchanged.  So it is enough to prove the third identity. 
  
  Let $v \in \mathcal{A}(\riem_1)$ and $u \in \mathcal{A}(\riem_2)$, and denote the Schiffer kernels of $\riem_k$ by $L_k$ for $k=1,2$.  Then setting $M = \mathbf{T}_{1,1}^* \mathbf{T}_{2,1} + \mathbf{T}_{1,2}^*\mathbf{T}_{2,2}$ and applying Theorems \ref{th:Schiffer_vanishing_identity} and \ref{th:adjoint_identities} yields that 
  \begin{align*}
   2i \left< v, M u \right> & = 
    2i \left< \mathbf{T}_{1,1} v , \mathbf{T}_{2,1} u \right> + 2i \left< \mathbf{T}_{1,2} v,\mathbf{T}_{2,2} u \right> \\
    & =  \iint_{1,z} \iint_{1,w} \iint_{2,\zeta} L_{\mathscr{R}}(z,w) \wedge_w \overline{v(w)}
    \wedge_z \overline{L_{\mathscr{R}}(z,\zeta)} \wedge_\zeta u(\zeta)  \\
    & \ \ \ + \iint_{2,z} \iint_{1,w} \iint_{2,\zeta} L_{\mathscr{R}}(z,w) \wedge_w \overline{v(w)}
    \wedge_z \overline{L_{\mathscr{R}}(z,\zeta)} \wedge_\zeta u(\zeta) \\
    & = \iint_{1,z} \iint_{1,w} \iint_{2,\zeta} \left(L_{\mathscr{R}}(z,w) - L_1(z,w) \right) \wedge_w \overline{v(w)} \wedge_z
    \overline{L_{\mathscr{R}}(z,\zeta)}\wedge_\zeta u(\zeta) \\
    & \ \ \ + \iint_{2,z} \iint_{1,w} \iint_{2,\zeta} L_{\mathscr{R}}(z,w) \wedge_w \overline{v(w)}
    \wedge_z \overline{\left( L_{\mathscr{R}}(z,\zeta) -L_2(z,\zeta) \right)} \wedge_\zeta u(\zeta)  \\
    & =  \iint_{1,w} \iint_{2,\zeta}  \overline{v(w)} \wedge_w u(\zeta) \wedge_\zeta \iint_{1,z} \left(L_{\mathscr{R}}(z,w) - L_1(z,w) \right) \wedge_z 
    \overline{L_{\mathscr{R}}(z,\zeta)}      \\
    & \ \ \ + \iint_{1,w} \iint_{2,\zeta} \overline{v(w)}  \wedge_w u(\zeta) \wedge_\zeta \iint_{2,z} L_{\mathscr{R}}(z,w) \wedge_z 
    \overline{\left( L_{\mathscr{R}}(z,\zeta) -L_2(z,\zeta) \right)}. 
  \end{align*}
  Reorganizing the two terms above we obtain 
  \begin{align*}
   2i \left< v, M u \right> & = - \iint_{1,w} \iint_{2,\zeta}  \overline{v(w)} \wedge_w u(\zeta) \wedge_\zeta \iint_{1,z} L_1(z,w) \wedge_z
    \overline{L_{\mathscr{R}}(z,\zeta)}      \\
    & \ \ \ - \iint_{1,w} \iint_{2,\zeta} \overline{v(w)} \wedge_w u(\zeta) \wedge_\zeta \iint_{2,z} L_{\mathscr{R}}(z,w) \wedge_z 
    \overline{L_2(z,\zeta)}  \\
    & \ \ \ + 2 \iint_{1,w} \iint_{2,\zeta} \overline{v(w)} \wedge_w u(\zeta) \wedge_\zeta \iint_{\mathscr{R},z} L_{\mathscr{R}}(z,w) \wedge_z
    \overline{L_{\mathscr{R}}(z,\zeta)}.
  \end{align*}
  Observing that $\zeta$ is not in the closure of $\riem_1$, the first term vanishes by Schiffer's identity (i.e. Theorem \ref{th:Schiffer_vanishing_identity}) applied to $L_1$.  Similarly the second term vanishes because $z$ is not in the closure of $\Sigma_2$.  Thus applying Lemma \ref{le:double_L_vanishes_general_genus}  and Theorem \ref{th:adjoint_identities} yield that
  \begin{align*}
   2i \left< v,Mu \right> & =  2 \iint_{1,w} \iint_{2,\zeta} \overline{v(w)} \wedge_w u(\zeta) \wedge_\zeta K_{\mathscr{R}}(w,\zeta) \\
   & = - {2}  \iint_{1,w} \iint_{2,\zeta} \overline{v(w)} \wedge_w K_{\mathscr{R}}(w,\zeta)  \wedge_\zeta u(\zeta) \\
   & = - 2i \left< v, \overline{\mathbf{R}}_1  \overline{\mathbf{S}}_2 u\right>.
  \end{align*}
  This completes the proof in the case of non-zero genus.  If $\mathscr{R}$ has genus zero, then all the computations above are still valid.  We need only observe that in the last step $K_{\mathscr{R}}=0$ by Example \ref{ex:sphere_kernels}.    
 \end{proof}

 Taking complex conjugates and using the adjoint identities of Theorem \ref{th:adjoint_identities}, we also have for non-zero genus: 
 \begin{align} \label{eq:other_general_adjoint_double}
    \mathbf{I} & = \mathbf{T}_{1,1} \mathbf{T}_{1,1}^* + \mathbf{T}_{2,1} \mathbf{T}_{2,1}^* + {\mathbf{S}}_1^* {\mathbf{S}}_1 \nonumber \\
    \mathbf{I} & = \mathbf{T}_{1,2} \mathbf{T}_{1,2}^* + \mathbf{T}_{2,2} \mathbf{T}_{2,2}^* + {\mathbf{S}}_2^* {\mathbf{S}}_2 \nonumber \\
    0 & = \mathbf{T}_{1,1} \mathbf{T}_{1,2}^* + \mathbf{T}_{2,1} \mathbf{T}_{2,2}^* + {\mathbf{S}}_1^* {\mathbf{S}}_2 \\ \nonumber
    0 & = \mathbf{T}_{2,2} \mathbf{T}_{2,1}^* + \mathbf{T}_{1,2} \mathbf{T}_{1,1}^* + {\mathbf{S}}_2^* {\mathbf{S}}_1, 
  \end{align}
 and in the genus zero case we have 
  \begin{align*}  
    \mathbf{I} & = \mathbf{T}_{1,1} \mathbf{T}_{1,1}^* + \mathbf{T}_{2,1} \mathbf{T}_{2,1}^*  \\
    \mathbf{I} & = \mathbf{T}_{1,2} \mathbf{T}_{1,2}^* + \mathbf{T}_{2,2} \mathbf{T}_{2,2}^*  \\
    0 & = \mathbf{T}_{1,1} \mathbf{T}_{1,2}^* + \mathbf{T}_{2,1} \mathbf{T}_{2,2}^*  \\ 
    0 & = \mathbf{T}_{2,2} \mathbf{T}_{2,1}^* + \mathbf{T}_{1,2} \mathbf{T}_{1,1}^*.
  \end{align*}
  
 Finally we have the following identity:
  \begin{theorem}[Quadratic adjoint identities, part III]  \label{th:general_adjoint_identity_double_final}  If $\mathscr{R}$ has non-zero genus, then
   \begin{align*}
     0 & = \mathbf{T}_{1,1} \overline{\mathbf{S}}_1^* + \mathbf{T}_{2,1} \overline{\mathbf{S}}_2^* \\
     0 & = \mathbf{T}_{1,2} \overline{\mathbf{S}}_1^* + \mathbf{T}_{2,2} \overline{\mathbf{S}}_2^*.
   \end{align*}
  \end{theorem}
  \begin{proof}
   First, recall that $\mathbf{S}_k^* = \mathbf{R}_k$ by Theorem \ref{th:adjoint_identities}. Thus the identities are equivalent to showing that  
   \[  \iint_{\mathscr{R}} L_{\mathscr{R}}(z,\zeta) \wedge_\zeta \overline{\alpha(\zeta)} = 0   \]
   for $z \in \riem_k$, $k=1,2$ and all $\alpha\in \mathcal{A}(\mathscr{R})$.  
  
   Fix $z \in \riem_k$.  Let $\gamma_\varepsilon$ be a curve given by 
   $|z-\zeta|=\varepsilon$ in a local coordinate chart, with orientation chosen to be positive with respect to $z$.  
  Stokes' theorem yields that the principal value integral is given by 
   \[   -\frac{1}{\pi i} \lim_{\varepsilon \searrow 0}  \int_{\gamma_\varepsilon}  \partial_z \mathscr{G}(z;\zeta, q) \overline{\alpha(\zeta)}  \]
   where $\mathscr{G}$ is Green's function of $\mathscr{R}$.  
   By Lemma \ref{le:limiting_circle_Schiffer_identity} this is zero.  
  \end{proof}
  
Taking complex conjugates and using \ref{th:adjoint_identities} we also obtain 
  \begin{align} \label{eq:general_adjoint_identity_double_final}
   0 & = \mathbf{S}_1 \mathbf{T}_{1,1} + \mathbf{S}_{2} \mathbf{T}_{1,2} \nonumber \\
   0 & = \mathbf{S}_{1} \mathbf{T}_{2,1} + \mathbf{S}_{2} \mathbf{T}_{2,2}.
  \end{align}
\end{subsection}
\begin{subsection}{The Schiffer operators on harmonic measures}  \label{se:Schiffer_on_harmonic_measures}
A relationship between the Schiffer operators and the harmonic measure is established in the following result:
\begin{theorem} \label{th:T_and_S_on_harmonic_measures} 
 Let  $d\omega$ be a harmonic measure on $\riem_1$.  Then
 \[  \mathbf{T}_{1,1} \overline{\partial} \omega = - \partial \omega + \mathbf{R}_1 \mathbf{S}_1 {\partial} \omega      \]
 and
     \[  \mathbf{T}_{1,2} \overline{\partial} \omega = \mathbf{R}_2 \mathbf{S}_1 \partial \omega. \]
\end{theorem}
\begin{proof}  Assume that $\omega=1$ on one boundary of $\riem_1$ and $0$ on the others.  It is enough to prove the claim for such $\omega$.  {We will need to use a particular set of limiting curves in computing the boundary integrals, for which the computation simplifies.}
{Let $\Gamma_\varepsilon$ denote the union of the level sets $\omega = \varepsilon$, $\omega= 1-\varepsilon$}. Using \cite[Lemma 3.17]{Schippers_Staubach_scattering_I},  for $\varepsilon$ sufficiently small this consists of $n$ disjoint curves each homotopic to $\partial_k \riem_1$ for a particular $k$.  Also, let $\gamma_r$ be as in Lemma \ref{le:limiting_circle_Schiffer_identity}.  
 We then have that, fixing $q \in \riem_2$, 
 \begin{align*}
    \mathbf{T}_{1,1} \overline{\partial} \omega & = \frac{1}{ \pi i } \iint_{\riem_1} 
    \partial_z \partial_w \mathscr{G}(w;z,q) \wedge_w \overline{\partial} \omega (w) \\
    & = \lim_{\varepsilon \searrow 0} \frac{1}{ \pi i } \int_{\Gamma_\varepsilon,w} \partial_z \mathscr{G}(w;z,q)\, \overline{\partial} \omega(w) - 
    \lim_{r \searrow 0}\frac{1}{ \pi i } \int_{\gamma_r,w} \partial_z \mathscr{G}(w;z,q)\, \overline{\partial} \omega(w).
 \end{align*}
 Applying Lemma \ref{le:limiting_circle_Schiffer_identity} twice and using the fact that $\overline{\partial}\omega(w)=- \partial \omega(w)$ on the level curves $\Gamma_\varepsilon$, we see that 
 \begin{align*}
     \mathbf{T}_{1,1} \overline{\partial} \omega(z) & = - \lim_{\varepsilon \searrow 0} \frac{1}{ \pi i } \int_{\Gamma_\varepsilon,w} \partial_z \mathscr{G}(w;z,q)\, {\partial}\, \omega(w) \\
     & = - \lim_{\varepsilon \searrow 0} \frac{1}{ \pi i } \int_{\Gamma_\varepsilon,w} \partial_z \mathscr{G}(w;z,q)\, {\partial} \omega(w) + 
    \lim_{r \searrow 0}\frac{1}{ \pi i } \int_{\gamma_r,w} \partial_z \mathscr{G}(w;z,q)\, {\partial} \omega(w) \\
     &  \ \ \ \ \ \    - 
    \lim_{r \searrow 0}\frac{1}{ \pi i } \int_{\gamma_r,w} \partial_z \mathscr{G}(w;z,q)\, {\partial} \omega(w) \\
    & = - \frac{1}{ \pi i } \iint_{\riem_1} 
    \partial_z \overline{\partial}_w \mathscr{G}(w;z,q) \wedge_w  {\partial} \omega (w) - \partial \omega(z)  \\
    & = \mathbf{R}_1 \mathbf{S}_1 \partial \omega (z) - \partial \omega (z).
 \end{align*}
 The proof of the second claim is similar, except with the integrals over $\gamma_r$ removed, since for $z \in \riem_2$ there are no singularities in $\riem_1$.
\end{proof}

\begin{definition}
 We say that $u \in \mathcal{A}_{\text{harm}}(\mathscr{R})$ is \emph{piecewise exact} if 
\[  \mathbf{R}_k^{\mathrm{h}} u  \in \mathcal{A}^{\mathrm{e}}_{\mathrm{harm}}(\riem_k)   \]
for $k=1,2$, where $\mathbf{R}_k^{\mathrm{h}}$ is as in Definition \ref{def: restriction ops}.  We denote the space of piecewise exact harmonic forms on $\mathscr{R}$ by $\gls{peform}(\mathscr{R})$. 
\end{definition}

{\begin{definition}\label{def:exact overfare}
  If $\riem_2$ is connected, the \emph{exact overfare} can be defined as follows. Given 
 the spaces $\mathcal{A}^\mathrm{e}$ of Definition \ref{def: exact holo and harm forms} we define 
 \[  \gls{exacto}: \mathcal{A}^{\mathrm{e}}(\riem_2) \rightarrow \mathcal{A}^{\mathrm{e}}(\riem_1) \]
 to be the unique operator satisfying 
 \begin{equation}\label{charac of exact overfare}
   \mathbf{O}^{\mathrm{e}}_{2,1} d = d \mathbf{O}_{2,1}.    
 \end{equation}      
 If $\riem_1$ is connected we may define $\mathbf{O}^\mathrm{e}_{1,2}$ in the same way.
\end{definition}}
\begin{corollary} \label{co:overfare_piecewise_exact} 
 For any harmonic measure $d\omega \in \mathcal{A}_{\mathrm{harm}}(\riem_1)$, $\mathbf{S}_1 \partial \omega + \overline{\mathbf{S}}_1 \overline{\partial} \omega = \mathbf{S}_1^{\mathrm{h}} d\omega \in \mathcal{A}_{\mathrm{harm}}^{\mathrm{pe}}(\mathscr{R})$.  
 Furthermore, if $\riem_2$ is connected, then
 \[  \mathbf{O}^{\mathrm{e}}_{2,1} \left(  \mathbf{R}^{\mathrm{h}}_2 \mathbf{S}_1 d \omega \right) =   \mathbf{R}^\mathrm{h}_1 \mathbf{S}^\mathrm{h}_1  d \omega + d\omega.   \]
\end{corollary}
\begin{proof}
 Since all operators involved are complex linear, it is enough to prove this for real harmonic measures $d\omega \in \mathcal{A}_{\mathrm{hm}}(\riem_1)$.  For such harmonic measures, by Theorems \ref{th:jump_derivatives} and \ref{th:T_and_S_on_harmonic_measures},
 \[     d\mathbf{J}^q_{1,1}\, \omega  = \partial \omega + \mathbf{T}_{1,1} \overline{\partial} \omega + \overline{\mathbf{R}}_1 \overline{\mathbf{S}}_1 \overline{\partial} \omega = \mathbf{R}_1 \mathbf{S}_1 \partial \omega + \overline{\mathbf{R}}_1 \overline{\mathbf{S}}_1 \overline{\partial} \omega  \]
 and 
 \[  d \mathbf{J}^q_{1,2}\, \omega = \mathbf{T}_{1,2} \overline{\partial} \omega + \overline{\mathbf{R}}_2 \overline{\mathbf{S}}_1 \overline{\partial} \omega = \mathbf{R}_2 \mathbf{S}_1 \partial \omega + \overline{\mathbf{R}}_2 \overline{\mathbf{S}}_1 \overline{\partial} \omega  \]
 which proves the first claim. 
 
 Now if $\riem_2$ is connected, then $\mathbf{O}_{2,1}^{\mathrm{e}}$ is well-defined (if not, it's only defined up to addition of a harmonic measure on $\riem_1$.) 
 By the transmitted jump formula (Theorem \ref{th:Overfare_with_correction_functions}), 
 \begin{align*}
     \dot{\mathbf{O}}_{2,1} \dot{\mathbf{J}}_{1,2} \, \dot{\omega} = \dot{\mathbf{J}}_{1,1} \, \dot{\omega} - \dot{\omega}.
 \end{align*}
 Taking $d$ of both sides and applying the previous two equations and \eqref{charac of exact overfare}, completes the proof of the theorem. 
\end{proof}
\begin{remark}  \label{re:nonself-overfare}
 Note that this shows that an element of $\mathcal{A}^{\mathrm{pe}}_{\text{harm}}$ need not be its own overfare.
\end{remark}

Another identity for the harmonic measures and $\mathbf{S}$ operator is the following. 

\begin{theorem} \label{th:S_on_harmonic_measures}
 Let $\omega_1$ be a harmonic function on $\riem_1$ which is constant on each boundary curve.  Let $\omega_2 = \mathbf{O}_{1,2}\, \omega_1$.  If either $\riem_1$ or $\riem_2$ is connected, then 
 \[ \mathbf{S}_1^{\mathrm{h}} \, d\omega_1 = -\mathbf{S}_2^{\mathrm{h}} \, d \omega_2. \]
\end{theorem}
\begin{proof} Assume that $\riem_1$ is connected. 
 By Theorem \ref{th:J_same_both_sides}
 \[ \dot{\mathbf{J}}_{1,2}\, \dot{\omega}_1 = - \dot{\mathbf{J}}_{2,2} \dot{\mathbf{O}}_{1,2}\, \dot{\omega}_1 = - \dot{\mathbf{J}}_{2,2}\, \dot{\omega}_2.    \]
 Applying $d$ to both sides we get
 \begin{align*}
   \mathbf{R}_2 \mathbf{S}_1  \partial \omega_1 +  \overline{\mathbf{R}}_2 \overline{\mathbf{S}}_1 \overline{\partial}   \omega_1 & = \mathbf{T}_{1,2}\, \overline{\partial} \omega_1 +  {\mathbf{R}}_2  {\mathbf{S}}_1 \partial \omega_1 =
  d \mathbf{J}_{1,2}\, \omega_1 = - d \mathbf{J}_{2,2}\, \omega_2  \\
  & = - \partial \omega_2 - \mathbf{T}_{22}\, \overline{\partial} \omega_2 - \overline{\mathbf{R}}_{2} \overline{\mathbf{S}}_2\, \overline{\partial} \omega_2 \\ &  =  -{\mathbf{R}}_{2} {\mathbf{S}}_2 {\partial} \omega_2 - \overline{\mathbf{R}}_{2} \overline{\mathbf{S}}_2 \, \overline{\partial} \omega_2
 \end{align*}
 where we have used Theorem \ref{th:T_and_S_on_harmonic_measures}.  Therefore $\mathbf{R}_2^\mathrm{h} \mathbf{S}_1^\mathrm{h}\, d \omega_1 = - \mathbf{R}_2^\mathrm{h} \mathbf{S}_2^\mathrm{h} \, \omega_2$ which proves the claim by analytic continuation to $\mathscr{R}$. 
 
 If, on the other hand, $\riem_2$ is connected, we have $\omega_1=\mathbf{O}_{2,1} \omega_2$. One can now repeat the proof with the roles of $1$ and $2$ switched.
\end{proof}

This immediately leads to a characterization of which harmonic measures lie in the kernel of $\mathbf{S}$ and $\mathbf{T}_{12}$. 
{
\begin{corollary} \label{co:bridgeworthy_TFAE}
Let $\omega_1 \in \mathcal{D}_{\mathrm{harm}}(\riem_1)$ have constant boundary values. 
The following statements are equivalent. 
\begin{enumerate}[label=(\arabic*),font=\upshape]
    \item $\mathbf{S}_1\, \partial \omega_1 =0$; 
    \item $\mathbf{T}_{1,2}\, \overline{\partial} \omega_1 =0$;
    \item $\mathbf{T}_{1,1}\, \overline{\partial} \omega_1 = -\partial \omega_1$.
\end{enumerate}
If at least one of $\riem_1$ or $\riem_2$ is connected, then $(1)-(3)$ are also equivalent to each of the following: 
\begin{enumerate}[label=(\arabic*),font=\upshape] \setcounter{enumi}{3} 
    \item $\mathbf{S}_2 \, \partial\, \mathbf{O}_{1,2}\, \omega_1 =0$;
    \item $\mathbf{T}_{2,1}\, \overline{\partial}\, \mathbf{O}_{1,2}\, \omega_1 =0$;
    \item $\mathbf{T}_{2,2}\, \overline{\partial}\, \mathbf{O}_{1,2}\, \omega_1 = -\partial \, \mathbf{O}_{1,2}\, \omega_1$.
\end{enumerate}

The complex conjugates of the statements above also hold.
\end{corollary} 
\begin{proof}
  The equivalence of the first three claims follows from Theorem \ref{th:T_and_S_on_harmonic_measures}, as does the equivalence of claims (4) through (6). If one of $\riem_1$ and $\riem_2$ is connected, then by Theorem \ref{th:S_on_harmonic_measures}, the holomorphic part of $\mathbf{S}_1^{\mathrm{h}} d \omega_1$ is zero if and only if the holomorphic part of $\mathbf{S}_2^{\mathrm{h}} d \mathbf{O}_{1,2} \omega_1$ is zero. This proves the equivalence of (1) and (4).
  The remaining claim is obvious.
\end{proof} 
}

 We review some facts about harmonic one-forms on compact Riemann surfaces; see for example \cite[Chapter III]{Farkas_Kra}. We start by recalling  the standard way to define harmonic one-forms $H_C$ such that 
\begin{equation}\label{eq:H_forms_definition}
     (\alpha,\ast H_C) = \int_C \alpha 
\end{equation}   
for all $\alpha \in \mathcal{A}_{\text{harm}}(\mathscr{R})$.  
Given a curve $C$, let $\Omega$ be a strip to the left of $C$, and let $f$ a real-valued function which is $1$ on $C$, smooth on $\Omega$, and $0$ outside of $\Omega$.  Thus, there is a steady increase from $0$ to $1$ as one approaches $C$ from the left, and a jump back down to $0$ as one crosses the curve.  Then $df$ is smooth, and we have 
\begin{align*}
 (\alpha,\ast df)_{\mathscr{R}} & = \iint_{\mathscr{R}} \alpha \wedge \ast \ast \overline{df} = -\iint_{\mathscr{R}} \alpha \wedge   {df}  \\
 & = -\iint_{\Omega} \alpha \wedge   {df} = \int_C f \alpha \\
 &= \int_C \alpha.
\end{align*} 
By the Hodge theorem, there is a unique harmonic one-form $H_C$ in the same equivalence class as $df$.  Now co-exact forms are orthogonal to closed forms because 
\[  (\alpha,\ast dg )_{\mathscr{R}} ={-} \iint_{\mathscr{R}} \alpha \wedge dg =  \iint_{\mathscr{R}} d (g \alpha) =0.    \]
Thus 
\[  (\alpha,\ast H_C)_{\mathscr{R}}  = (\alpha,\ast df)_{\mathscr{R}} = \int_C \alpha. \]

Now let $\mathscr{R}$ be separated into two surfaces $\riem_1$ and $\riem_2$ by a collection of curves $\partial_k \riem_1$ as in Figure \ref{fig:boundary_curves}.  We assume that $\riem_1$ and $\riem_2$ are both connected.
\begin{figure}
     \includegraphics[width=9cm]{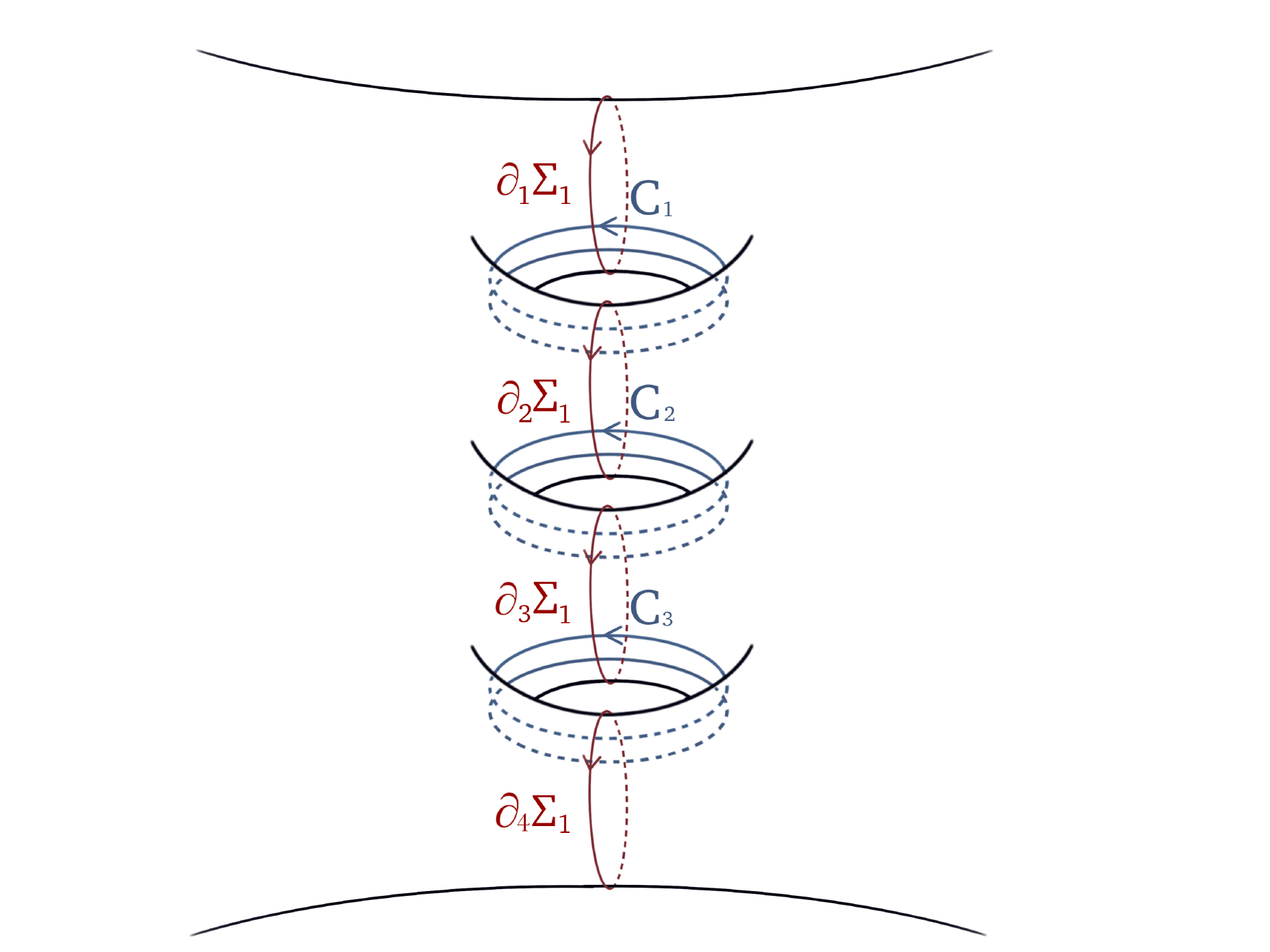}
     \caption{The one-forms $df_{k}$}
     \label{fig:boundary_curves}
 \end{figure} 
The strips defining $df_k$, $k=1,2,3$ are the horizontal strips and the boundary curves are the vertical curves. Since $H_{C_k}$ and $df_k$ are in the same cohomology class, we have (by taking analytic limiting curves approaching $\partial_m \riem_1$)
\[ \int_{\partial_m \riem_1}  H_{C_k} = \int_{\partial_m \riem_1} df_k. \]
Here we let $k=1,\ldots,n-1$ where $n$ is the number of boundary curves (in Figure \ref{fig:boundary_curves} we have $n=4$).
 
We see that
\begin{align}  \label{eq:integrals_of_Hcs} \nonumber
 \int_{\partial_1 \riem_1} H_{C_k} & = \left\{ \begin{array}{ll} -1 & k=1 \\ 0 & \text{otherwise} \end{array} \right. \\  \nonumber
\int_{\partial_2 \riem_1} H_{C_k} & = \left\{ \begin{array}{ll} 1 & k=1 \\ -1 & k=2 \\ 0 & \text{otherwise} \end{array} \right. \\ 
\vdots & \ \ \ \ \ \ \ \vdots \\ \nonumber
\int_{\partial_{n-1}  \riem_1} H_{C_k} & = \left\{ \begin{array}{ll} 1 & k={n-2} \\ -1 & k=n-1 \\ 0 & \text{otherwise} \end{array} \right. \\ \nonumber 
\int_{\partial_n \riem_1} H_{C_k} & = \left\{ \begin{array}{ll} 1 & k={n-1}  \\  0 & \text{otherwise} \end{array} \right. 
\end{align}
Furthermore, the integral around any $C_k$ or internal homology curve is zero.
\begin{remark}
 The reason for the differing first and last integrals is that we haven't included the redundant $H_{C_n}$ which corresponds to a curve traversing the outside of the entire surface.
\end{remark}
 
We have the following result:\\
\begin{theorem} \label{th:piecewise_exact_characterize}  
Assume that $\riem_1$ and $\riem_2$ are connected.   Then  
\[ \mathcal{A}_{\mathrm{harm}}^{\mathrm{pe}}(\mathscr{R})  = \{   \mathbf{S}_1 \partial \omega +   \overline{\mathbf{S}}_1 \overline{\partial} \omega : d\omega \in { \mathcal{A}_{\mathrm{harm}}(\riem_1) } \}.   \] 
\end{theorem}
\begin{proof}
 The fact that $\mathbf{S}_1 \partial \omega +   \overline{\mathbf{S}}_1 \overline{\partial} \omega$ is piecewise exact for $d\omega \in \mathcal{A}_{\mathrm{hm}}(\riem_1)$ follows from Theorem \ref{th:jump_derivatives}. We need to show that such forms span the space. 
 
 To do so we first need an identity. 
 Let $\omega_k$ be the harmonic function which is one on the $k$th boundary curve and zero on the remaining boundary curves. Set 
 \[  \beta_k = \mathbf{S}_1 \partial \omega_k +   \overline{\mathbf{S}}_1 \overline{\partial} \omega_k   \]
 We then have that by definition of $H_{C_j}$  
 \begin{align*}
   \int_{C_j} \beta_k & = (\beta_k,\ast H_{C_j})_{\mathscr{R}} \\
   & = (\mathbf{S}^{\mathrm{h}}_1 d\omega,\ast H_{C_j})_{\mathscr{R}} \\
   & = (d\omega, \mathbf{R}_1^{\mathrm{h}} \ast H_{C_j})_{\riem_1} \\
   & = \left(d\omega, \mathbf{R}_1^{\mathrm{h}}\left[ \frac{1}{2} \left(\ast H_{C_j} -i H_{C_j} \right)+  \frac{1}{2} \left(\ast H_{C_j} +i H_{C_j} \right) \right] \right)_{\riem_1}.
 \end{align*}
 So 
 \begin{align*}
   \int_{C_j} \beta_k & =  \left(  \partial \omega_k, \frac{1}{2} \mathbf{R}_1\left( \ast H_{C_j} -i H_{C_j} \right)  \right)_{\riem_1} +
   \left(  \, \overline{\partial} \omega_k, \frac{1}{2}\overline{\mathbf{R}_1}  \left( \ast H_{C_j} + i H_{C_j} \right)  \right)_{\riem_1}
   \\  & =  \left( d \omega_k, \frac{1}{2} \mathbf{R}_1\left( \ast H_{C_j} -i H_{C_j} \right)  \right)_{\riem_1} +
   \left(  \, d\omega_k, \frac{1}{2}\overline{\mathbf{R}_1}  \left( \ast H_{C_j} + i H_{C_j} \right)  \right)_{\riem_1}
   \\
   & = \left(  d \omega_k,  \mathbf{R}^{\mathrm{h}}_1 \ast H_{C_j}  \right)_{\riem_1}.
 \end{align*}
 
 So we have
 \begin{align} \label{eq:integral_beta_is_integral_H}
  \int_{C_j} \beta_k  
  & = - \iint_{\riem_1} d \omega_k \wedge H_{C_j} \nonumber \\
  & = - \int_{\partial \riem_1} \omega_k H_{C_j} \nonumber \\
  & = - \int_{\partial_k \riem_1} H_{C_j}.
 \end{align}
 It now follows immediately from \eqref{eq:integrals_of_Hcs} that the $\beta_k$ span $\mathcal{A}^{\mathrm{pe}}_{\mathrm{harm}}(\mathscr{R})$. This completes the proof.
\end{proof}

We also have the following elementary fact.
\begin{proposition} \label{pr:period_submatrix_also_positive}
 Fix a bordered Riemann surface $\riem$ of type $(g,n)$. Fix a subcollection $\gamma_1,\ldots,\gamma_m$ of the boundary curves $\{ \partial_1 \riem,\ldots, \partial_n \riem \}$. For any $c_1,\ldots,c_m \in \mathbb{C}$, there is an $\omega \in \mathcal{A}_{\mathrm{hm}}(\riem)$ whose boundary values are only non-zero on $\gamma_1,\ldots,\gamma_m$, such that 
 \[  \int_{\gamma_k} \omega = c_k.   \]
\end{proposition}
\begin{proof}
 Let $j_1,\ldots,j_m$ be the indices of the subcollection of curves; that is, $\gamma_{l} =\partial_{j_l} \riem$ for $l=1,\ldots,m$. Following a similar strategy as in the proof of Corollary \ref{co:boundary_periods_specified_starmeasure}, it is readily seen that we need to prove the existence of a solution to the system of equations 
  \[  c_k=  \int_{\gamma_k}  \sum_{l=1}^n a_l \ast d \omega_{j_l} =\sum_{l=1}^n \int_{\partial_{j_k} \riem}   a_l \ast d \omega_{j_l}= \sum_{l=1}^n\Pi_{j_k j_l} a_l .      \]
 This follows from the fact that any square submatrix of a positive-definite matrix is also positive-definite.
\end{proof}

Now we define a subclass of the harmonic measures which lie in the kernel of $\mathbf{T}_{1,2}$. 
  
 \begin{definition}  
  We say that $\omega \in \mathcal{D}_{\text{harm}}(\riem_1)$ is \emph{bridgeworthy} if
  \begin{enumerate}
   \item it is constant on each boundary curve;
   \item on any pair of boundary curves $\partial_k \riem_1$ and $\partial_m \riem_1$ that bound the same connected component of $\riem_2$, the boundary values of $\omega$ are equal.
  \end{enumerate}
  
  We say that $\alpha \in \mathcal{A}_{\mathrm{hm}}(\riem_1)$ is bridgeworthy if $\alpha = d\omega$ for some bridgeworthy harmonic function $\omega$.  Denote the collection of bridgeworthy harmonic functions by
 $\gls{Dbw}(\riem_1)$,  and the collection of bridgeworthy  harmonic one-forms by 
 $\gls{Abw}(\riem_1)$.  The same definitions apply to $\riem_2$.
 \end{definition}
 The name is meant to invoke the following geometric picture: $\omega$ is bridgeworthy if it has the same constant value on any pair of boundary curves which are connected by a ``bridge'' in $\riem_2$. 
 
 \begin{remark} If $\riem_1$ has more than one connected component, then  a bridgeworthy harmonic one-form has anti-derivatives which are not bridgeworthy. This is because one can add to a bridgeworthy harmonic function $\omega$ any function which is constant on connected components without changing $d\omega$. 
 \end{remark}

We will also need the following characterization of the kernel of $\mathbf{S}_k^{\mathrm{h}}$.
\begin{proposition}  \label{pr:kernel_S1h}   { Assume that either $\riem_1$ or $\riem_2$ is connected}.  Fix $k=1$ or $k=2$.  
 Let $d\omega_k \in \mathcal{A}_{\mathrm{hm}}(\riem_k)$.
 The following are equivalent.
 \begin{enumerate}[label=(\arabic*),font=\upshape]
     \item $\overline{\partial} \omega_k \in \overline{\partial} \mathcal{D}_{\mathrm{bw}}(\riem_k)$.
     \item $\partial \omega_k \in \partial \mathcal{D}_{\mathrm{bw}}(\riem_k)$.
     \item $\mathbf{S}_k^\mathrm{h} d\omega_k =0$.
     \item {$d\omega_k \in \mathcal{A}_{\mathrm {bw}}(\riem_k)$.}
 \end{enumerate}
\end{proposition}
\begin{proof}  We assume throughout that $k=1$. It is then necessary and  sufficient to show the equivalence for both the cases that $\riem_1$ is connected and that $\riem_2$ is connected. The case $k=2$ is obtained by symmetry. 

  We first show that (1) implies (3); assume that (1) holds. {If $\riem_2$ is connected, then $\omega_1$ is constant, so $d\omega_1=0$ and (3) follows trivially.  Now assume that $\riem_1$ is connected.}   Let $H \in \mathcal{D}_{\mathrm{bw}}(\riem_1)$ be such that $\overline{\partial} H= \overline{\partial} \omega_1$.  Then $\mathbf{O}_{1,2} H$ is constant on connected components of $\riem_2$, so by Theorem \ref{th:J_same_both_sides}, $\dot{\mathbf{J}}_{1,2} H = - \dot{\mathbf{J}}_{2,2} \dot{\mathbf{O}}_{1,2} \dot{H}$ =0.  So by Theorem \ref{th:jump_derivatives}
  \[  \mathbf{T}_{1,2} \overline{\partial} H + \overline{\mathbf{R}}_1 \overline{\mathbf{S}}_1 \overline{\partial} H = d \dot{\mathbf{J}}_{1,2} \dot{H} = 0.   \]
  Since the holomorphic and anti-holomorphic parts must both be zero, we have $\mathbf{T}_{1,2} \overline{\partial} \omega_1 = 0$ and $\overline{\mathbf{R}}_1 \overline{\mathbf{S}}_1 \overline{\partial} \omega_1 =0$.  The latter implies that  $\overline{\mathbf{S}}_1 \overline{\partial} \omega_1 =0$ by analytic continuation.  The former together with Corollary \ref{co:bridgeworthy_TFAE} to $\overline{\omega}_1$, implies that
  \[  \mathbf{R}_1 \mathbf{S}_1 \partial \omega_1 =  0.  \]
  Hence $\mathbf{S}_1 \partial \omega_1 =0$ and therefore $\mathbf{S}_1^\mathrm{h} d\omega_1 =0$. This shows that (1) implies (3).  A similar argument shows that (2) implies (3). 
  
 {Now assume that (3) holds, and that $\riem_2$ is connected.} We will show that both (1) and (2) hold. Since holomorphic and anti-holomorphic parts of $\mathbf{S}_1^\mathrm{h} d\omega_1$ are zero, we have $\mathbf{S}_1 \partial \omega_1 = 0$ and 
  $\overline{\mathbf{S}}_1 \overline{\partial} \omega_1=0$.  By Corollary \ref{co:bridgeworthy_TFAE}, we also have that 
\[  \mathbf{T}_{1,2} \overline{\partial} \omega_1 = 0\ \ \ \text{and} \ \ \ \overline{\mathbf{T}}_{1,2} \partial \omega_1 =0.  \]
Here, to show the left equation, we have applied the equivalence of parts (1) and (3) of Corollary \ref{co:bridgeworthy_TFAE} directly; whereas to show the right equation, we applied the conjugates of the equivalence of parts (1) and (3) to $\overline{\omega_1}$ to see that 
\[ 0 = \overline{\mathbf{S}}_1 \overline{\partial} \omega_1   =  \overline{\mathbf{S}_1 \partial \overline{\omega_1}} \ \ \Rightarrow  \ \ 0 = \overline{\mathbf{T}_{1,2} \overline{\partial} \overline{\omega_1}}  =\overline{\mathbf{T}}_{1,2} \partial \omega_1.  \]

We thus have that
\[ d \mathbf{J}^q_{1,2} \omega_1 = \mathbf{T}_{1,2} \overline{\partial} \omega_1 + \overline{\mathbf{R}}_1 \overline{\mathbf{S}}_1 \overline{\partial} \omega_1 = 0,  \]
and so $\mathbf{J}^q_{1,2} \omega_1$ is constant on connected components of $\riem_2$. Hence $H = \mathbf{O}_{2,1} \mathbf{J}_{1,2}^q \omega_1$ is bridgeworthy.  Applying Theorem \ref{th:Overfare_with_correction_functions} yields 
\[  \dot{\mathbf{J}}_{1,1} \dot{\omega}_1 - \dot{H} = \dot{\omega}_1,  \]
from which we obtain
\[  d \dot{\mathbf{J}}_{1,1} \dot{\omega}_1 - d \dot{H}  = d \dot{\omega}_1.   \]
Now Theorems \ref{th:jump_derivatives} and \ref{th:S_on_harmonic_measures} yield that
\[ d \dot{\mathbf{J}}_{1,1} \dot{\omega}_1 =\partial \dot{\omega}_1 + \mathbf{T}_{1,1} \overline{\partial} \dot{\omega}_1 + \overline{\mathbf{R}}_1 \overline{\mathbf{S}}_1 \overline{\partial} \dot{\omega}_1 = 0,  \]
and inserting this in the above equation we obtain that $\overline{\partial} \omega_1 = - \overline{\partial} H$. This proves that (3) implies (1) {in the case that $\riem_2$ is connected.} 
A similar argument, using the fact that $\overline{\mathbf{J}}_{1,2}^q \omega_1$ is constant on connected components of $\riem_2$ shows that $\partial \omega_1 = - \partial G$ 
where $G = \mathbf{O}_{2,1} \mathbf{J}_{1,2}^q \omega_1$, so (3) implies (2) {in the case that $\riem_2$ is connected.}

{Now assume that $\riem_1$ is connected and that (3) holds.  We will show that (4) holds. We have that $\mathbf{S}_1^h d\omega_1 = 0$.  Then by Theorem \ref{th:S_on_harmonic_measures} $\mathbf{S}_2^h d\mathbf{O}_{2,1} \omega_1 =0$.  Thus
\[   \mathbf{S}_2 \partial \mathbf{O}_{1,2} \omega_1 = 0 \ \ \ \text{and} \ \ \ \overline{\mathbf{S}}_2 \overline{\partial} \mathbf{O}_{1,2} \omega_1 = 0.    \]
Applying Corollary \ref{co:bridgeworthy_TFAE} parts (1)--(3) and Theorem \ref{th:jump_derivatives} we see that
\[ d {\mathbf{J}}_{2,1} {\mathbf{O}}_{1,2} \omega_1 = \mathbf{T}_{2,1} \overline{\partial} \mathbf{O}_{2,1} \omega_1 + \overline{\mathbf{R}}_1
 \overline{\mathbf{S}}_2  \overline{\partial} \mathbf{O}_{1,2} \omega_1 =0. \]
Thus $\mathbf{J}_{2,1} \mathbf{O}_{1,2}$ is locally constant.  Similarly
\[   d \mathbf{J}_{2,2} \mathbf{O}_{1,2} \omega_1 = \overline{\partial} \mathbf{O}_{1,2} \omega_1 + \mathbf{T}_{2,2} \overline{\partial} \mathbf{O}_{1,2} \omega_1 + \overline{\mathbf{R}}_2 \overline{\mathbf{S}}_2 \overline{\partial} \mathbf{O}_{1.2} \omega_1 =0    \]
so $\mathbf{J}_{2,2} \mathbf{O}_{1,2} \omega$ is constant.  Thus by Theorem 
\ref{th:Overfare_with_correction_functions} part (2) we see that
\[ 0= - \dot{\mathbf{O}}_{1,2} \dot{\mathbf{J}}_{2,1} \dot{\mathbf{O}}_{1,2} \dot{\omega}_1 + \dot{\mathbf{J}}_{2,2} \dot{\mathbf{O}}_{1,2} \dot{\omega}_1 = \dot{\mathbf{O}}_{1,2} \dot{\omega}_1. \]
Thus $\mathbf{O}_{1,2} \omega$ is constant, that is, $\omega_1$ is bridgeworthy.  Thus (3) implies (4) in the case that $\riem_1$ is connected. 
 It is obvious that (4) implies (1) and (4) implies (2) independently of the connectivity assumption.  
 
 In summary, we have shown the equivalence of (1), (2), and (3), and furthermore that (4) implies (1) and (2).  It remains to show that (1) implies (4).   }
Assuming that (1) holds, we have that $\omega_1 = \tilde{\omega_1} + h$ where $\tilde{\omega_1}$ is bridgeworthy and $h$ is holomorphic on $\riem_1$.  So $h=\omega_1-\tilde{\omega_1}$ has constant boundary values on $\partial \riem_1$. Fix a connected component $\riem_1^0$ and treat it as a subset of its double, so that the boundary is an analytic curve. We have that by Schwarz reflection $h$ extends to a holomorphic function on a neighbourhood of $\partial \riem_1^0$. Since $h$ is constant on the boundary it is constant on $\riem_1^0$.
{We have shown that $\omega_1 = \tilde{\omega_1} + c$ where $c$ is constant on connected components and $\tilde{\omega_1}$ is bridgeworthy; thus $d\omega_1 = d \tilde{\omega_1} \in \mathcal{A}_{\mathrm bw}(\riem_1)$.
}      
\end{proof}

\end{subsection}
\end{section} 
\begin{section}{Schiffer operators: cohomology and index theorems} \label{se:index_cohomology}
\begin{subsection}{Schiffer operators and cohomology}  \label{se:Schiffer_cohomology}

 In this section, we investigate the effect of the Schiffer operators with the cohomology classes of the one-forms on $\riem_1$ and $\riem_2$. The main tool to do so is the overfared jump formula. The results of this section will be used in the next two sections in determining their kernels, cokernels, and Fredholm indices.   

 We begin by characterizing the kernels and images of the restriction operator $\mathbf{R}_k$ and its adjoint $\mathbf{S}_k$.  

 \begin{theorem}  \label{th:S_kernel_range}   The Schiffer operators $\mathbf{R}_k$ and $\mathbf{S}_k$ satisfy
  \begin{align*} 
    \mathrm{Ker}(\mathbf{R}_k) & = \{0\} \\
    \mathrm{Im}(\mathbf{S}_k) & = \mathcal{A}(\mathscr{R}) \\
    \mathrm{Ker}(\mathbf{S}_k) & = [\mathbf{R}_k \mathcal{A}(\mathscr{R}) ]^\perp 
  \end{align*}
  for $k=1,2$.  The image of $\mathbf{R}_k$ is a $g$-dimensional subspace. 
  The corresponding statements hold for the complex conjugates.
 \end{theorem}
 \begin{proof} 
  The first statement is proven using analytic continuation.  The second statement follows from the first, and {Theorem \ref{th:adjoint_identities} which in turn yields that $\text{Im}(\mathbf{S}_k) = [\text{Ker} \, \mathbf{R}_k]^\perp$.}  The remaining statements are elementary.
 \end{proof}
 
 This yields the following:
 \begin{theorem}   \label{th:S_an_isomorphism} We have
  \begin{align*}
      \mathbf{S}_k \mathbf{R}_k : \mathcal{A}(\mathscr{R}) & \rightarrow \mathcal{A}(\mathscr{R}) \\
      \overline{\mathbf{S}}_k \overline{\mathbf{R}}_k : \overline{\mathcal{A}(\mathscr{R})} & \rightarrow \overline{\mathcal{A}(\mathscr{R})}
  \end{align*}
  are isomorphisms.
 \end{theorem}
 \begin{proof} It is enough to prove the first claim. 
  By Theorem \ref{th:S_kernel_range}, we have that $\mathbf{S}_k$ is surjective.  Thus for any $y \in \mathcal{A}(\mathscr{R})$ there is a $u \in \mathcal{A}(\riem_k)$ such that $\mathbf{S}_k u = y$.  Writing $u=v+w$ in terms of the orthogonal decomposition $\mathcal{A}(\riem_k) = [\mathbf{R}_k\mathcal{A}(\mathscr{R})]^\perp \oplus [\mathbf{R}_k \mathcal{A}(\mathscr{R})]$, since by Theorem \ref{th:S_kernel_range} we also have that $\mathbf{S}_k v=0$, we see that 
  \[ y = \mathbf{S}_k u = \mathbf{S}_k w   \]
  so $y \in \mathrm{Im}(\mathbf{S}_k \mathbf{R}_k),$ and thus $\mathbf{S}_k \mathbf{R}_k$ is surjective. Now since $\text{Ker}(\mathbf{S}_k \mathbf{R}_k )=(\mathrm{Im} (\mathbf{S}_k \mathbf{R}_k ))^{\perp}= \mathcal{A}(\mathscr{R}) ^{\perp}=0 $, we also have that $\mathbf{S}_k \mathbf{R}_k$ is injective.
  
 \end{proof}
  In what follows, we will apply the identities of Section \ref{se:Schiffer_Cauchy} to investigate how the Schiffer operators affect the cohomology classes of the one-forms to which they are applied.  The spaces
  \[  [\mathbf{R}_k \mathcal{A}(\mathscr{R})]^\perp = \{   {\alpha} \in \mathcal{A}(\riem_k) : (\alpha,{R}_k \beta) =0 \ \ \ \  \forall \beta \in \mathcal{A}(\mathscr{R}) \}= 
  [\mathbf{R}_k \mathcal{A}(\mathscr{R})]^\perp   \]
  and their complex conjugates
 \[   [\overline{\mathbf{R}}_k \overline{\mathcal{A}(\mathscr{R})}]^\perp  
    =  \{   \overline{\alpha} \in \overline{\mathcal{A}(\riem_k)} : (\overline{\alpha},{R}_k \overline{\beta}) =0 \ \ \ \  \forall \overline{\beta} \in \overline{\mathcal{A}(\mathscr{R})} \}  \]
  introduced in the proof above will play an important role.    
  \begin{remark}
    Throughout,  $[\mathbf{R}_k \mathcal{A}(\mathscr{R})]^\perp$ will always refer to the orthogonal complement in $\mathcal{A}(\riem_k)$ rather that in $\mathcal{A}_{\mathrm{harm}}(\riem_k)$, and similarly for $[\overline{\mathbf{R}}_k \overline{\mathcal{A}(\mathscr{R})}]^\perp$.  
  \end{remark}

     By a capped surface, we mean the special case that $\riem_1$ consists of $n$ simply-connected domains.   We say that $\riem_2$ is capped by $\riem_1$. 
    
    Recalling  $\mathbf{O}^\mathrm{e}_{2,1}$ from {Definition \ref{def:exact overfare}} and the projection operator \[  \overline{\mathbf{P}}_{1}: \mathcal{A}_{\mathrm{harm}}(\riem_1) \rightarrow \overline{\mathcal{A}(\riem_1)}, \] we have the following theorem of   M. Shirazi \cite{Shirazi_thesis,Schippers_Shirazi_Staubach}. 
  \begin{theorem} \label{th:Mohammad_isomorphism}
    Assume that $\riem_2$ is capped by $\riem_1$.  Then 
   \[\mathbf{T}_{1,2}([\overline{\mathbf{R}}_1 \overline{\mathcal{A}(\mathscr{R})}]^\perp  ) = 
   \mathcal{A}^\mathrm{e}(\riem_2) \]
   and 
   $\mathbf{T}_{1,2}:[\overline{\mathbf{R}}_1 \overline{\mathcal{A}(\mathscr{R})}]^\perp   \rightarrow \mathcal{A}^\mathrm{e}(\riem_2)$ is an isomorphism with inverse $-\overline{\mathbf{P}}_1 \mathbf{O}^\mathrm{e}_{2,1}$. In particular, 
   \[ \mathrm{Im}(\overline{\mathbf{P}}_1\mathbf{O}_{2,1}^{\mathrm e}) = [\overline{\mathbf{R}}_1 \overline{\mathcal{A}(\mathscr{R})}]^\perp.  \]
  \end{theorem}
  \begin{proof}
   We show that $\mathbf{T}_{1,2}([\overline{\mathbf{R}}_1 \overline{\mathcal{A}(\mathscr{R})}]^\perp) \subseteq  
   \mathcal{A}^\mathrm{e}(\riem_2)$.  Since each connected component of $\riem_1$ is simply connected, for any $\overline{\alpha} \in [\overline{\mathbf{R}}_1 \overline{\mathcal{A}(\mathscr{R})}]^\perp$, there is an $H \in \overline{\mathcal{D}(\riem_1)}$ such that $\overline{\alpha} =\overline{\partial} H$. Since by Theorem \ref{th:S_kernel_range} $\overline{\mathbf{S}}_1 \overline{\alpha} =0$, Theorem   \ref{th:jump_derivatives} yields that
   \[ \mathbf{T}_{1,2} \overline{\alpha} = \mathbf{T}_{1,2} \overline{\alpha}  + \overline{\mathbf{S}}_1 \overline{\alpha} = d \mathbf{J}^q_{1,2} H \in \mathcal{A}^\mathrm{e}(\riem_2).   \]
   Note that the computation is valid for any fixed value of $q$.
   
   To show that it is onto, let ${\beta} \in \mathcal{A}^\mathrm{e}(\riem_2)$, so that there is some $h \in \mathcal{D}(\riem_2)$ such that $\partial h = \beta$. 
   Setting $H= - \mathbf{O}_{2,1} h$, we have by Theorem \ref{th:J_same_both_sides} 
   \[  \dot{\mathbf{J}}_{1,2} \dot{\mathbf{O}}_{1,2} H = - \dot{\mathbf{J}}_{1,2} \dot{\mathbf{O}}_{2,1} h = + \dot{\mathbf{J}}_{2,2}  h = \dot{h}. \]
    where we have used Theorem \ref{th:jump_on_holomorphic}.  
   Since $\beta = \partial h = dh$, Theorem
   \ref{th:jump_derivatives} yields
   \[  \beta = d \dot{\mathbf{J}}_{1,2} \dot{H} = \mathbf{T}_{1,2} \overline{\partial} H + \overline{\mathbf{R}}_2 \overline{\mathbf{S}}_1 \overline{\partial} H.  \]
   Since the left hand side is holomorphic, $\overline{\mathbf{R}}_2 \overline{\mathbf{S}}_1 \overline{\partial} H=0$, so by Theorem \ref{th:S_kernel_range} $\overline{\partial} H \in
  [\overline{\mathbf{R}}_1 \overline{\mathcal{A}(\mathscr{R})}]^\perp$.
  
  Next we show that $\mathbf{T}_{1,2}$ is injective on $[\overline{\mathbf{R}}_1 \overline{\mathcal{A}(\mathscr{R})}]^\perp$. Let $\overline{\alpha}  \in [\overline{\mathbf{R}}_1 \overline{\mathcal{A}(\mathscr{R})}]^\perp$. Again since the components of $\riem_1$ are simply connected, we can assume that $\overline{\alpha} = \overline{\partial} H$ for some $H \in \overline{\mathcal{D}(\riem_1)}$. By Theorem
  \ref{th:Overfare_with_correction_functions} 
  \[   \dot{\mathbf{O}}_{2,1} \dot{\mathbf{J}}_{12}  \dot{H} = \dot{\mathbf{J}}_{1,1} \dot{H} - \dot{H}
         \]
  so differentiating and applying Theorems \ref{th:jump_derivatives}
  and \ref{th:S_kernel_range} we obtain
  \[  \mathbf{O}^e_{1,2} \mathbf{T}_{1,2} \overline{\alpha}  = d \mathbf{J}^q_{1,1} H = \partial H + \mathbf{T}_{1,1} \overline{\partial} H - dH = -\overline{\alpha} + \mathbf{T}_{1,1} \overline{\alpha}.   \]
  Thus $-\overline{\mathbf{P}}_1\mathbf{O}^\mathrm{e}_{2,1}$ is a left inverse for the restriction of $\mathbf{T}_{1,2}$ to $[\overline{\mathbf{R}}_1 \overline{\mathcal{A}(\mathscr{R})}]^\perp$. This proves injectivity, and since we already have surjectivity and boundedness, the restriction of $\mathbf{T}_{1,2}$ is invertible with inverse $-\overline{\mathbf{P}}_1\mathbf{O}^\mathrm{e}_{2,1}$ as claimed.
  \end{proof}
 
 We will improve and extend this theorem in different ways below. 
 The following corollary and lemma allows us to make use of the jump formula to examine cohomology classes.  
 
\begin{corollary} \label{co:specify_internal_periods}
 Let $\riem$ be a Riemann surface of type $(g,n)$ with internal homology basis $\{\gamma_1,\ldots,\gamma_{2g}\}$. For any constants $\lambda_1,\ldots,\lambda_{2g} \in \mathbb{C}$ there is an $\alpha \in \mathcal{A}(\riem)$ such that
 \[  \int_{\gamma_k} \alpha =  \lambda_k,  \ \ \ \ \ k=1,\ldots,2g.    \]
 The same claim obviously holds for $\overline{\mathcal{A}(\riem)}$.  
\end{corollary}
\begin{proof}  Sew on caps to $\riem$ to obtain a compact surface $\mathscr{R}$ of genus $g$, where $\riem_1$ are the caps and $\riem_2=\riem$.  So we prove the claim for $\riem= \riem_2$.  

 By the Hodge theorem applied to $\mathscr{R}$, there is a $\zeta = \xi + \overline{\eta} \in \mathcal{A}_{\text{harm}}(\mathscr{R})$ such that
 \[   \int_{\gamma_k} \zeta = \lambda_k      \]
 for $k=1,\ldots,n$.  Now since $\mathbf{S}_1 \mathbf{R}_1:\mathcal{A}(\mathscr{R}) \rightarrow \mathcal{A}(\mathscr{R})$ is an isomorphism by Theorem \ref{th:S_kernel_range},   there is a $\overline{\sigma} \in \overline{\mathcal{A}(\mathscr{R})}$ such that $\overline{\mathbf{S}}_1 \overline{\mathbf{R}}_1 \overline{\sigma} = \overline{\eta}$. 
 
 Since the components of $\riem_1$ are simply connected, $\overline{\mathbf{R}}_1 \overline{\sigma}$ is exact, so there is an $H \in \overline{\mathcal{D}_{\mathrm{harm}}(\riem_1)}$ such that $\overline{\partial} H = \overline{\mathbf{R}}_1 \overline{\sigma}$.  So 
 by Theorem \ref{th:jump_derivatives} we have 
 \[ \overline{\mathbf{R}}_2 \overline{\mathbf{S}}_1 \overline{\mathbf{R}}_1 \overline{\sigma}  + \mathbf{T}_{1,2} \overline{\mathbf{R}}_1 \overline{\sigma} = d \mathbf{J}^q_{1,2} H \in \mathcal{A}^\mathrm{e}_{\text{harm}}(\riem_2)  \]
 so 
 $\mathbf{R}_2 \xi - \mathbf{T}_{1,2} \overline{\mathbf{R}}_1 \overline{\sigma}$
 has the desired periods. 
\end{proof}

 \begin{lemma} \label{le:dbar_representative}
  Let $\riem$ be an arbitrary Riemann surface of type $(g,n)$.  Given any $\overline{\alpha} \in \overline{\mathcal{A}(\riem)}$, there is an $h \in \mathcal{D}_{\mathrm{harm}}(\riem)$ such that $\overline{\partial} h = \overline{\alpha}$.  Any other such $\tilde{h}$ is such that $h-\tilde{h} \in \mathcal{D}(\riem)$.  
  
  The corresponding statement holds for $\alpha \in \mathcal{A}(\riem)$, replacing $\overline{\partial}$ with $\partial$.  
 \end{lemma}
\begin{proof} Fix $\overline{\alpha} \in \overline{\mathcal{A}(\riem)}$.
 First, we show that there is a $\beta \in \mathcal{A}(\riem)$ such that $\overline{\alpha} - \beta$ is exact. 
 By Corollary \ref{co:specify_internal_periods}, for any $\overline{\alpha} \in \overline{\mathcal{A}(\riem)}$ we may find a ${\delta} \in {\mathcal{A}(\riem)}$ such that 
 \[  \int_{\gamma_k} ( \overline{\alpha} - {\delta} ) = 0    \]
 for $k=1,\ldots,2g$.  So it is enough to show that for any constants $\mu_1,\ldots,\mu_n$ such that $\mu_1 + \cdots + \mu_n = 0$  there is a ${\nu} \in {\mathcal{A}(\riem_2)}$ such that 
 \begin{equation} \label{eq:rep_lemma_temp}
    \int_{\partial_k \riem} \nu = \mu_k    
 \end{equation}
 for $k=1,\ldots,n$.  By Corollary \ref{co:boundary_periods_specified_starmeasure} there is a harmonic measure $d\omega \in \mathcal{A}_{\text{hm}}(\riem)$ such that $\ast d \omega$ satisfies \eqref{eq:rep_lemma_temp}.  Setting
 \[   \nu = \ast d \omega + i d\omega \in {\mathcal{A}(\riem)},      \]
 since $d\omega$ is exact, $\nu$ has the same periods as $\ast d \omega$. Setting $\beta = \delta + \nu$ proves the claim. 
 
 So let $\beta$ be such that $\overline{\alpha} - \beta$ is exact.  Letting $h$ be such that $d h= \overline{\alpha} - \beta$ we then have that $\overline{\partial} h=\overline{\alpha}$ as claimed. If $\overline{\partial} \tilde{h} = \overline{\alpha}$ then $\overline{\partial} (h-\tilde{h}) =0$ so $h-\tilde{h} \in \mathcal{D}(\riem)$.  
\end{proof}

 \begin{theorem} \label{th:Schiffer_cohomology}
  For any $\overline{\alpha} \in \overline{\mathcal{A}(\riem_1)}$,
  \begin{enumerate}
      \item  [$(1)$] $\mathbf{T}_{1,2} \overline{\alpha}+\overline{\mathbf{R}}_2 \overline{\mathbf{S}}_1 \overline{\alpha}$ is exact on $\riem_2$; and 
      \item [$(2)$] $-\overline{\alpha} + \mathbf{T}_{1,1} \overline{\alpha}+\overline{\mathbf{R}}_1 \overline{\mathbf{S}}_1 \overline{\alpha}$ is exact on $\riem_1$.
  \end{enumerate}
  If $\mathscr{R}$ has genus zero, then $\mathbf{T}_{1,2} \overline{\alpha}$ and $-\overline{\alpha} + \mathbf{T}_{1,2} \overline{\alpha}$ are exact. 
  
  The same statements apply to the complex conjugates, and all statements hold with $1$ and $2$ interchanged.  
 \end{theorem}
 \begin{proof}
  By Lemma \ref{le:dbar_representative} there is an $h \in \mathcal{D}_{\mathrm{harm}}(\riem_1)$ such that $\overline{\partial} h = \overline{\alpha}$. 
  Also, Theorem \ref{th:jump_derivatives} yields that
  \begin{equation} \label{eq:jump_cohomology_identity}
    d \mathbf{J}^q_1 h = \left\{ \begin{array}{ll}  \partial h + \mathbf{T}_{1,1} \overline{\partial} h + \overline{\mathbf{R}}_1 \overline{\mathbf{S}}_1   \overline{\partial} h  & \riem_1 \\
    \mathbf{T}_{1,2} \overline{\partial} h + \overline{\mathbf{R}}_2 
    \overline{\mathbf{S}}_1 \overline{\partial} h & \riem_2. 
    \end{array} \right.         
  \end{equation}
  Now using $dh= \partial h + \overline{\partial} h = \partial h + \overline{\alpha}$ we see that 
  \begin{equation} \label{eq:jump_cohomology_identity_11}
   d \mathbf{J}^q_{1,1} h =  dh - \overline{\alpha} + \mathbf{T}_{1,1} \overline{\alpha}  + \overline{\mathbf{R}}_1 \overline{\mathbf{S}}_1   \overline{\alpha}   
  \end{equation}
  and 
  \begin{equation} \label{eq:jump_cohomology_identity_12} 
       d \mathbf{J}^q_{1,2} h =  \mathbf{T}_{1,2} \overline{\alpha} + \overline{\mathbf{R}}_2 
    \overline{\mathbf{S}}_1 \overline{\alpha}. 
  \end{equation}
  This proves claims (1) and (2).  If $\mathscr{R}$ has genus zero, then the third claim follows from the fact that $K_\mathscr{R}=0$ (see Example \ref{ex:sphere_kernels}).  
  The remaining claims are obvious.
 \end{proof}
 
 This simple fact is surprisingly illuminating.  We list two immediate corollaries. 
 \begin{corollary}  \label{co:one_plus_T_periods}
  Let $\overline{\alpha} \in \overline{A(\riem_1)}$.  
  For any curve $c$ in $\riem_1$, 
  \begin{align*}
  \int_{c} (\overline{\alpha} - \mathbf{T}_{1,1} \overline{\alpha}) & =  \int_c \overline{\mathbf{S}}_1 \overline{\alpha} \\ & =  \left<  \overline{\alpha}, \ast \overline{\mathbf{R}}_1 H_c  \right>_{\riem_1} \\ & =  \int_c \left[ \mathbf{I} - \mathbf{T}_{1,1}^* \mathbf{T}_{1,1} - \mathbf{T}_{1,2}^* \mathbf{T}_{1,2} \right] \overline{\alpha} 
  \end{align*}
  where $H_c$ is associated to $c$ by \emph{(\ref{eq:H_forms_definition})}.  The same formulas hold with $1$ and $2$ interchanged, as do the complex conjugates.  
 \end{corollary} 
 \begin{proof}
  The first equality follows directly from Theorem \ref{th:Schiffer_cohomology}.  The second equality follows from the definition of $H_c$ and Theorem \ref{th:adjoint_identities}, observing that $\ast$ commutes with $\overline{\mathbf{R}}_1$.  The final equality follows from the identity $\mathbf{T}_{1,1}^* \mathbf{T}_{1,1} + \mathbf{T}_{1,2}^* \mathbf{T}_{1,2} = \mathbf{I} - \overline{\mathbf{R}}_1 \overline{\mathbf{S}}_1$ given in Theorem \ref{th:general_adjoint_double_identities}.  
 \end{proof}

 \begin{corollary}   \label{co:image_Tonetwo_periods}
  For any curve $c$ in $\riem_2$ and $\overline{\alpha} \in 
  \overline{\mathcal{A}(\riem_2)}$
  \begin{align*}
   -\int_{c} \mathbf{T}_{1,2}   \overline{\alpha} & =  \int_c \overline{\mathbf{S}}_2 \overline{\alpha} \\
   & = \left< \overline{\alpha}, \ast \overline{\mathbf{R}}_2 H_c \right>_{\riem_2}  \\
   & = \int_c \left[ \mathbf{I} - \mathbf{T}_{2,2}^* \mathbf{T}_{2,2} - \mathbf{T}_{2,1}^* \mathbf{T}_{2,1} \right] \overline{\alpha}.
  \end{align*}
  The same statement holds with $1$ and $2$ interchanged, and with complex conjugates.
 \end{corollary}
 \begin{proof}  The proof is identical to that of Corollary \ref{co:one_plus_T_periods}, except in the last step we use the identity
   $\mathbf{T}_{2,2}^* \mathbf{T}_{2,2} + \mathbf{T}_{2,1}^* \mathbf{T}_{2,1} + \overline{\mathbf{R}}_2 \overline{\mathbf{S}}_2= \mathbf{I}$
   of Theorem \ref{th:general_adjoint_double_identities}.  
 \end{proof}

   If $\riem_2$ is connected, recall Definition \ref{def:exact overfare}: for 
 $\alpha \in \mathcal{A}^{\mathrm e}_{\text{harm}}(\riem_2)$, we define
 \[   \mathbf{O}^{\mathrm{e}}_{2,1} \alpha= d \mathbf{O}_{2,1} h      \]
 for $dh = \alpha$.  
  \begin{proposition} \label{prop:exact_transmitted_jump} Assume that $\riem_2$ is connected.  
  For $\overline{\alpha} \in \overline{\mathcal{A}(\riem_1)}$ we have
  \[  \mathbf{O}^{\mathrm{e}}_{2,1} \left[ \mathbf{T}_{1,2} \overline{\alpha} + \overline{\mathbf{R}}_{2} \overline{\mathbf{S}}_1 \overline{\alpha} \right] = -\overline{\alpha} + \mathbf{T}_{1,1} \overline{\alpha} + \overline{\mathbf{R}}_{1} \overline{\mathbf{S}}_1 \overline{\alpha}. \]
  In particular, if $\overline{\alpha} \in (\overline{\mathbf{R}}_1 \overline{\mathcal{A}(\mathscr{R})})^\perp$ we have
   \[  \mathbf{O}^{\mathrm{e}}_{2,1}  \mathbf{T}_{1,2} \overline{\alpha} = - \overline{\alpha} + \mathbf{T}_{1,1} \overline{\alpha}.   \]
   The complex conjugate statements hold, as do the statements with the roles of $1$ and $2$ interchanged.
  \end{proposition}
  \begin{proof}
   By Lemma \ref{le:dbar_representative} there is an $h \in \mathcal{D}_{\mathrm{harm}}(\riem_1)$ such that $\overline{\partial} h = \overline{\alpha}$.   Differentiating both sides of the first expression appearing in Theorem \ref{th:Overfare_with_correction_functions} part (1)
   %or \ref{prop:transmitted_jump_with_corrections_caps} as appropriate 
   proves the first claim.  By Theorem \ref{th:S_kernel_range}, $\mathbf{S}_1 \overline{\alpha} = 0$ when $\overline{\alpha} \in (\overline{\mathbf{R}}_1 \overline{\mathcal{A}(\mathscr{R})})^\perp$, which completes the proof. 
  \end{proof}
 \end{subsection}
 \begin{subsection}{Kernel and image of the Schiffer operator $\mathbf{T}_{1,2}$} \label{se:Schiffer_kernel_image}

 In this section, we determine the kernel and image of the Schiffer operator $\mathbf{T}_{1,2}$ in various configurations.
 
 We require a generalization of Theorem \ref{th:Mohammad_isomorphism}. Namely, we would like to characterize the kernel and image of $\mathbf{T}_{1,2}$ in general. We begin with a partial characterization.
 
 \begin{theorem} \label{th:general_T12_kernel_on_perp}  Assume that $\riem_2$ is connected.  Then 
  \[ \mathrm{Ker} ( \mathbf{T}_{1,2}) \cap [\overline{\mathbf{R}}_1 \overline{\mathcal{A}(\mathscr{R})}]^\perp = \{0 \}.   \] 
 \end{theorem}
 \begin{proof}
  Assume that $\overline{\alpha} \in \mathrm{Ker} ( \mathbf{T}_{1,2}) \cap [\overline{\mathbf{R}}_1 \overline{\mathcal{A}(\mathscr{R})}]^\perp$. By Lemma \ref{le:dbar_representative} there is a $H \in \mathcal{D}_{\text{harm}}(\riem_1)$ such that $\overline{\partial} H = \overline{\alpha}$. We have $\overline{\mathbf{S}}_1 \overline{\partial} H=0$ by Theorem \ref{th:S_kernel_range} so by Theorem \ref{th:jump_derivatives}
  \[  d \mathbf{J}^q_{1,2} H = \mathbf{T}_{1,2} \overline{\alpha} + \overline{\mathbf{R}}_2 \overline{\mathbf{S}}_1 \overline{\alpha} = 0. \]
  Therefore $\mathbf{J}^q_{1,2} H$ is constant, from which it follows that $\mathbf{O}_{2,1} \mathbf{J}^q_{1,2} H$ is constant. By Theorem \ref{th:Overfare_with_correction_functions} we have
  \[  d(H - \mathbf{J}^q_{1,1} H) = - d \mathbf{O}_{2,1} \mathbf{J}^q_{1,2} H = 0   \]
  so again using Theorem \ref{th:jump_derivatives} and the fact that $\overline{\mathbf{S}}_1 \overline{\partial} H=0$ we obtain 
  \[ \partial H + \overline{\partial} H - \mathbf{T}_{1,1} \overline{\partial} H = 0,    \]
  so equating holomorphic and anti-holomorphic parts
  \[  \overline{\partial} H =0.  \]
  This completes the proof.
 \end{proof}
 
 We also have the following.
 \begin{theorem} \label{th:general_T12_on_perp_surjective_to_exact} 
 Assume that $\riem_2$ is connected.  
 The image of $[\overline{\mathbf{R}}_1 \overline{\mathcal{A}(\mathscr{R})}]^\perp$ under $\mathbf{T}_{1,2}$ is $\mathcal{A}^{\mathrm{e}}(\riem_2)$. 
 \end{theorem}
 \begin{proof}
  Given any $\beta \in \mathcal{A}^\mathrm{e}(\riem_2)$, let $h \in \mathcal{D}(\riem_2)$ be such that $\partial h = \beta$, which exists by conjugating Lemma \ref{le:dbar_representative}.  Note that $h$ is not necessarily uniquely defined. Set $H=-\mathbf{O}_{2,1}h$. Applying Theorems \ref{th:J_same_both_sides} and \ref{th:jump_on_holomorphic} we obtain that
  \[  \dot{\mathbf{J}}_{1,2} \dot{H} =  \dot{\mathbf{J}}_{2,2} h = \dot{h}.  \]
   Differentiating using Theorem \ref{th:jump_derivatives} we see that
  \[  \mathbf{T}_{1,2} \overline{\partial} H + \overline{\mathbf{R}}_2 \overline{\mathbf{S}}_1 \overline{\partial} H = \beta.  \]
  Since $\beta$ is holomorphic $\overline{\mathbf{R}}_2 \overline{\mathbf{S}}_1 \overline{\partial} H = 0$ so $\overline{\mathbf{S}}_1 \overline{\partial} H = 0$ and hence $\overline{\partial} H \in [\overline{\mathbf{R}}_1 \overline{\mathcal{A}(\mathscr{R})}]^\perp$ by Theorem \ref{th:S_kernel_range}. Furthermore $\mathbf{T}_{1,2} \overline{\partial} H = \beta$ completing the proof.
 \end{proof}
 
 Theorems \ref{th:general_T12_on_perp_surjective_to_exact} and \ref{th:general_T12_kernel_on_perp} taken together generalize Theorem \ref{th:Mohammad_isomorphism} and \cite[Theorem 4.22]{Schippers_Staubach_Plemelj}. We will extend it still further below. For now, we observe the following corollary.
 Recall the projections $\mathbf{P}_k= \mathbf{P}_{\riem_k}$ and 
 $\overline{\mathbf{P}}_k = \overline{\mathbf{P}}_{\riem_k}$ defined by 
 \eqref{eq:hol_antihol_projections_Bergman}.
 \begin{corollary} \label{co:connected_get_iso}
  Let $\riem_2$ be connected.  Then the restriction of $\mathbf{T}_{1,2}$ to $[\overline{\mathbf{R}}_1 \overline{\mathcal{A}(\mathscr{R})}]^\perp$ is an isomorphism onto $\mathcal{A}^\mathrm{e}(\riem_2)$, with inverse $-\overline{\mathbf{P}}_1 \mathbf{O}^\mathrm{e}_{2,1}$. 
 \end{corollary}
 \begin{proof} 
  The restriction of $\mathbf{T}_{1,2}$ is surjective by Theorem \ref{th:general_T12_on_perp_surjective_to_exact}. Since $\riem_2$ is connected, any function $\omega$ with the bridge property must have the same constant value on each boundary of $\riem_1$, and hence must be constant. So the kernel is trivial.  
  
  Observe that since $\riem_2$ is connected $\mathbf{O}^\mathrm{e}_{2,1}$ is well-defined by the requirement that $d \mathbf{O}_{2,1} = \mathbf{O}^\mathrm{e}_{2,1}$. The fact that this is the inverse follows as in previous proofs. Let $H \in \mathcal{D}_{\text{harm}}(\riem_1)$ be such that $\overline{\partial} H = \overline{\alpha}$. By Theorems \ref{th:Overfare_with_correction_functions},  \ref{th:jump_derivatives}, and the fact that $\overline{\mathbf{S}}_1 \overline{\alpha}=0$ we see that 
  \[ \mathbf{O}^\mathrm{e}_{2,1} \overline{\alpha} = -\overline{\alpha} + \mathbf{T}_{1,1} \overline{\alpha}. \]
  The claim follows immediately.
 \end{proof}

 In order to determine the image of $\mathbf{T}_{1,2}$, we define certain natural subspaces of $\mathcal{A}(\riem_2)$.  Assume that $\riem_2$ is connected (but not necessarily $\riem_1$).  
 Let $c^k_1,\ldots,c^k_{m_k}$ be a fixed homology basis of simple closed curves for $\riem_k$ for $k=1,2$.  
 
 We then have a linear map 
 \begin{align}\label{capital Xi}
  \Xi_k: \mathcal{A}(\mathscr{R}) & \rightarrow \mathbb{C}^{m_k} \\\nonumber
  u & \mapsto \left( \int_{c^k_1} u,\ldots, \int_{c^k_{m_k}} u \right)
 \end{align}
 and the linear map 
 \begin{align}\label{capital Xibar}
  \overline{\Xi_k}: \overline{\mathcal{A}(\mathscr{R})} & \rightarrow \mathbb{C}^{m_k} \\\nonumber
  \overline{u} & \mapsto \left( \int_{c^k_1} \overline{u},\ldots, \int_{c^k_{m_k}} \overline{u} \right). 
 \end{align}
 We then define the subspaces 
 \begin{align}\label{capital Xk}
    X_k & = \text{Im}(\Xi_k) \subseteq \mathbb{C}^{m_k} \\
    \overline{X}_k & = \text{Im}(\overline{\Xi}_k) \subseteq \mathbb{C}^{m_k}.
 \end{align}
 Although $\Xi_k$ and $\overline{\Xi}_k$ depend on the choice of basis, $X_k$ and $\overline{X}_k$ do not.  
 It will sometimes be convenient to choose specific homology bases and an ordering.  Note that some curves in the homology base of $\mathcal{A}(\mathscr{R})$ may appear in the homology base of both $\riem_1$ and $\riem_2$, and that some curves in the homology base of $\mathcal{A}(\mathscr{R})$ might not appear in the homology base of either $\riem_1$ or $\riem_2$.

 We then define the following subspaces of $\mathcal{A}(\riem_2)$: 
 \begin{equation}\label{svensk a}
\gls{gota}({\riem}_k) = \left\{  \alpha \in \mathcal{A}({\riem}_k) : \left( \int_{c^k_1} \alpha ,\ldots,\int_{c^k_{m_k}} \alpha  \right) \in X_k  \right\}.    
 \end{equation}  
 and 
 \begin{equation}\label{svensk abar}
    \mathrm{\textgoth{A}}^-(\riem_k) = \left\{  \alpha \in \mathcal{A}(\riem_k) : \left( \int_{c^k_1} \alpha ,\ldots,\int_{c^{k}_{m_k}} \alpha  \right) \in \overline{X}_k         \right\}  
 \end{equation}     
 for $k=1,2$.  
 Note that it is most certainly {\bf not} true that $\mathrm{\textgoth{A}}^-(\riem_k) = \overline{\mathrm{\textgoth{A}}(\riem_k)}$.   
 
 The definition immediately implies that
 \begin{proposition}  For $k=1,2$, given $\alpha \in \mathcal{A}(\riem_k)$, it holds that
 $\alpha \in \emph{{\textgoth{A}}}(\riem_k)$ if and only if $\alpha$ is in the same cohomology class as an element of $\mathbf{R}_k\mathcal{A}(\mathscr{R})$.  Similarly, $\alpha \in \emph{{\textgoth{A}}}^-(\riem_k)$ if and only if $\alpha$ is in the same cohomology class as an element of $\overline{\mathbf{R}}_k \overline{\mathcal{A}(\mathscr{R})}$.  
 \end{proposition}
 
 Another useful fact is the following.
 \begin{proposition}  \label{pr:XplusXbar_is_all}  
  For $k=1,2$ 
  \[  X_k {+} \overline{X}_k = \mathbb{C}^{m_k}.  \]
 \end{proposition}
 \begin{proof}
   This is an immediate consequence of the Hodge theorem, which says that every cohomology class on $\mathscr{R}$ has a representative in $\mathcal{A}_{\text{harm}}(\mathscr{R})$.  Thus every possible configuration of periods of $c^k_1,\ldots,c^k_{m_k}$ in $\mathscr{R}$ (and so, in particular of of $c^k_1,\ldots,c^k_{m_k}$ in $\riem_k$)  can be attained by an element of  $\mathcal{A}_{\text{harm}}(\mathscr{R})$.  
 \end{proof}

 The following theorem establishes the behaviour of $\mathbf{T}_{1,2}$ on its entire domain.    

 \begin{theorem} 
  \label{th:Tonetwo_iso_improved}  Assume that $\riem_2$ is connected.  
  Let 
  \[ \overline{W}_1 = \left\{\overline{\mu} \in \overline{\mathbf{R}}_1 \overline{\mathcal{A}(\mathscr{R})};\,\, \overline{\mathbf{R}}_2 \overline{\mathbf{S}}_1 {\overline{\mu} } \in \overline{\mathcal{A}^{\mathrm{e}}(\riem_2)} \right\}  \]
    We have 
    \begin{equation*}
      \mathrm{Im}(\mathbf{T}_{1,2})  = \emph{{\textgoth{A}}}^-(\riem_2)
    \end{equation*}
    and
    \begin{equation*}
     \mathrm{Ker}(\mathbf{T}_{1,2}) \cong  \overline{W}_1.
    \end{equation*}
    The same claim follows for the complex conjugates and with $1$ and $2$ interchanged.
 \end{theorem}
 \begin{proof}
    It follows from Theorem \ref{th:Schiffer_cohomology} that $\text{Im}(\mathbf{T}_{1,2}) \subseteq \mathrm{\textgoth{A}}^-(\riem_2)$.  We show that $\mathrm{\textgoth{A}}^-(\riem_2)  \subseteq  \text{Im}(\mathbf{T}_{1,2})$. 
 Let $\beta \in \mathrm{\textgoth{A}}^-(\riem_2)$. Again applying Theorem  \ref{th:Schiffer_cohomology}, since $\overline{\mathbf{S}}_1 \overline{\mathbf{R}}_1$ is an isomorphism by Theorem \ref{th:S_an_isomorphism}, we can find $\overline{\gamma} \in \overline{\mathcal{A}(\mathscr{R})}$ such that 
   \[  \beta - \mathbf{T}_{1,2}   \overline{\mathbf{R}}_1 \overline{\gamma}  \in \mathcal{A}^e(\riem_2).  \]
 By Theorem \ref{th:general_T12_on_perp_surjective_to_exact} there is an $\overline{\alpha} \in (\overline{\mathbf{R}}_1 \overline{\mathcal{A}(R)})^\perp$ such that 
   \[  \mathbf{T}_{1,2} \overline{\alpha} = \beta - \mathbf{T}_{1,2}  \overline{\mathbf{R}}_1 \overline{\gamma} \]
   which completes the proof of the first claim.
   
   We now prove the second claim. 
   Let $\overline{\mu} \in \overline{W}_1$. Since $\mathbf{T}_{1,2} \overline{\mu} + \overline{\mathbf{R}}_2 \overline{\mathbf{S}}_1 \overline{\mu}$ is exact by Theorem \ref{th:Schiffer_cohomology}, so is $\mathbf{T}_{1,2} \mu$.  Thus by Theorem \ref{th:general_T12_on_perp_surjective_to_exact} there is an $\overline{\alpha} \in  (\overline{\mathbf{R}}_1 \overline{\mathcal{A}(R)})^\perp$ such that $\mathbf{T}_{1,2} \overline{\alpha} = - \mathbf{T}_{1,2} \overline{\mu}$ so $\overline{\alpha} + \overline{\mu} \in \mathrm{Ker} (\mathbf{T}_{1,2})$.  We define 
   \begin{align*}
      \Phi:\overline{W}_1 & \rightarrow \mathrm{Ker}(\mathbf{T}_{1,2}) \\
      \overline{\mu} & \mapsto \overline{\alpha} + \overline{\mu}.
   \end{align*}
   This is well-defined, since if $\overline{\alpha} + \overline{\mu}$ and $\overline{\beta} + \overline{\mu}$ are both in $\mathrm{Ker}( \mathbf{T}_{12})$ for $\overline{\beta},\overline{\alpha} \in [\overline{\mathbf{R}}_1 \overline{\mathcal{A}(\mathscr{R})}]^\perp$ then $\overline{\alpha} - \overline{\beta} \in [\overline{\mathbf{R}}_1 \overline{\mathcal{A}(\mathscr{R})}]^\perp \cap \mathrm{Ker}( \mathbf{T}_{1,2})$ so $\overline{\alpha} -\overline{\beta} =0$ by Theorem \ref{th:general_T12_kernel_on_perp}.  
   
   This map is surjective. Assume that $\overline{\gamma}  \in \mathrm{Ker}(\mathbf{T}_{1,2})$.  Write $\overline{\gamma} = \overline{\alpha} + \overline{\mu}$ for $\overline{\alpha} \in  (\overline{\mathbf{R}}_1 \overline{\mathcal{A}(R)})^\perp$ and $\overline{\mu} \in \overline{\mathbf{R}}_1 \overline{\mathcal{A}(\mathscr{R})}$. Since $\mathbf{T}_{1,2} \overline{\alpha} = -\mathbf{T}_{1,2} \overline{\mu}$ and the former is exact by Theorem \ref{th:general_T12_on_perp_surjective_to_exact} we see that $\mathbf{T}_{1,2} \overline{\mu}$ is exact. Thus by Theorem \ref{th:Schiffer_cohomology} $\overline{\mathbf{R}}_2 \overline{\mathbf{S}}_1 \overline{\mu}$ is exact.  So $\overline{\mu} \in \overline{W}_1$ and $\Phi(\overline{\gamma}) = \overline{\gamma}$. 
   
   This map is injective.  Assume that $\Phi(\overline{\mu})=0$. Then $\overline{\alpha} + \overline{\mu} =0$.  Using Theorem \ref{th:general_T12_kernel_on_perp} and the fact that $\overline{\alpha} \in [\overline{\mathbf{R}}_1 \overline{\mathcal{A}(\mathscr{R})}]^\perp$ we obtain that $\overline{\mu} \in [\overline{\mathbf{R}}_1 \overline{\mathcal{A}(\mathscr{R})}]^\perp$.  So $\overline{\mu}=0$. 
 \end{proof}
 \begin{remark} \label{re:explicit_Tpartialinverse_construction}
  The element $\overline{\alpha}$ corresponding to $\overline{\mu}$ in the definition of $\Phi$ can be constructed explicitly as follows.  Given $\overline{\mu} \in \overline{W}_1$, since $\mathbf{T}_{1,2} \overline{\mu}$ is exact, there is an $h \in \mathcal{D}(\riem_2)$ such that $\partial h = \mathbf{T}_{1,2} \overline{\mu}$. Set $H=  \mathbf{O}_{2,1} h$.   So 
  \[  \mathbf{J}_{1,2}^q H = -\mathbf{J}^q_{2,2} h = -h + c \]
  for some $c$ which is constant on connected components, by Theorems \ref{th:J_same_both_sides} and \ref{th:jump_on_holomorphic}.  So applying Theorem \ref{th:jump_derivatives} we obtain 
  \[  \mathbf{T}_{1,2} \overline{\partial} H + \overline{\mathbf{R}}_2  \overline{\mathbf{S}}_1 \overline{\partial} H =  d \mathbf{J}_{1,2} H = - \partial h = -\mathbf{T}_{1,2} \overline{\mu}. \]
  But since the right hand side is holomorphic, we must have that $\overline{\mathbf{R}}_2  \overline{\mathbf{S}}_1 \overline{\partial} H =0$ so by analytic continuation and Theorem \ref{th:S_kernel_range} we see that $\overline{\alpha} =\overline{\partial} H \in [\overline{\mathbf{R}}_1 \overline{\mathcal{A}(\mathscr{R})}]^\perp$ and
  \[  \mathbf{T}_{1,2} (\overline{\alpha} + \overline{\mu} ) =0. \]
  We may summarize this by saying that 
  \[  \Phi =  \mathbf{I} - \overline{\mathbf{P}}_1 d \mathbf{O}_{2,1} d^{-1} \mathbf{T}_{1,2}     \]
  observing that this is well-defined by the proof of the theorem.
 \end{remark}

   This has the following important consequence.
   \begin{corollary}  \label{co:T_plus_extra_surjective_capped}  Let $\riem_2$ be capped by $\riem_1$.  Then $\mathrm{ker}(\mathbf{T}_{1,2})$ is trivial.
   Furthermore, 
    any $\alpha \in \mathcal{A}(\riem_2)$ can be written 
    \[  \alpha = \mathbf{T}_{1,2} \overline{\gamma} + \mathbf{R}_2  \tau + \partial \omega  \]
    for unique $\overline{\gamma} \in \overline{\mathcal{A}(\riem_1)}$, $\tau \in \mathcal{A}(\mathscr{R})$, and $d\omega \in \mathcal{A}_{\mathrm{hm}}(\riem_2)$.     That is 
    \[    \mathcal{A}(\riem_2) = \emph{\textrm{\textgoth{A}}}^{-}(\riem_2) \oplus \mathbf{R}_2 \mathcal{A}(\mathscr{R}) \oplus \partial \mathcal{D}_{\mathrm{hm}}(\riem_2).       \]
    The same claim holds for complex conjugates.
   \end{corollary}
   \begin{proof}  
    Since any exact form on $\mathscr{R}$ is zero, we see that $\overline{W}_1 = \{0 \}$, so the kernel is zero by Theorem \ref{th:Tonetwo_iso_improved}.  
    Since the periods of elements of $\mathbf{R}_2 \mathscr{A}_{\mathrm{harm}}(\mathscr{R})$ are zero around the boundary curves $\partial_k \riem_2$ for all $k$, we see that $\ast \mathcal{A}_{\mathrm{hm}}(\riem_2)$ and $\mathbf{R}_2 \mathscr{A}_{\mathrm{harm}}(\mathscr{R})$ are linearly independent. Using the decomposition 
    \[  \partial \omega = \frac{1}{2}\left( d\omega + i \ast d\omega \right) \]
    shows that $\partial \mathcal{D}_{\mathrm{harm}}(\riem_2)$ and $\mathbf{R}_2 \mathcal{A}_{\mathrm{harm}}(\mathscr{R})$ are linearly independent. 
    
    Now $X_1$ and $\overline{X}_1$ are linearly independent, since each has dimension $g$, and the dimension of $X_1 + \overline{X}_1$ is $2g$.  Thus since $\text{Im}(\mathbf{T}_{1,2}) = {\mathrm{\textgoth{A}}}^{-}(\riem_2)$ by Theorem \ref{th:Tonetwo_iso_improved}, this proves that 
    \[  \mathcal{A}(\riem_2) = {\textrm{\textgoth{A}}}^{-}(\riem_2)  + \mathbf{R}_2 \mathcal{A}(\mathscr{R}) + \partial \mathcal{D}_{\mathrm{hm}}(\riem_2).  \]
    decomposition. Linear independence proves that the decomposition is a direct sum, and uniqueness of $\tau$. The uniqueness of $\overline{\partial} \omega$ follows from Theorem \ref{th:period_matrix_invertible}, and uniqueness of $\overline{\gamma}$ follows from triviality of the kernel of $\mathbf{T}_{1,2}$.  
   \end{proof}
   
\end{subsection}
\begin{subsection}{Fredholm index of the Schiffer operator} \label{th:index_and_examples}
 In this section, we derive the Fredholm index of the Schiffer operator $\mathbf{T}_{1,2}$ for the case where $\riem_1$ and $\riem_2$ are connected, and for the case that $\riem_2$ is capped by $\riem_1$. 
 We also determine general formulas for the dimensions of the kernel and cokernel which could be used to derive index theorems in other configurations.
 
 In the following, we first observe that 
  \begin{equation} \label{eq:kernel_Xi_and_W}
   \mathrm{dim} \, \mathrm{Ker}(\Xi_1) = \mathrm{dim}\, W_2. 
  \end{equation} 
  This follows directly from the definitions together with the fact that $\mathbf{R}_1 \mathbf{S}_1$ is an isomorphism by Theorem \ref{th:S_kernel_range}.  
 The same claim holds with $1$ and $2$ interchanged, as does the complex conjugate. 
 
 \begin{theorem}  \label{th:X_k_analysis}
  Assume that $\riem_1$ and $\riem_2$ are connected. Let $g$ be the genus of $\mathscr{R}$ and $g_1$,$g_2$ be the genuses of $\riem_1$ and $\riem_2$ respectively. 
  \begin{equation}  \label{eq:dimX1_and_W_2}
   \mathrm{dim}\, X_1   = g - \mathrm{dim} W_2 
  \end{equation}
  and
  \begin{equation}  \label{eq:dimX1_and_X1intersectX1bar}
   \mathrm{dim}\, X_1  = g_1 + \frac{n-1}{2} + \frac{1}{2} \, \mathrm{dim}(X_1 \cap \overline{X}_1). 
  \end{equation}
  The same claims hold with $1$ and $2$ interchanged. 
 \end{theorem}
 \begin{proof}
   The first claim follow from the fact that 
   \[  X_1 = \mathrm{Im} (\Xi_1), \ \ \  \]
   equation \eqref{eq:kernel_Xi_and_W}, and the fact that $\mathcal{A}(\mathscr{R})$ has dimension $g$.   
   
   To prove the second claim, first observe that every homology class in $\riem_1$ is represented by a homology class in $\mathscr{R}$ (note that this depends on the assumptions on the configuration $\riem_1$, $\riem_2$, $\mathscr{R}$).  So 
   \[ \mathrm{dim} (X_1+ \overline{X}_1) = 2g_1 + n-1.  \] 
   Using this together with the fact that 
   \[  \mathrm{dim}\, (X_1+\overline{X}_1) = \mathrm{dim}\, X_1 + \mathrm{dim}\,\overline{X}_1 - \mathrm{dim}\, (X_1 \cap \overline{X}_1) = 2 \, \mathrm{dim}\, X_1 - \mathrm{dim}\, (X_1 \cap \overline{X}_1)  \]
   proves the claim.
 \end{proof}
 Combining these two claims, together with properties of the Schiffer operator, results in the following. 
 \begin{theorem}  \label{th:Schiffer_index_theorem}
  Assume that $\riem_1$ and $\riem_2$ are connected. Let $g$ be the genus of $\mathscr{R}$ and $g_1$,$g_2$ be the genuses of $\riem_1$ and $\riem_2$ respectively. Then 
  \[ \mathrm{Index} (\mathbf{T}_{1,2}) = g_1 - g_2.  \]
  The same claim holds with $1$ and $2$ switched.  
 \end{theorem}
 \begin{proof}
  Combining the two equations in Theorem \ref{th:X_k_analysis}, we obtain \begin{align}  \label{eq:index_temp_zero}
      g- \mathrm{dim}\, W_2 & = g_1 + \frac{n-1}{2} + \frac{1}{2} \mathrm{dim}\,(X_1 \cap \overline{X}_1 ) \nonumber \\ 
      g- \mathrm{dim}\, W_1 & = g_2 + \frac{n-1}{2} + \frac{1}{2} \mathrm{dim}\,(X_2 \cap \overline{X}_2 ) 
  \end{align}
  so since $g_1 + g_2 + n-1 = g$ we obtain 
  \begin{align} \label{eq:index_proof_temp} 
     \mathrm{dim}\, W_2 + \frac{1}{2} \mathrm{dim}\, (X_1\cap \overline{X}_1) & = g_2 + \frac{n-1}{2} \nonumber \\ 
      \mathrm{dim}\, W_1 + \frac{1}{2} \mathrm{dim}\, (X_2\cap \overline{X}_2) & = g_1 + \frac{n-1}{2}
  \end{align}
  
  Next we compute the dimension of the cokernel of $\mathbf{T}_{1,2}$. By Theorem \ref{th:Tonetwo_iso_improved}, we have that all harmonic forms with periods in $\overline{X}_2$ are in the image of $\mathbf{T}_{1,2}$, from which we conclude that
  \begin{align*}
   \mathrm{dim}\, \mathrm{Coker}( \mathbf{T}_{1,2}) & = 2g_2 + n-1 -
   \mathrm{dim}\, X_2  \\
   & = g_2 + \frac{n-1}{2} - \frac{1}{2} \mathrm{dim}\,(X_2 \cap \overline{X}_2)
  \end{align*}
  where we have used equation \eqref{eq:dimX1_and_X1intersectX1bar} with $1$ replaced by $2$.  However, since
  $\mathbf{T}_{1,2}^* = \overline{\mathbf{T}}_{2,1}$, by Theorem \ref{th:adjoint_identities} we have 
  \[  \mathrm{dim}\, \mathrm{Coker}( \mathbf{T}_{1,2}) = \mathrm{dim}\, \mathrm{Ker}\, \mathbf{T}_{2,1} = \mathrm{dim}\, W_2 \]
  where we have used Theorem \ref{th:Tonetwo_iso_improved}. Thus 
  \begin{equation*}
   \mathrm{dim}\, W_2 = g_2 + \frac{n-1}{2} - \frac{1}{2} \mathrm{dim}\,(X_2 \cap \overline{X}_2)
  \end{equation*} 
  which upon comparison with \eqref{eq:index_proof_temp} yields that
  \begin{equation} \label{eq:intersections_X12_equal}
     \frac{1}{2} \, \mathrm{dim}\,(X_1 \cap \overline{X}_1) =  
  \frac{1}{2} \, \mathrm{dim}\,(X_2 \cap \overline{X}_2).   
  \end{equation} 
  Now using this fact together with equation \eqref{eq:index_proof_temp} we obtain  
  \begin{align*}
     \mathrm{dim}\, \mathrm{Ker}( \mathbf{T}_{1,2}) - \mathrm{dim}\, \mathrm{Coker}( \mathbf{T}_{1,2}) & = \mathrm{dim}\, \mathrm{Ker}( \mathbf{T}_{1,2}) - \mathrm{dim}\, \mathrm{Ker}( \overline{\mathbf{T}}_{2,1}) \\
     & =\mathrm{dim}\, W_1 - \mathrm{dim}\, W_2 \\
     & = g_1 - g_2 
  \end{align*}
  as claimed. 
  
  To prove the final claim, just switch the roles of $1$ and $2$ in the proof, which can be done by the symmetry of the conditions.
 \end{proof}
 
 \begin{remark}  Under the same assumptions, the proof also shows the following interesting facts. 
  By equations \eqref{eq:index_temp_zero} and \eqref{eq:index_proof_temp}, together with the fact that
  $\mathrm{dim}\, (X_1 \cap \overline{X}_1) = \mathrm{dim}\, (X_2 \cap \overline{X}_2)$ by \eqref{eq:intersections_X12_equal}, we obtain  
  \[ g_1 -g_2 = \mathrm{dim}\, W_1 - \mathrm{dim}\,W_2 = \mathrm{dim}\, X_1 - \mathrm{dim}\, X_2.  \] 
  Furthermore, \eqref{eq:index_proof_temp} implies that 
  \[  \mathrm{dim}\, (X_k \cap X_k) = n - 1 \ \  \mathrm{mod}\, 2  \]
  for $k=1,2$.  
 \end{remark}
 
 We also have the following.
 \begin{theorem} \label{th:index_capped_surface}
  Let $\riem_2$ be a surface of genus $g$ capped by $\riem_1$, where $\riem_1$ has $n$ connected components.  Then 
  \[  \mathrm{Index} ( \mathbf{T}_{1,2}) = - \mathrm{Index} ( \mathbf{T}_{2,1}) = 1-n-g.  \]
 \end{theorem}
 \begin{proof}
  The fact that the index of $\mathbf{T}_{1,2}$ is $1-n-g$ follows directly from Corollary \ref{co:T_plus_extra_surjective_capped} together with the facts that the dimension of $\partial \mathcal{D}_{\mathrm{harm}}(\riem_2)$ is $n-1$ and the dimension of $\mathcal{A}(\mathscr{R})$ is $g$. Using \ref{th:adjoint_identities} we have 
  \[ \mathrm{Index} ( \mathbf{T}_{2,1}) = - \mathrm{Index} ( \mathbf{T}_{2,1}^* )=  - \mathrm{Index} ( \overline{\mathbf{T}}_{1,2}) = - \mathrm{Index} ( {\mathbf{T}}_{1,2})\]
  which completes the proof.
 \end{proof}
 
 {The index theorems above connect conformally invariant quantities (the index of $\mathbf{T}_{1,2}$) to topologically invariant quantities. 
 The Schiffer operators are conformally invariant, as we saw in \eqref{eq:Schiffer_operators_conformally_invariant}. Thus their spectra, kernels, images, and indices are all conformally invariant. 
 Because the spaces $W_k$ are conformally but not obviously topologically invariant, it is interesting that they cancel in the proof of Theorem \ref{th:Schiffer_index_theorem}, and only topological data remains. The question then arises: are the dimensions of the cokernel and kernel themselves topological invariants? In other words, is it possible to choose topologically equivalent configurations with distinct dimensions for the cokernel and kernel of $\mathbf{T}_{1,2}$?  Either answer would be of great interest. 
 
 In fact, the kernels and cokernels are related to the image of the period map of $\mathscr{R}$, restricted to homology curves in $\riem_1$ or $\riem_2$.  The following example, the case of a genus two torus sliced by one curve, illustrates this. We also explicitly compute $W_1$ (defined implicitly in Theorem \ref{th:Tonetwo_iso_improved}) for this example, which turns out to be trivial. Although it therefore does not provide a counterexample to the topological invariance of the cokernel and kernel, the approach might however be a promising way to seek one. We leave this as an open problem.
 }

{Returning to the problem of computation of $W_1$, following \cite{Roydenperiods}, we start by recalling some basic facts regarding periods and related matrices. Let the compact Riemann surface \(\mathscr{R}\) have a canonical homology basis \(\left\{A_{j}, B_{j}\right\},\) where the \(A_{j}\) and \(B_{j}\)
are smooth simple closed curves with intersection numbers given
by
$$
\begin{array}{l}{\left[A_{j} \times B_{k}\right]=\delta_{j k}} \\ {\left[A_{j} \times A_{k}\right]=0} \\ {\left[B_{j} \times B_{k}\right]=0}\end{array}.
$$
When we are not interested in the intersection properties, we set $C_{j+g}=B_{j},\, 1 \leq j \leq g.$

If $\alpha$ and $\alpha^{\prime}$ are two closed forms with periods $a_{j}, b_{j}$ and $a_{j}^{\prime}, b_{j}^{\prime},$ respectively, around $\left\{A_{j}, B_{j}\right\}$, then  \emph{Riemann's bilinear relations} (or Riemann's period relations) state that
\begin{equation}\label{Riemann relations}
    \int_{\mathscr{R}} \alpha \wedge \alpha^{\prime}=\sum_{j}\left(a_{j} b_{j}^{\prime}-a_{j}^{\prime} b_{j}\right).
\end{equation}

Now let \(\omega_{j}\) be the harmonic one-form on \(\mathscr{R}\) whose period
around \(C_{k}\) is \(\delta_{j k} \) and \(w_{j}\) be the holomorphic one-form on \(\mathscr{R}\)  whose periods around \(A_{k}\) are \(\delta_{j k} .\) Then the entries of the so-called \emph{Riemann matrix} \(\Pi=\left[\pi_{j k}\right]\)
are given by
$$
\pi_{j k}=\int_{B_{k}} w_{j}.
$$
Since \(w_{j} \wedge w_{k}=0,\) 
\eqref{Riemann relations} yields that
\begin{equation}\label{symmetry of Riemann}
    0=\int w_{j} \wedge w_{k}=\pi_{k j}-\pi_{j k}.
\end{equation}

Thus \(\Pi\) is a symmetric matrix.

Note also that, since period of \(w_{j}\) is \(\delta_{j k}\) around \(A_{k}\) and is \(\pi_{j k}\) around \(B_{k},\)
one has
\begin{equation}\label{Roydens trick}
  w_{j}=\omega_{j}+\sum_{k} \pi_{j k} \omega_{k+g}.  
\end{equation}
\begin{figure}
     \includegraphics[width=9cm]{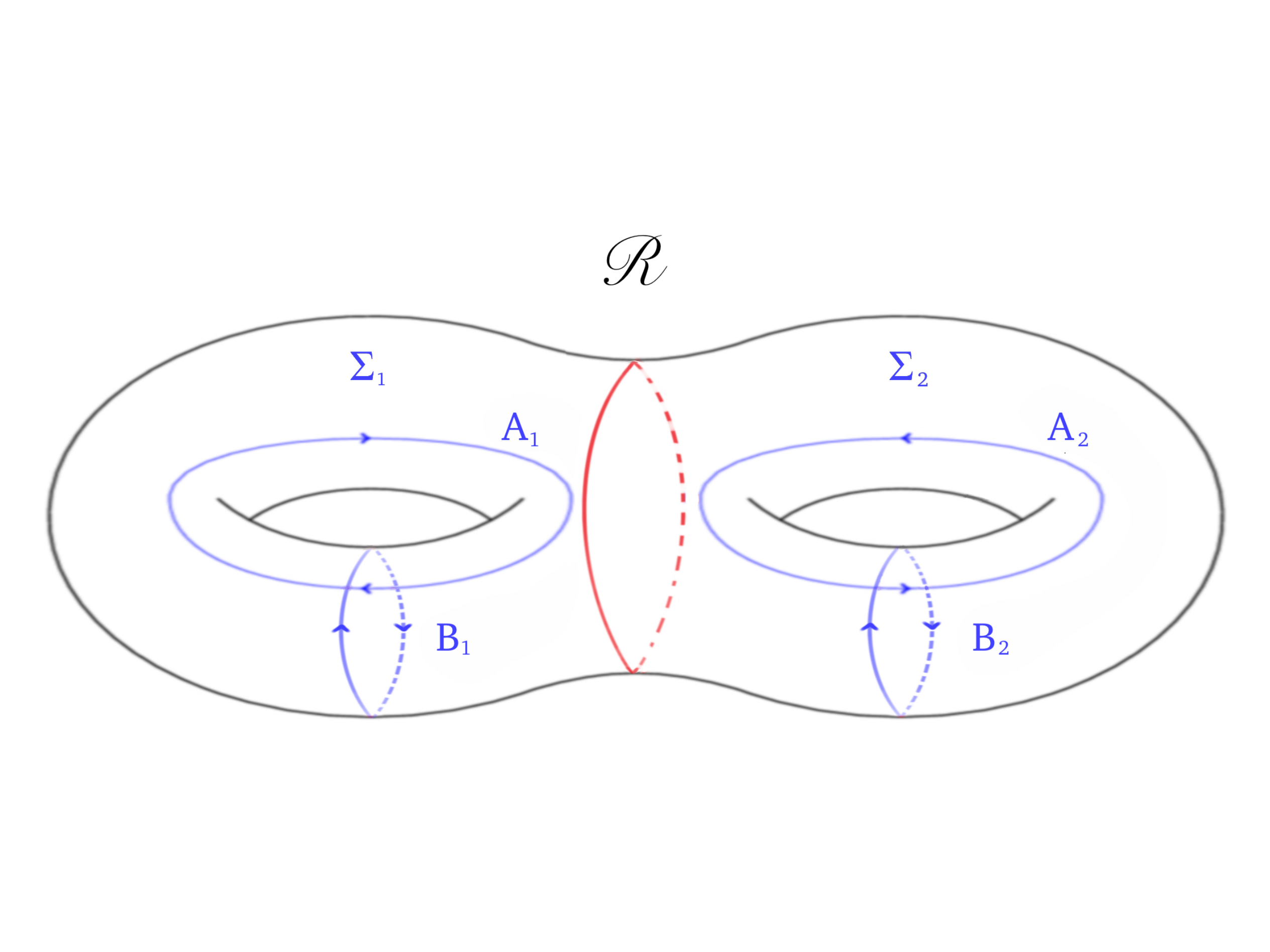}
     \caption{Genus two Riemann surface and its homology basis}
     \label{fig:period_curves}
 \end{figure} 
Now let $\mathscr {R}$ be the Riemann surface depicted in Figure \ref{fig:period_curves} and let us apply the information above to the problem of characterization of the set 
\begin{equation}\label{charac Vj}
    V_2:=\{w\in{\mathcal{A}(\mathscr{R})};\, \mathbf{R}_{2}{w} \in {\mathcal{A}^{\mathrm{e}}(\Sigma_2)}\}.
\end{equation}

We note that in the computation of $W_1$ mentioned above, we can confine ourselves to the computation of $V_2$. Here we have that \(w_{1}=\omega_{1}+\sum_{k=1}^{2}\pi_{1 k} \omega_{k+2}\) and
\(w_{2}=\omega_{2}+\sum_{k=1}^{2} \pi_{2 k} \omega_{k+2}\) and we seek a holomorphic one-form given by \(N_{1} w_{1}+N_{2} w_{2},\) with $N_j \in \mathbb{C},$ which is in \(V_{2}\). Therefore
$$\int_{A_{2}} (N_{1} w_{1}+N_{2} w_{2})=0 .$$ However, \eqref{Roydens trick} yields that $$0=\int_{A_{2}} N_{1} w_{1}+N_{2} w_{2}=\int_{A_{2}} \Big(N_{1} \omega_{1}+N_{1} \sum_{k=1}^{2} \pi_{1 k} \omega_{k+2}+N_{2} \omega_{2}+N_{2} \sum_{k} \pi_{2 k} \omega_{k+2}\Big)
=N_{2}.$$
Furthermore
\(\int_{B_{2}} N_{1} w_{1}=\int_{B_{2}} (N_{1} \omega_{1}+N_{1} \sum_{k=1} ^{2}\pi_{1 k} \omega_{k+2})=N_{1} \pi_{12}\).
Hence \(V_{2}\) is non-empty if and only if \(\pi_{12}=0\). Therefore by \eqref{symmetry of Riemann}, for $\mathscr{R}$ as in Figure \ref{fig:period_curves}, $V_2$ is non-empty if and only if the Riemann matrix has the form
\begin{equation}
  \begin{pmatrix}
\pi_{11} & 0\\
0 & \pi_{22}
\end{pmatrix}
 \end{equation}
 However, as was shown by M. Gerstenhaber in \cite{Gerstenhaber}, no  surface of genus 2 has  a diagonal matrix for a Riemann matrix, and therefore $V_2$ is indeed empty.

} 
\end{subsection}
\end{section}

\printnoidxglossary[sort=def]


\begin{thebibliography}{99}

\bibitem{Ahlfors_Sario} Ahlfors, L. V.; Sario, L.
 { Riemann surfaces}.
   { Princeton Mathematical Series}, No. 26 Princeton University Press, Princeton, N.J. 1960.
   

 
  
\bibitem{AskBar}  Askaripour, N.; and Foth, T. On holomorphic $k$-differentials on open Riemann surfaces. Complex Var. Elliptic Equ. {\bf 57} (2012), no. 10, 1109–-1119.
   
\bibitem{AskBar2} Askaripour, N.; and Barron, T.  On extension of holomorphic $k$-differentials on open Riemann surfaces. Houston J. Math. {\bf 40} (2014), no. 4, 117--1126.
 


  \bibitem{BergmanSchiffer}  Bergman, S.; and Schiffer, M. 
 { Kernel functions and conformal mapping.} Compositio Math. {\bf 8}, (1951), 205--249.

\bibitem{ConwayII}  Conway, J. B. { Functions of one complex variable }II. Graduate Texts in Mathematics, 159. Springer-Verlag, New York, 1995.


\bibitem{Courant_Schiffer}  Courant, R.  {Dirichlet's principle, conformal mapping, and minimal surfaces.} With an appendix by M. Schiffer. Reprint of the 1950 original. Springer-Verlag, New York-Heidelberg, 1977.

\bibitem{Eynard_notes} Eynard, B. Lecture notes on Riemann surfaces.  arXiv:1805.06405.


\bibitem{Farkas_Kra}  Farkas, H. M.; and Kra, I. Riemann surfaces. Second edition. Graduate Texts in Mathematics, 71. Springer-Verlag, New York, 1992.


\bibitem{Gerstenhaber} Gerstenhaber, M.  On  a  theorem on  Haupt and Wirtinger concerning the  periods of  adifferential of the first  kind, anda related topological theorem, Proc. Amer. Math. Soc. {\bf 4} (1953), 476--481.

\bibitem{Lehto}
Lehto, O. {Univalent functions and {T}eichm\"uller spaces}. Graduate Texts
  in Mathematics, Vol. 109, Springer-Verlag, New York, 1987.
  

\bibitem{Nap_Yulm} Napalkov, V. V., Jr.; Yulmukhametov, R. S. On the Hilbert transform in the Bergman space. (Russian) Mat. Zametki {\bf 70} (2001), no. 1, 68--78; translation in Math. Notes {\bf 70} (2001), no. 1-2, 61--70.

   
   \bibitem{Royden}  Royden, H. L. {Function theory on compact Riemann surfaces}. J. Analyse Math. {\bf 18}, (1967), 295--327. 
   
   
\bibitem{Roydenperiods} Royden, H. L. The Variation of Harmonic Differentials and their Periods. Complex analysis, 211–223, Birkh\"auser, Basel, 1988. 


\bibitem{Schiffer_first} Schiffer, M. {The kernel function of an orthonormal system}. Duke Math. J. {\bf 13}, (1946). 529 -- 540.
\bibitem{Schiffer_Spencer} Schiffer, M. and Spencer, D.  {Functionals on finite Riemann surfaces}. Princeton University Press, Princeton, N. J., 1954. 

\bibitem{Schippers_Shirazi_Faber} Schippers, E.; and Shirazi, M. Faber series for $L^2$ holomorphic one-forms on Riemann surfaces with boundary.  arXiv:2303.15677v1. To appear in Comp. Methods and Func. Theor. 

\bibitem{Schippers_Shirazi_Staubach} Schippers, E., Shirazi, M. and Staubach, W.  Schiffer comparison operators and approximations on Riemann surfaces bordered by quasicircles,  J. Geom. Anal. {\bf 31} (2021), no. 6, 5877–-5908.

\bibitem{Schippers_Staubach_Plemelj} Schippers, E.; Staubach, W. Plemelj-Sokhotski isomorphism for quasicircles in Riemann surfaces and the Schiffer operator. Math. Ann. {\bf 378} (2020), no. 3-4, 1613-–1653.



\bibitem{Schippers_Staubach_Grunsky_expository}  Schippers, E.; and Staubach, W.  Analysis on quasicircles-A unified approach through transmission and jump problems. EMS Surv. Math. Sci. {\bf 9} (2022), no. 1, pp. 31--97.

\bibitem{Schippers_Staubach_scattering_arxiv} Schippers, E; and Staubach, W.  A scattering theory of harmonic one-forms on Riemann surfaces.  arXiv:2112.00835v1

\bibitem{Schippers_Staubach_scattering_I} Schippers, E.; and Staubach, W. Overfare of harmonic functions on Riemann surfaces.   
New York J. Math. {\bf 31} (2025) 321--367.


\bibitem{Schippers_Staubach_scattering_II} Schippers, E.; and Staubach, W. Overfare of harmonic one-forms on Riemann surfaces.   
New York J. Math. {\bf 30} (2024) 1437--1478.

\bibitem{Schippers_Staubach_scattering_IV} Schippers, E.; and Staubach, W. Scattering theory on Riemann surfaces II: the scattering matrix and generalized period mappings. Accepted for publication in Communications in Contemporary Mathematics.

\bibitem{Schippers_Staubach_Carlos_paper} Schippers, E; and Staubach, W. A Survey of Scattering Theory on Riemann Surfaces with Applications in Global Analysis and Geometry. Special issue of Vietnam Journal of Mathematics dedicated to Carlos E. Kenig's 70th birthday, {\bf 51} 4 (2023), 911--934.

\bibitem{Shirazi_thesis}Shirazi, M. Faber and Grunsky Operators on Bordered Riemann Surfaces of Arbitrary Genus and the Schiffer Isomorphism". PhD Dissertation, University of Manitoba 2020. 
\bibitem{Shirazi_Grunsky} Shirazi, M. Faber and Grunsky Operators Corresponding to Bordered Riemann Surfaces, Conform. Geom. Dyn. {\bf 24} (2020), 177--201.

\end{thebibliography}
\end{document}